\documentclass[11pt]{amsart}
\usepackage{latexsym,amscd,amssymb, graphicx, color, amsthm, bm, amsmath, cancel, enumitem, ytableau}
\allowdisplaybreaks
\usepackage{tikz}
\usepackage[margin=1in]{geometry}
\usepackage{hyperref}
\usepackage[all,cmtip]{xy}
\usepackage{mathtools}
\usepackage{rotating}
\usetikzlibrary{math, arrows.meta}
\usetikzlibrary{calc}
\usepackage[backend=bibtex]{biblatex}
\numberwithin{equation}{section}

\newtheorem{theorem}{Theorem}[section]
\newtheorem{proposition}[theorem]{Proposition}
\newtheorem{corollary}[theorem]{Corollary}
\newtheorem{lemma}[theorem]{Lemma}

\newtheorem{problem}[theorem]{Problem}

\newtheorem{remark}[theorem]{Remark}
\newtheorem{definition}[theorem]{Definition}

\theoremstyle{definition}

\newcommand{\Hilb}{{\mathrm{Hilb}}}

\newcommand{\rev}{{\mathrm{rev}}}

\newcommand{\gl}{\mathrm{GL}}
\newcommand{\Frob}{{\mathrm{Frob}}}
\newcommand{\symm}{{\mathfrak{S}}}

\newcommand{\grFrob}{{\mathrm{grFrob}}}

\newcommand{\CC}{{\mathbb{C}}}
\newcommand{\QQ}{{\mathbb{Q}}}
\newcommand{\ZZ}{{\mathbb{Z}}}
\newcommand{\PP}{{\mathbb{P}}}

\newcommand{\spa}{\mathrm{span}}

\newcommand{\xxx}{{\mathbf{x}}}

\newcommand{\mmm}{{\mathfrak{m}}}

\newcommand{\dschub}{\widetilde{S}}

\newcommand{\ind}{\mathrm{Ind}}
\newcommand{\sym}{{\mathrm{Sym}}}
\newcommand{\qsym}{{\mathrm{Qsym}}}
\newcommand{\DES}[1]{{\mathrm{DES}(#1)}}
\newcommand{\FIX}[1]{{\mathrm{FIX}(#1)}}
\newcommand{\fix}[1]{{\mathrm{fix}(#1)}}
\newcommand{\std}[1]{{\mathrm{std}(#1)}}

\newcommand{\grass}[2]{{\mathrm{Gr}(#1,#2)}}
\newcommand{\bound}[2]{{\mathrm{Bound}(#1,#2)}}
\newcommand{\codim}[2]{{\mathrm{codim}_{#1}(#2)}}
\newcommand{\F}{{\mathcal{F}}}
\newcommand{\alt}{{\operatorname{Alt}}}

\title{A family of algebras that connects positroids,\\the Theta conjecture, the Delta conjecture,\\and Eulerian quasisymmetric functions}

\author{Hai Zhu}

\address{Department of Mathematics, UC San Diego, La Jolla, CA 92093, USA}
\email{haz138@ucsd.edu}
\date{\today}

\begin{document}

\begin{abstract}
    We define a family of algebras $\mathbb{Q}[\mathbf{x}_n]/I_{n,a,b}$ whose Hilbert series, when $n\le a+b\le n+2$, $a>0$, and $0<b<n$, enumerate certain positroid cells by codimension. Equivalently, these series enumerate certain affine permutations by length. We determine the graded $\mathfrak{S}_n$-module structure of these algebras by using exact sequences to derive an invertible recursion for their graded Frobenius characteristics when $a+b\ge n$ and then passing to the limit $a\rightarrow\infty$. In particular, the graded Frobenius characteristic $\mathrm{grFrob}(\mathbb{Q}[\mathbf{x}_{a+b-1}]/I_{a+b-1,a,b};q)=\operatorname{rev}_q(\omega((\Theta_{e_{a-1}}\Theta_{e_{b-1}}\nabla e_1)\mid_{t=0}))$, where $\rev_q$ reverses the $q$-coefficients and the specialization $(\Theta_{e_{a-1}}\Theta_{e_{b-1}}\nabla e_1)\mid_{t=0}$ appears in the Theta conjecture. For $a,b\ge1$, we also show that $e_a^\perp\mathrm{grFrob}(\mathbb{Q}[\mathbf{x}_{a+b}]/I_{a+b,a,b};q)=q^{\binom{a}{2}}\operatorname{rev}_q(\omega((\Delta_{e_{a-1}}^\prime e_b)\mid_{t=0}))$, where $\Delta_{e_{k-1}}^\prime e_n$ occurs in the Delta conjecture. In the three cases $a+b-n=0,1,2$, the ungraded Frobenius characteristics agree with three families of Eulerian quasisymmetric functions, respectively. We also construct an open subvariety of $(\mathbb{P}^{k-1})^n$ whose rational Chow ring is isomorphic to $\mathbb{Q}[\mathbf{x}_n]/I_{n,k,n-k}$. This variety admits a stratification induced by certain positroids. Its strata need not be affine, and their codimension generating function equals $\mathrm{Hilb}(\mathbb{Q}[\mathbf{x}_n]/I_{n,k,n-k};q)$.
\end{abstract}

\maketitle

\section{Introduction}\label{sec:introduction}
Let $\xxx_n=(x_1,x_2,\dots,x_n)$ be the list of variables. The symmetric group $\symm_n$ acts on the polynomial ring $\QQ[\xxx_n]$ by permuting variables, making the quotient $\QQ[\xxx_n]/I$ a graded $\symm_n$-module for any $\symm_n$-invariant homogeneous ideal $I$ with $\deg x_i=1$. We give new constructions below, where the graded Frobenius characteristic (see Subsection~\ref{subsec:symfunc-symgp}) of a graded $\symm_n$-module $A$ is $\grFrob(A;q)\coloneqq\sum_{d\ge 0}q^d\cdot\Frob(A_d)$ and we write $\Hilb(A;q)\coloneqq\sum_{d=0}^\infty q^d\dim(A_d)$ for the Hilbert series of $A$.

\begin{definition}\label{def:symfpoly}
    For any subset $S\subseteq[n]$ and integer $b\ge 0$, let \[e_b(S)\coloneqq\sum_{T\in\binom{S}{b}}\prod_{i\in T}x_i\]
    be the elementary symmetric polynomial on the variable set $\{x_i\,;\, i\in S\}$, with the convention that
    $e_0(S)=1$ and $e_b(S)=0$ for $b>\lvert S\rvert$.
\end{definition}

\begin{definition}\label{def:R_{n,a,b}}
    For integers $n>0$ and $a,b\ge0$, define two ideals $I_{n,a,b},I_{n,\infty,b}\subseteq\mathbb{Q}[\xxx_n]$ by
    \[I_{n,a,b}\coloneqq\langle x_i^a,e_b([n]\setminus\{i\})\,:\, i\in[n] \rangle\]
    and
    \[I_{n,\infty,b}\coloneqq\langle e_b([n]\setminus\{i\})\,:\, i\in[n] \rangle.\]
    Let $R_{n,a,b}\coloneqq\mathbb{Q}[\xxx_n]/I_{n,a,b}$ and $R_{n,\infty,b}\coloneqq\mathbb{Q}[\xxx_n]/I_{n,\infty,b}$.
\end{definition}

\begin{definition}\label{def:F_{n,a,b}}
    For integers $n>0$ and $a,b\ge0$, let \[\F_{n,a,b}\coloneqq\grFrob(R_{n,a,b};q)\] and \[\F_{n,\infty,b}\coloneqq\grFrob(R_{n,\infty,b};q).\]
\end{definition}

The family of graded $\symm_n$-algebras $R_{n,a,b}$ defined above connects four seemingly different objects. Their Hilbert series admit interpretations in terms of positroid cells and affine permutations, while their graded Frobenius characteristics are related to Theta operators, Delta operators, and Eulerian quasisymmetric functions. We also provide a geometric realization of a subfamily through the Chow ring of an open subvariety of a product of projective spaces. We review these objects and present our corresponding main results, respectively, in the four following subsections.

\subsection{Main results about positroids and affine permutations}\label{subsec:intto-posiroid}
Positroids and bounded affine permutations are equivalent up to a bijection. Both are defined in Subsection~\ref{subsec:positroids-affine-perm}; see Speyer's survey~\cite{speyer2024richardsonvarietiesprojectedrichardson} for a detailed overview. The combinatorics of positroids was systematically developed by Postnikov in his study of the totally nonnegative Grassmannian \cite{Postnikov2006}, which admits a cell decomposition indexed by positroids. Knutson, Lam, and Speyer studied the corresponding stratification of the complex Grassmannian by open positroid varieties, indexed by bounded affine permutations, whose lengths give the codimensions of the corresponding strata \cite{KnutsonLamSpeyer2013}. Feigin, Lanini, and Pütz constructed quiver Grassmannians admitting cellular decompositions whose cell posets are dual to those of the totally nonnegative Grassmannians \cite{FeiginLaniniPutz2023}. Positroids are also closely connected to Schubert matroids \cite{Oh2011} and cluster algebras \cite{GalashinLam2023}.

We give the Frobenius characteristics of a subfamily of $R_{n,a,b}$ via affine permutations, where $F_{S,n}$ denotes fundamental quasisymmetric functions defined in Subsection~\ref{subsec:quasisym}, $\dschub_n^k$ is the set of affine permutations of average displacement $k$ as defined in Subsection~\ref{subsec:positroids-affine-perm}, and $\DES{w}$ is the descent set of a word $w$ defined in Subsection~\ref{subsec:combinatorics}.

\begin{theorem}[{Corollary~\ref{cor:a+b-n=0,1,2}}]\label{thm:main:a+b-n=0,1,2}
    Let $n,a,b>0$ satisfy $0\le a+b-n\le 2$ and $b<n$. Then
    \[\F_{n,a,b} = \sum_{\substack{f\in\widetilde{S}_n^0 \\ \text{$i-a<f(i)<i+b$ for all $i$}}} q^{\ell(f)} \cdot F_{\DES{f\mid_{[n]}},n}.\]
\end{theorem}
Most of the remaining main theorems mentioned in this section follows from this theorem.

In this theorem, we may use the translation bijection $f\mapsto f+k$ to change $\dschub_n^0$ to $\dschub_n^k$ in the summing range, and an appropriate choice of $k$, such as $k=a-1$, makes $f$ a bounded affine permutation (i.e. $i\le f(i)\le i+n$ for all $i$). Write $\bound{k}{n}$ for the set of bounded affine permutations in $\dschub_n^k$, and let $\Pi_f$ be the positroid variety in the Grassmannian $\grass{k}{n}$ indexed by $f$; see Subsection~\ref{subsec:positroids-affine-perm} for both definitions. This theorem immediately implies the Hilbert series formula counting affine permutations or, equivalently, positroid varieties. 
\begin{theorem}[{Corollary~\ref{cor:positroid}}]\label{thm:main:positroid-counting}
    For $1<k\le n$, we have
    \[\Hilb(R_{n,k+1,n-k+1};q) = \sum_{f\in\bound{k}{n}} q^{\ell(f)} = \sum_{f\in\bound{k}{n}} q^{\codim{\grass{k}{n}}{\Pi_f}}.\]
\end{theorem}

\subsection{Main results related to the Delta conjecture and the Theta conjecture}\label{subsec:intro-delta-theta}
Relevant terminologies in this section are defined in Subsection~\ref{subsec:theta}. Let $\rev_q$ denotes the reversal of $q$-coefficients: For $p=a_0+a_1q+\dots+a_dq^d\in R[q]$ with $a_0,\dots,a_d\in R$ and $a_d\neq0$, define $\rev_q(p)\coloneqq a_d+a_{d-1}q+\dots+a_0q^d$, with the convention that $\rev_q 0=0$

The theory of Macdonald polynomials and diagonal coinvariants has motivated the study of several operators on symmetric functions, including the nabla operator $\nabla$ and the Delta operators $\Delta_g$ and $\Delta_g^\prime$. Haglund, Remmel, and Wilson formulated the Delta conjecture, which predicts combinatorial interpretations of $\Delta'_{e_{k-1}}e_n$ in terms of decorated labelled Dyck paths \cite{HaglundRemmelWilson2018}. Its rise version was subsequently proved by D'Adderio and Mellit \cite{DAdderioMellit2022}. At $t=0$, Haglund, Rhoades, and Shimozono related these symmetric functions to generalized coinvariant algebras, whose combinatorics is governed by ordered set partitions \cite{HaglundRhoadesShimozono2018}. More precisely, their graded Frobenius characteristics agree with $\rev_q(\omega((\Delta'_{e_{k-1}}e_n)\mid_{t=0}))$. The subfamily of algebras $R_{n,a,b}$ such that $a+b=n$ provides another representation theoretical model for this symmetric function. 
\begin{theorem}[{Corollary~\ref{cor:delta}}]\label{thm:main:e_k-perp}
    For $a,b\ge 1$, we have
    \[e_a^\perp\F_{a+b,a,b} = q^{\binom{a}{2}}\rev_q(\omega((\Delta'_{e_{a-1}}e_{b})\mid_{t=0})).\]
\end{theorem}
Here $e_a^\perp$ is the adjoint operator of multiplying $e_a$; see Subsection~\ref{subsec:symfunc-symgp} for details.

Extending the framework of the Delta conjecture, D'Adderio, Iraci, and Vanden Wyngaerd introduced the Theta operators $\Theta_g$ to formulate refinements of the Delta conjecture and conjectural formulas for the Frobenius characteristics of super-diagonal coinvariant rings \cite{DAdderioIraciVandenWyngaerd2021}. In particular, the symmetric functions $\left.(\Theta_{e_k}\Theta_{e_\ell}\nabla e_{n-k-\ell})\right|_{t=0}$ are conjecturally related to coinvariant rings with one set of commuting variables and two sets of anticommuting variables. Iraci, Nadeau, and Vanden Wyngaerd gave combinatorial interpretations of these specializations in terms of segmented Smirnov words, establishing further connections with the Delta conjectures \cite{MR4766722}. The subfamily of algebras $R_{n,a,b}$ such that $a+b=n+1$ gives a new model for this symmetric function when $n-k-\ell=1$.
\begin{theorem}[{Theorem~\ref{thm:theta-conjecture}}]\label{thm:main:theta}
    For integers \(a,b>0\), we have
    \[
        \F_{a+b-1,a,b}
        =
        \rev_q\Big(
        \omega\left(
        \left.
        \bigl(
        \Theta_{e_{a-1}}\Theta_{e_{b-1}}\nabla e_1
        \bigr)\right|_{t=0}
        \right)\Big).
    \]
\end{theorem}

\subsection{Main results connected to Eulerian quasisymmetric functions}\label{subsec:intro-eulerian}
See Subsection~\ref{subsec:eulerian} for detailed definitions. Eulerian quasisymmetric functions were introduced by Shareshian and Wachs as symmetric-function refinements of the classical Eulerian polynomials and Eulerian numbers \cite{ShareshianWachs2010}. The functions $Q_{n,k}$ are defined as sums of fundamental quasisymmetric functions indexed by permutations with $k$ excedances. Despite their quasisymmetric definition, these functions are symmetric and arise as the Frobenius characteristics of the symmetric-group representations on graded components of the cohomology of the permutahedral toric variety. Shareshian and Wachs also introduced the refined functions $Q_{n,k,r}$ by recording the number of fixed points, with $Q_{n,k,0}$ corresponding to derangements \cite{ShareshianWachs2010}.

In a related development, Shareshian and Wachs introduced the symmetric-function analogues $\widetilde Q_n(\xxx,t)$ of the binomial Eulerian polynomials, whose coefficients $\widetilde{Q}_{n,k}\coloneqq[t^k]\widetilde Q_n(\xxx,t)$ describe the symmetric-group representations on graded components of the cohomology of the stellahedral toric variety \cite{ShareshianWachs2020}. Liao subsequently constructed explicit permutation bases for these representations using Stembridge codes and their extensions. He also obtained a fundamental quasisymmetric expansion of $\widetilde Q_n(\xxx,t)$ in terms of decorated permutations, extending the combinatorial interpretation of the original Eulerian quasisymmetric functions \cite{Liao2024}.

Forgetting the grading of our family of algebras $R_{n,a,b}$ at $a+b-n=0,1,2$ gives $Q_{n,k,0},Q_{n,k},\widetilde{Q}_{n,k}$, respectively.
\begin{theorem}[{Theorem~\ref{thm:a+b-n=0:Eulerian}}]\label{thm:main:eulerian-0}
    Let $n$ be a positive integer.  For every $0\leq k\leq n$,
    \[
        \Frob(R_{n,k,n-k})=Q_{n,k,0}.
    \]
\end{theorem}
\begin{theorem}[{Theorem~\ref{thm:a+b-n=1:Eulerian}}]\label{thm:main:eulerian-1}
    Let $n$ be a positive integer.  For every $0\leq k\leq n$,
    \[
        \Frob(R_{n,k+1,n-k})=Q_{n,k}.
    \]
\end{theorem}
\begin{theorem}[{Corollary~\ref{cor:Eulerian-2}}]\label{thm:main:eulerian-2}
    For $1<k\le n$,
    \[\Frob(R_{n,k+1,n-k+1}) = \widetilde{Q}_{n,k}.\]
\end{theorem}

Interestingly, comparing Theorem~\ref{thm:main:e_k-perp} and Theorem~\ref{thm:main:eulerian-0} immediately proves the following connection between the Delta conjecture and Eulerian quasisymmetric functions.
\begin{corollary}\label{cor:euler-delta}
    For $a,b\ge 1$,
    \[e_a^\perp Q_{a+b,a,0}=\omega((\Delta'_{e_{a-1}}e_{b})\mid_{q=1,\,t=0}).\]
\end{corollary}

Similarly, comparing Theorem~\ref{thm:main:theta} and Theorem~\ref{thm:main:eulerian-1} immediately yields the following connection between the Theta conjecture and Eulerian quasisymmetric functions, which also follows from \cite{liao2026bicoloredeulerianquasisymmetricfunctions}.
\begin{corollary}\label{cor:euler-theta}
    For $a,b\ge 1$,
    \[Q_{a+b-1,a-1}=\omega\left(
        \left.
        \bigl(
        \Theta_{e_{a-1}}\Theta_{e_{b-1}}\nabla e_1
        \bigr)\right|_{q=1,\,t=0}
        \right).\]
\end{corollary}

\subsection{Main results about Chow ring interpretation}
\label{subsec:intro-chow}

Chow rings and relevant standard results are reviewed in
Subsection~\ref{subsec:chow}.
Pawlowski and Rhoades~\cite{MR4029695} constructed an open
subscheme $X_{n,k}\subseteq(\PP^{k-1})^n$ by
\[
X_{n,k}\coloneqq
\left\{
(\ell_1,\dots,\ell_n)\in(\PP^{k-1})^n:
\sum_{i=1}^n\ell_i=\CC^k
\right\}
\]
and proved that its cohomology ring is isomorphic to the
generalized coinvariant algebra defined in
\cite{HaglundRhoadesShimozono2018}, with graded Frobenius
character
$\rev_q(\omega((\Delta'_{e_{k-1}}e_n)\mid_{t=0}))$.

We construct a similar space of line configurations
\[
U_{n,k}\coloneqq
\left\{
(\ell_1,\dots,\ell_n)\in(\PP^{k-1})^n:
\sum_{\substack{i=1\\i\neq j}}^n\ell_i=\CC^k
\text{ for all }j\in[n]
\right\}
\]
to provide $R_{n,k,n-k}$ with a Chow ring interpretation.

\begin{theorem}[{Theorem~\ref{thm:chow}}]\label{thm:main:chow}
    For $1\le k< n$, we have the $\QQ$-algebra
    isomorphism
    \[
    CH^*_\QQ(U_{n,k})\cong R_{n,k,n-k}
    \]
    which maps $c_1(\mathcal{L}_i)$ to the residue
    class of $x_i$ for all $i\in[n]$, where
    $\mathcal{L}_i$ is the pullback of the
    tautological line bundle over the $i$-th
    projective space factor.
\end{theorem}

The space $U_{n,k}$ admits a positroid stratification
indexed by strictly bounded affine permutations
\[
\bound{k}{n}^{\circ}\coloneqq
\{f\in\bound{k}{n}:i<f(i)<i+n
\text{ for all }i\in[n]\},
\]
which correspond to positroids with neither loops nor
coloops. Let $U_f$ denote the stratum indexed by
$f\in\bound{k}{n}^{\circ}$.
The codimension of $U_f$ in $U_{n,k}$ equals
$\ell(f)$, as for the corresponding positroid variety
in $\grass{k}{n}$~\cite{KnutsonLamSpeyer2013}.
Combining Theorems~\ref{thm:main:chow}
and~\ref{thm:main:a+b-n=0,1,2}, we obtain the
following geometric interpretation of the Hilbert
series of $R_{n,k,n-k}$; see Section~\ref{sec:chow} for details.

\begin{corollary}[{Corollary~\ref{cor:chow-positroid-hilbert}}]\label{cor:main:chow-positroid}
    For $1\le k<n$, we have
    \[
    \begin{aligned}
    \Hilb(R_{n,k,n-k};q)
    =\Hilb(CH^*_\QQ(U_{n,k});q)=\sum_{f\in\bound{k}{n}^{\circ}}q^{\ell(f)}=\sum_{f\in\bound{k}{n}^{\circ}}
      q^{\codim{U_{n,k}}{U_f}}.
    \end{aligned}
    \]
\end{corollary}
However, the author cannot geometrically explain why the codimension generating function of these strata agrees with the Hilbert series of the Chow ring.

\subsection{Outline}\label{subsec:outline}
The rest of this paper is structured as follows. We first review standard or technical definitions and results in Section~\ref{sec:background}. Then we use Sections~\ref{sec:F_{n,a,b}-recurs}, \ref{sec:F_{n,infty,b}-recurs}, and \ref{sec:F_{n,a,b}-closed} to prove a closed formula of $\F_{n,a,b}$ for $a+b\ge n$, which yields Theorem~\ref{thm:main:a+b-n=0,1,2}, the starting point of all main results. Our proof strategy is providing an invertible recursive formula for $\F_{n,a,b}$ with $a<\infty$ and $a+b\ge n$ in Section~\ref{sec:F_{n,a,b}-recurs} so that we can use the outputs of $\F_{n,a,b}$ for larger values of $a$ to compute their counterparts for smaller values of $a$. Consequently, it suffices to let $a\rightarrow\infty$ and turn to study $\F_{n,\infty,b}$, which will be handled via the recursive formula in Section~\ref{sec:F_{n,infty,b}-recurs}. With these two recursive formulas, we eventually give the closed formula for $\F_{n,a,b}$ in Section~\ref{sec:F_{n,a,b}-closed}.

After proving this closed formula, we focus on the cases when $a+b-n=0,1,2$. Section~\ref{sec:Theta-conjecture} reveals the connections of the case when $a+b-n=0,1$ to the Delta conjecture and the Theta conjecture, respectively. Section~\ref{sec:eulerian-quasisym} gives the connections of the cases when $a+b-n=0,1,2$ to three families of Eulerian symmetric functions, respectively. Section~\ref{sec:chow} gives a geometric interpretation to the case when $a+b-n=0$. Finally, Section~\ref{sec:future} proposes our open problems to study in the future.

\section{Background}\label{sec:background}

\subsection{Combinatorics}\label{subsec:combinatorics}
A \emph{partition} of $n$ is a finite weakly decreasing sequence
$\lambda=(\lambda_1,\dots,\lambda_\ell)$ of positive integers
such that $|\lambda|\coloneqq\lambda_1+\cdots+\lambda_\ell=n$;
we write $\lambda\vdash n$ and call $\ell$ its \emph{length}.
The empty partition is the unique partition of $0$.
The \emph{Young diagram} of $\lambda$, drawn in English convention,
consists of left-justified rows of boxes, with $\lambda_i$ boxes
in the $i$-th row from the top.
A \emph{semistandard Young tableau} (SSYT) of shape $\lambda$
is a filling of this diagram with positive integers that
weakly increase along rows from left to right and strictly
increase down columns.
A \emph{standard Young tableau} (SYT) of shape $\lambda\vdash n$
is an SSYT containing each of $1,\dots,n$ exactly once;
its entries therefore strictly increase along both rows and columns.
We denote these sets by $\operatorname{SSYT}(\lambda)$
and $\operatorname{SYT}(\lambda)$, respectively.

For example, $\lambda=(3,2)\vdash5$ has length $2$, and
\[
T=
\begin{ytableau}
1 & 1 & 3\\
2 & 4
\end{ytableau},
\qquad
S=
\begin{ytableau}
1 & 2 & 4\\
3 & 5
\end{ytableau}
\]
are an SSYT and an SYT of shape $(3,2)$, respectively.

For a word $w=w_1\cdots w_n$ over a totally ordered alphabet,
its \emph{descent set} is
\[
\operatorname{DES}(w)
\coloneqq\{i\in[n-1]:w_i>w_{i+1}\}.
\]
For example, $\operatorname{DES}(21332)=\{1,4\}$. We identify $w$ with the map $i\mapsto w_i$ form $[n]$ to the alphabet.

\subsection{Quasisymmetric functions}\label{subsec:quasisym}

Throughout this paper, symmetric and quasisymmetric functions have
coefficients in $\QQ$.
We write $\xxx=(x_1,x_2,\dots)$ for their infinite ordered list
of variables, which should be distinguished from the finite list
$\xxx_n=(x_1,\dots,x_n)$ of polynomial variables in $\QQ[\xxx_n]$.

A bounded-degree formal power series $f$ in $\xxx$ with coefficients
in $\QQ$ is \emph{quasisymmetric} if
\[
[x_{i_1}^{a_1}\cdots x_{i_r}^{a_r}]f
=
[x_{j_1}^{a_1}\cdots x_{j_r}^{a_r}]f
\]
for all positive integers $a_1,\dots,a_r$ and strictly increasing
index sequences $i_1<\cdots<i_r$ and $j_1<\cdots<j_r$.
These functions form the graded $\QQ$-algebra
$\qsym=\bigoplus_{n\geq0}\qsym_n$;
see \cite[Section~4.1]{MR3810249}.

For $n\geq1$ and $S\subseteq[n-1]$, define the
\emph{fundamental quasisymmetric function}
\[
F_{S,n}(\xxx)=
\sum_{\substack{1\leq i_1\leq\cdots\leq i_n\\
j\in S\Rightarrow i_j<i_{j+1}}}
x_{i_1}\cdots x_{i_n},
\]
and set $F_{\varnothing,0}=1$.
The functions $\{F_{S,n}:S\subseteq[n-1]\}$ form a linear basis
of $\qsym_n$.
We use increasing indices; the decreasing-index convention of
\cite[Section~2]{ShareshianWachs2010} instead defines
\[
F^{\downarrow}_{S,n}(\xxx)=
\sum_{\substack{i_1\geq\cdots\geq i_n\geq1\\
j\in S\Rightarrow i_j>i_{j+1}}}
x_{i_1}\cdots x_{i_n}.
\]
Reversing the index sequence gives
$F_{S,n}=F^{\downarrow}_{n-S,n}$, where
$n-S\coloneqq\{n-j:j\in S\}$.
In particular,
$F_{\varnothing,n}=h_n$ and $F_{[n-1],n}=e_n$,
with $h_n,e_n$ defined below.

A quasisymmetric function is \emph{symmetric} if it is invariant
under every finite permutation of the variables $\xxx$. For a homogeneous symmetric function $f$ of degree $n$, the two
conventions give the same fundamental expansion coefficients:
\[
f=\sum_{S\subseteq[n-1]}c_SF_{S,n}
\quad\Longleftrightarrow\quad
f=\sum_{S\subseteq[n-1]}c_SF^{\downarrow}_{S,n}.
\]
Indeed, after specializing to $x_1,x_2,\dots,x_N$, reversing the variables
fixes every symmetric polynomial and interchanges $F_{S,n}$ with
$F^{\downarrow}_{S,n}$. The assertion follows since this holds for
every $N$.

Let $u=u_1\cdots u_m$ and $v=v_1\cdots v_n$ be words on a
totally ordered alphabet, with all $m+n$ letters distinct.
A \emph{shuffle} of $u$ and $v$ is a word obtained by interleaving
their letters while preserving the order within each word.
Write $\operatorname{Sh}(u,v)$ for the set of such shuffles.
The fundamental basis satisfies the multiplication rule
\cite[Theorem~4.1]{MR3810249}
\[
F_{\operatorname{DES}(u),m}\,
F_{\operatorname{DES}(v),n}
=
\sum_{w\in\operatorname{Sh}(u,v)}
F_{\operatorname{DES}(w),m+n}.
\]

Define the $\QQ$-linear involution $\Omega:\qsym\to\qsym$ by
\[
\Omega(F_{S,n})=F_{[n-1]\setminus S,n},
\qquad
\Omega(1)=1,
\]
and write $\Omega_n$ for its restriction to $\qsym_n$.
On symmetric functions, $\Omega$ coincides with the involution
$\omega$ defined below; see
\cite[Section~3.6]{LuotoMykytiukVanWilligenburg2013}.

\subsection{Symmetric functions and the representation theory of symmetric groups}
\label{subsec:symfunc-symgp}

A quasisymmetric function is \emph{symmetric} if it is invariant
under every finite permutation of the variables $\xxx$.
The symmetric functions form the graded $\QQ$-subalgebra
$\sym=\bigoplus_{n\geq0}\sym_n$ of $\qsym$.
We recall standard notation; see \cite[Chapter~7]{MR4621625}.

For $r\geq1$, the elementary, complete homogeneous, and power-sum
symmetric functions are
\[
\begin{aligned}
e_r(\xxx)&=\sum_{1\leq i_1<\cdots<i_r}x_{i_1}\cdots x_{i_r},\\
h_r(\xxx)&=\sum_{1\leq i_1\leq\cdots\leq i_r}x_{i_1}\cdots x_{i_r},\\
p_r(\xxx)&=\sum_{i\geq1}x_i^r,
\end{aligned}
\]
with $e_0=h_0=1$ and the conventions that $e_r=h_r=0$ if $r<0$.
For a partition $\lambda=(\lambda_1,\dots,\lambda_\ell)$,
write $e_\lambda=\prod_i e_{\lambda_i}$, and similarly
define $h_\lambda$ and $p_\lambda$.
The Schur function of shape $\lambda$ is
\[
s_\lambda(\xxx)=
\sum_{T\in\operatorname{SSYT}(\lambda)}
\prod_{c\in\lambda}x_{T(c)},
\]
where $c$ ranges over the boxes of the Young diagram.
Each of $\{e_\lambda\}$, $\{h_\lambda\}$, $\{p_\lambda\}$,
and $\{s_\lambda\}$ is an $\QQ$-basis of $\sym$.

The \emph{Hall inner product} is the $\QQ$-bilinear form
determined by
$\langle s_\lambda,s_\mu\rangle=\delta_{\lambda\mu}$.
For $f\in\sym$, let $f^\perp$ be the adjoint of multiplication by $f$:
\[
\langle f^\perp g,h\rangle=\langle g,fh\rangle
\qquad(g,h\in\sym).
\]
The $\QQ$-algebra involution $\omega$ is determined by
$\omega(e_r)=h_r$; equivalently,
\[
\omega(p_r)=(-1)^{r-1}p_r,
\qquad
\omega(s_\lambda)=s_{\lambda'},
\]
where $\lambda'$ is obtained by transposing the Young diagram. In particular, $\omega$ is unitary with respect to the Hall inner product. 

All these definitions can be linearly extended to $\sym\otimes_\QQ \QQ(q,t)$.

A symmetric function $f\in\sym\otimes_\QQ \QQ[[q]]$ is called \emph{Schur positive} if all its Schur coefficients belong to $\QQ_{\ge 0}[[q]]$. We define a partial order $\le$ on $f\in\sym\otimes_\QQ \QQ[[q]]$ by $f\le g\Leftrightarrow\text{$g-f$ is Schur positive}$.

We also use \emph{plethystic substitution}
\cite[Section~3.3]{MR2051783}.
Inside plethystic brackets, by a slight abuse of notation, $\xxx$ denotes the formal sum of
all variables $x_i$, so $\xxx=x_1+x_2+\cdots$.
For a symmetric function $A=A(x_1,x_2,\dots;q,t)\in\sym\otimes_\QQ\QQ(q,t)$, set
\[
p_r[A]=A(x_1^r,x_2^r,\dots;q^r,t^r).
\]
Since $\sym\otimes_\QQ\QQ(q,t)=\QQ(q,t)[p_1,p_2,\dots]$ where $p_1,p_2,\dots$ are algebraically independent, the plethystic substitution $f[A]$ with $f\in\sym\otimes_\QQ\QQ(q,t)$
is uniquely defined by replacing each $p_r$ in $f$ by $p_r[A]$
and keeping the coefficients of $f$ in $\QQ(q,t)$ fixed.
In particular,
\[
p_r\!\left[\frac{\xxx}{1-q}\right]
=\frac{p_r(\xxx)}{1-q^r},
\qquad
p_r[-\xxx]=-p_r(\xxx),\qquad
f[-\xxx]=(-1)^d\omega(f)\text{ if $f\in\sym_d\otimes_\QQ\QQ(q,t)$}.
\]
For $f,g\in\sym\otimes_\QQ\QQ(q,t)$, the expression $f[g]$ is the
\emph{plethysm} of $f$ and $g$.
We will also use the addition formulas
$h_r[A+B]=\sum_{i=0}^r h_i[A]h_{r-i}[B]$ and $e_r[A+B]=\sum_{i=0}^r e_i[A]e_{r-i}[B]$.

We study the representation theory of symmetric groups via the Frobenius characteristic; see \cite{MR1824028}.
For a finite-dimensional $\QQ[\mathfrak{S}_n]$-module $V$
with character $\chi_V$, its \emph{Frobenius characteristic} is
\[
\Frob(V)=
\sum_{\lambda\vdash n}\frac{\chi_V(\lambda)}{z_\lambda}p_\lambda,
\qquad
z_\lambda=\prod_{i\geq1}i^{m_i(\lambda)}m_i(\lambda)!,
\]
where $m_i(\lambda)$ is the multiplicity of $i$ in $\lambda$
and $\chi_V(\lambda)$ is the character value on cycle type $\lambda$. All $\symm_n$-irreducibles $\{S^\lambda\}_{\lambda\vdash n}$ are called \emph{Specht modules}.
If $V\cong\bigoplus_{\lambda\vdash n}(S^\lambda)^{\oplus a_\lambda}$
is its decomposition into irreducibles, then
$\Frob(V)=\sum_{\lambda\vdash n}a_\lambda s_\lambda$.
In particular,
\[
\Frob(\mathbf{1})=h_n,
\qquad
\Frob(\operatorname{sgn})=e_n,
\qquad
\Frob(V\otimes\operatorname{sgn})=\omega(\Frob(V)).
\]

For a graded $\QQ[\mathfrak{S}_n]$-module
$V=\bigoplus_{d\geq0}V_d$ with finite-dimensional homogeneous
components and a degree-preserving action, define its \emph{graded Frobenius characteristic}
\[
\grFrob(V;q)=\sum_{d\geq0}q^d\Frob(V_d).
\]
If $V$ is infinite-dimensional, the sum is interpreted as a formal
$q$-series.
Its Schur coefficients record graded multiplicities, and
\[
\Hilb(V;q)=\sum_{d\geq0}q^d\dim V_d
=\big\langle\grFrob(V;q),p_1^n\big\rangle.
\]

Under the Frobenius characteristic map, the operator
$e_k^\perp$ corresponds to antisymmetrization with respect
to $k$ letters; see, e.g., \cite{MR4621625,MR1824028}.
More precisely, for a $\QQ[\mathfrak{S}_n]$-module $V$ and
$0\le k\le n$, define the $\mathfrak{S}_{n-k}$-module
\[
\operatorname{Alt}_k(V)
\coloneqq
\left\{
v\in V:
\sigma\cdot v=\operatorname{sgn}(\sigma)v
\text{ for all }\sigma\in\mathfrak{S}_k
\right\},
\]
where $\mathfrak{S}_{n-k}\times\mathfrak{S}_k$ is regarded
as the Young subgroup of $\mathfrak{S}_n$ acting on
the first $n-k$ and last $k$ letters, respectively.
Equivalently, $\operatorname{Alt}_k(V)$ is the image of
the antisymmetrizer
\[
\frac{1}{k!}
\sum_{\sigma\in\mathfrak{S}_k}
\operatorname{sgn}(\sigma)\sigma
\]
acting on $V$. We then have
\[
\Frob(\operatorname{Alt}_k(V))
=
e_k^\perp\Frob(V).
\]
The same identity holds for graded modules, with
$\Frob$ replaced by $\grFrob$.

\subsection{Modified Macdonald symmetric functions and relevant operators}
\label{subsec:theta}

Throughout this section, all functions $f,g\in\sym\otimes_\QQ\QQ(q,t)$. For $\mu\vdash n$, the \emph{modified Macdonald symmetric function}
$\widetilde H_\mu=\widetilde H_\mu(\xxx;q,t)\in\sym_n\otimes_\QQ\QQ(q,t)$ is uniquely
characterized by
\[
\begin{aligned}
\widetilde H_\mu[(1-q)\xxx]
&\in\operatorname{span}_{\QQ(q,t)}
   \{s_\lambda:\lambda\vdash n,\ \lambda\unrhd\mu\},\\
\widetilde H_\mu[(1-t)\xxx]
&\in\operatorname{span}_{\QQ(q,t)}
   \{s_\lambda:\lambda\vdash n,\ \lambda\unrhd\mu'\},\\
\langle\widetilde H_\mu,s_{(n)}\rangle&=1,
\end{aligned}
\]
where $\lambda\unrhd\mu$ means
$\sum_{i=1}^j\lambda_i\geq\sum_{i=1}^j\mu_i$ for every $j$;
see \cite[Definition~3.5.2]{MR2051783}.
Set $\widetilde H_\varnothing=1$.
These functions form an $\QQ(q,t)$-basis of $\sym\otimes_\QQ\QQ(q,t)$, and their
Schur coefficients belong to $\mathbb Z_{\geq0}[q,t]$.

For a partition $\mu$, let $(i,j)$ denote the cell in row $i$ and column $j$, and set
\[
B_\mu=\sum_{(i,j)\in\mu}q^{j-1}t^{i-1},
\qquad
T_\mu=\prod_{(i,j)\in\mu}q^{j-1}t^{i-1},
\qquad
\Pi_\mu=\prod_{\substack{(i,j)\in\mu\\(i,j)\neq(1,1)}}
          (1-q^{j-1}t^{i-1}).
\]
The \emph{nabla operator} and \emph{Delta operators} are the
$\QQ(q,t)$-linear operators determined by
\[
\nabla\widetilde H_\mu=T_\mu\widetilde H_\mu,
\qquad
\Delta_g\widetilde H_\mu=g[B_\mu]\widetilde H_\mu,
\qquad
\Delta'_g\widetilde H_\mu=g[B_\mu-1]\widetilde H_\mu;
\]
see \cite[Section~2]{DAdderioIraciVandenWyngaerd2021}.
On $\sym_n\otimes_\QQ\QQ(q,t)$ with $n\geq1$, we have
\[
\nabla=\Delta_{e_n}=\Delta'_{e_{n-1}},
\qquad
\Delta_{e_r}=\Delta'_{e_r}+\Delta'_{e_{r-1}}
\quad(r\geq1),
\]
and for any $k>n$ we have $\Delta_{e_k}=\Delta_{e_{k-1}}^\prime=0$ on $\sym_n\otimes_\QQ\QQ(q,t)$; see \cite[Section~2]{DAdderioIraciVandenWyngaerd2021}.

Let $M=(1-q)(1-t)$, and define the invertible
$\QQ(q,t)$-linear operator $\Pi$ on $\sym_{\geq1}\otimes_\QQ\QQ(q,t)$ by
$\Pi\widetilde H_\mu=\Pi_\mu\widetilde H_\mu$.
For $g\in\sym\otimes_\QQ\QQ(q,t)$ and $f$ of positive degree, the
\emph{Theta operator} is defined by
\[
\Theta_g(f)=
\Pi\left(g\left[\frac{\xxx}{M}\right]\Pi^{-1}(f)\right).
\]
For homogeneous $g$ of positive degree, set $\Theta_g(1)=0$;
for $c\in\QQ(q,t)$, set $\Theta_c=c\,\mathrm{id}$.
Extend these definitions by linearity.
If $g$ is homogeneous of degree $r$, then $\Theta_g$ raises
symmetric-function degree by $r$.
See \cite[Section~3]{DAdderioIraciVandenWyngaerd2021}.

We shall use the following identities for Theta operators.
By \cite[Theorem~3.1]{DAdderioIraciVandenWyngaerd2021}, for
$1\leq k\leq n$,
\begin{equation}\label{eq:theta-delta}
    \Theta_{e_{n-k}}\nabla e_k=\Delta'_{e_{k-1}}e_n.
\end{equation}
In particular, $\Theta_{e_{r-1}}\nabla e_1=e_r$ for $r\geq1$.
We also use the commutation relation
\begin{equation}\label{eq:theta-h-perp}
    h_j^\perp\Theta_{e_k}
    =
    \sum_{i=0}^{j}
    \Theta_{e_{k-j+i}}\Delta_{e_{j-i}}h_i^\perp
    \qquad(j,k\geq0),
\end{equation}
proved in \cite[Theorem~4.2]{DAdderioRomero2023}.
Here and below, $e_r=h_r=0$ and $\Theta_{e_r}=0$ for $r<0$,
while $e_0=h_0=1$ and $\Theta_{e_0}=\mathrm{id}$.
The coproduct rule for the Hall adjoints gives
\begin{equation}\label{eq:h-perp-product}
    h_j^\perp(fg)
    =
    \sum_{i=0}^{j}(h_i^\perp f)(h_{j-i}^\perp g)\quad(j\ge 0).
\end{equation}
In particular,
\[
    h_i^\perp h_m=h_{m-i}\quad(i\ge 0),
    \qquad
    h_i^\perp e_r=
    \begin{cases}
        e_r,&i=0,\\
        e_{r-1},&i=1,\\
        0,&i\geq2,
    \end{cases}
    \qquad(m\geq0,\ r\geq1).
\]
These standard skewing identities follow from the coproduct and
Pieri rules; see \cite[Chapter~7]{MR4621625}.

Write $[0]_q=0$ and $[j]_q=1+q+\cdots+q^{j-1}$ for $j\geq1$.
For $k\geq1$, define
\[
    c_{k,j}
    =
    (-1)^{k-j}q^{\binom{j+1}{2}-kj}
    \begin{bmatrix}k\\j\end{bmatrix}_q
    \qquad(0\leq j\leq k),
\]
and set $c_{k,j}=0$ otherwise. Here
\[
    \begin{bmatrix}k\\j\end{bmatrix}_q
    =
    \prod_{i=1}^{j}\frac{1-q^{k-j+i}}{1-q^i}
    \qquad(0\leq j\leq k)
\]
is the Gaussian binomial coefficient.
We shall use the finite plethystic expansion
\begin{equation}\label{eq:delta-prime-finite}
    \left.(\Delta'_{e_{k-1}}e_n)\right|_{t=0}
    =
    \sum_{j=0}^{k}c_{k,j}e_n[\xxx[j]_q]
    \qquad(n,k\geq1).
\end{equation}
For $n\geq k$, this follows from \eqref{eq:theta-delta},
the standard specialization
$\left.(\nabla e_k)\right|_{t=0}=\widetilde H_{(k)}$,
and \cite[Corollary~7.3]{DAdderioRomero2023}.
For $n<k$, the left-hand side vanishes, and the right-hand side also vanishes.
Indeed, for every $\lambda\vdash n$,
\[
    p_\lambda[\xxx[j]_q]
    =
    p_\lambda(\xxx)
    \prod_{r=1}^{\ell(\lambda)}
    \frac{1-(q^j)^{\lambda_r}}{1-q^{\lambda_r}}.
\]
This is a polynomial in $(\sym\otimes_\QQ\QQ(q))[q^j]$ of degree at most
$|\lambda|=n$, whose denominators are independent of $j$.
Expanding $e_n$ in the power-sum basis, we may therefore write
\[
    e_n[\xxx[j]_q]
    =
    \sum_{d=0}^{n}g_d(\xxx;q)q^{dj},
\]
where $g_d(\xxx;q)\in\sym_n\otimes_\QQ\QQ(q)$
is independent of $j$.
By the definition of $c_{k,j}$ and the finite
$q$-binomial theorem \cite[(1.10)]{DAdderioRomero2023}, for $0\le d\le n$ we have
\begin{align}\label{eq:sum-vanish}
\begin{aligned}
    \sum_{j=0}^{k}c_{k,j}q^{dj}
    &=
    (-1)^k
    \sum_{j=0}^{k}
    (-1)^j q^{\binom{j}{2}}
    \begin{bmatrix}k\\j\end{bmatrix}_q
    \bigl(q^{d+1-k}\bigr)^j\\
    &=
    (-1)^k\prod_{i=0}^{k-1}
    (1-q^{d+1-k+i})=0
\end{aligned}
\end{align}
where the last equality follows from that for $0\leq d\leq n<k$ the factor indexed by
$i=k-1-d$ is $1-q^0=0$.
Consequently,
\[
    \sum_{j=0}^{k}c_{k,j}e_n[\xxx[j]_q]
    =
    \sum_{d=0}^{n}g_d(\xxx;q)
    \sum_{j=0}^{k}c_{k,j}q^{dj}
    =0.
\]

The following convolution identity will be useful.

\begin{lemma}\label{lem:delta-t-zero-convolution}
For integers $n,k\geq1$,
\[
    \left.
    \left(\sum_{m=0}^{n-1}(-1)^m h_m
    \Delta_{e_k}e_{n-m}\right)
    \right|_{t=0}
    =
    q^{n-k}
    \left.(\Delta'_{e_{k-1}}e_n)\right|_{t=0}.
\]
\end{lemma}

\begin{proof}
The Gaussian binomial recurrence $\begin{bmatrix}k+1\\j\end{bmatrix}_q = q^j\begin{bmatrix}k\\j\end{bmatrix}_q + \begin{bmatrix}k\\j-1\end{bmatrix}_q$ gives
\[
    c_{k,j}+c_{k+1,j}=q^{-k}c_{k,j-1}.
\]
Thus \eqref{eq:delta-prime-finite} and
$\Delta_{e_k}=\Delta'_{e_k}+\Delta'_{e_{k-1}}$ on $\sym_r\otimes_\QQ\QQ(q,t)$ with $k,r\ge 1$ yield
\begin{equation}\label{eq:delta-finite}
    \left.(\Delta_{e_k}e_r)\right|_{t=0}
    =
    q^{-k}\sum_{j=0}^{k}c_{k,j}
    e_r[\xxx[j+1]_q]
    \qquad(k,r\geq1).
\end{equation}
This also holds for $r=0,k\ge 1$: $\Delta_{e_k}1=0$ and
$\sum_{j=0}^{k}c_{k,j}=0$ by \eqref{eq:sum-vanish}.
We may therefore append the zero term $m=n$ to the sum
in the statement of the current lemma. Since
$e_m[-\xxx]=(-1)^m h_m[\xxx]$ and
$[j+1]_q-1=q[j]_q$, the addition formula for elementary
symmetric functions gives
\begin{align}\label{eq:he-e}
    \sum_{m=0}^{n}(-1)^m h_m[\xxx]
    e_{n-m}[\xxx[j+1]_q]
    = \sum_{m=0}^{n}e_m[-\xxx]
    e_{n-m}[\xxx[j+1]_q] =
    e_n[\xxx([j+1]_q-1)]
    =
    q^n e_n[\xxx[j]_q].
\end{align}
Substituting \eqref{eq:delta-finite} and then applying \eqref{eq:he-e} and \eqref{eq:delta-prime-finite} proves the lemma, namely
\begin{align*}
    &\left.
    \left(\sum_{m=0}^{n-1}(-1)^m h_m
    \Delta_{e_k}e_{n-m}\right)
    \right|_{t=0} = \left.
    \left(\sum_{m=0}^{n}(-1)^m h_m
    \Delta_{e_k}e_{n-m}\right)
    \right|_{t=0} \\
    =&\sum_{m=0}^{n}(-1)^m h_m
    q^{-k}\sum_{j=0}^{k}c_{k,j}
    e_{n-m}[\xxx[j+1]_q] = q^{-k}\sum_{j=0}^k c_{k,j} \sum_{m=0}^n (-1)^m h_m e_{n-m}[\xxx[j+1]_q] \\
    =&q^{-k}\sum_{j=0}^k c_{k,j}q^n e_n[\xxx[j]_q] = q^{n-k}\left.(\Delta'_{e_{k-1}}e_n)\right|_{t=0}.
\end{align*}
\end{proof}

For $1\leq k\leq n$, the \emph{generalized coinvariant algebra}
of Haglund, Rhoades, and Shimozono is
\[
    I_{n,k}
    =
    \bigl\langle
        x_1^k,\ldots,x_n^k,\,
        e_n([n]),\ldots,
        e_{n-k+1}([n])
    \bigr\rangle,
    \qquad
    R_{n,k}=\QQ[\xxx_n]/I_{n,k};
\]
see \cite[Definition~1.1]{HaglundRhoadesShimozono2018}.
The grading is induced by $\deg x_i=1$, and $\symm_n$ acts
by permuting the variables. The generators $x_i^k$ make the
quotient finite-dimensional, while the homogeneity and
$\symm_n$-stability of $I_{n,k}$ make $R_{n,k}$ a graded
$\symm_n$-module.
The graded Frobenius formula of
\cite[Theorem~6.11]{HaglundRhoadesShimozono2018} gives
\begin{equation}\label{eq:generalized-coinvariant-delta}
    \grFrob(R_{n,k};q)
    =
    \rev_q\left(
        \omega\left(
            \left.(\Delta'_{e_{k-1}}e_n)\right|_{t=0}
        \right)
    \right)
    \qquad(1\leq k\leq n).
\end{equation}

We record the $q$-degrees needed for coefficient reversal.
For $1\leq a\leq b$, the generalized Artin basis of $R_{b,a}$
\cite[Theorem~4.13]{HaglundRhoadesShimozono2018}, together with
\eqref{eq:generalized-coinvariant-delta}, gives
\[
\deg_q\left(\left.(\Delta'_{e_{a-1}}e_b)\right|_{t=0}\right)
=(a-1)b-\binom{a}{2}.
\]
This specialization vanishes for $a>b$.
For all $a,b\geq1$, the recursion in \cite[Proposition~3.2]{MR4766722}
shows by induction that the maximum $q$-exponent in the
nonnegative fundamental expansion of
\cite[Proposition~5.3]{MR4766722} is
$\binom{n}{2}-\binom{k+1}{2}-\binom{\ell+1}{2}$,
which equals $(a-1)(b-1)$ upon setting
$n=a+b-1$, $k=a-1$, and $\ell=b-1$. It follows that
\[
\deg_q\left(
\left.(\Theta_{e_{a-1}}\Theta_{e_{b-1}}\nabla e_1)\right|_{t=0}
\right)=(a-1)(b-1).
\]
Here $\deg_q$ denotes the largest degree among the Schur
coefficients, which are polynomials in $q$ in these cases.
Consequently,
\[
\begin{aligned}
\rev_q\left(\left.(\Delta'_{e_{a-1}}e_b)\right|_{t=0}\right)
&=
q^{(a-1)b-\binom{a}{2}}\left(
\left.(\Delta'_{e_{a-1}}e_b)\right|_{q=q^{-1},\,t=0}\right)
&& (1\leq a\leq b),\\
\rev_q\left(
\left.(\Theta_{e_{a-1}}\Theta_{e_{b-1}}\nabla e_1)\right|_{t=0}
\right)
&=
q^{(a-1)(b-1)}\left(
\left.(\Theta_{e_{a-1}}\Theta_{e_{b-1}}\nabla e_1)
\right|_{q=q^{-1},\,t=0}\right)
&& (a,b\geq1).
\end{aligned}
\]
Here $q=q^{-1}$ means replacing $q$ with $q^{-1}$.
These formulas remain valid after applying $\omega$.

\subsection{Eulerian quasisymmetric functions}\label{subsec:eulerian}

For $\sigma=\sigma_1\cdots\sigma_n\in\mathfrak S_n$, define
\[
\begin{aligned}
&\operatorname{EXC}(\sigma)=\{i\in[n]:\sigma_i>i\},
\qquad
\operatorname{exc}(\sigma)=|\operatorname{EXC}(\sigma)|,
\\
&\operatorname{fix}(\sigma)=|\{i\in[n]:\sigma_i=i\}|,
\qquad
\operatorname{des}(\sigma)=|\{i\in[n-1]:\sigma_i>\sigma_{i+1}\}|.
\end{aligned}
\]
For $n\geq1$ and $0\leq k\leq n-1$, the \emph{Eulerian number}
$A_{n,k}$ counts permutations of $[n]$ with exactly $k$
excedances, or equivalently, exactly $k$ descents:
\[
    A_{n,k}
    \coloneqq
    \#\{w\in\symm_n:\operatorname{exc}(w)=k\}
    =
    \#\{w\in\symm_n:\operatorname{des}(w)=k\}.
\]
We set $A_{n,k}=0$ for $k<0$ or $k\geq n$. See \cite{FoataSchutzenbergerEulerian1970} for further details.

Let $\overline{\sigma}$ be the word obtained by barring
$\sigma_i$ whenever $i\in\operatorname{EXC}(\sigma)$, and use
the alphabet order
\[
\overline{1}<\cdots<\overline{n}<1<\cdots<n.
\]
Define $\operatorname{DEX}(\sigma)
=\operatorname{DES}(\overline{\sigma})$. This is used to define a quasisymmetric (actually symmetric) function analog $Q_{n,k}$ of $A_{n,k}$. Following \cite[Section~2]{ShareshianWachs2010}, define
\[
Q_{n,k}(\xxx)=
\sum_{\substack{\sigma\in\mathfrak S_n\\
                 \operatorname{exc}(\sigma)=k}}
F_{\operatorname{DEX}(\sigma),n}(\xxx),
\qquad
Q_{n,k,r}(\xxx)=
\sum_{\substack{\sigma\in\mathfrak S_n\\
                 \operatorname{exc}(\sigma)=k\\
                 \operatorname{fix}(\sigma)=r}}
F_{\operatorname{DEX}(\sigma),n}(\xxx).
\]
Thus $Q_{n,k}=\sum_{r=0}^nQ_{n,k,r}$, whereas $Q_{n,k,0}$
is the sum over derangements with $k$ excedances.
Set $Q_{0,0}=Q_{0,0,0}=1$ and interpret empty sums as zero.
These functions are symmetric and $h$-positive
\cite[Theorem~5.1]{ShareshianWachs2010}.

Write $Q_n(\xxx,t)=\sum_{k\geq0}Q_{n,k}(\xxx)t^k$.
The \emph{binomial Eulerian symmetric functions} are defined by
\cite{ShareshianWachs2020}
\begin{equation}\label{eq:tildeQ-Q}
\widetilde Q_n(\xxx,t)
=h_n(\xxx)+t\sum_{m=1}^n h_{n-m}(\xxx)Q_m(\xxx,t),
\qquad
\widetilde Q_{n,k}(\xxx)=[t^k]\widetilde Q_n(\xxx,t).
\end{equation}
In particular,
\[
\widetilde Q_{n,0}=h_n,
\qquad
\widetilde Q_{n,k}
=\sum_{m=1}^n h_{n-m}Q_{m,k-1}
\quad(1\leq k\leq n).
\]

We use the increasing-index convention for $F_{S,n}$ from
Subsection~\ref{subsec:quasisym}, whereas
\cite{ShareshianWachs2010} uses decreasing indices.
Both conventions give the same Eulerian symmetric functions:
after restricting to $x_1,\dots,x_N$, reversing these variables
interchanges $F_{S,n}$ and $F^\downarrow_{S,n}$ and fixes every
symmetric polynomial.
Since the sums defining the Eulerian functions are symmetric,
they agree in both conventions for every $N$, hence also in
infinitely many variables.
The same holds for $\widetilde Q_n$ and its coefficients $\widetilde{Q}_{n,k}$ by
their definitions above.

\subsection{Positroids and affine permutations}
\label{subsec:positroids-affine-perm}

A rank-$k$ \emph{matroid} $M$ on $[n]$ is a nonempty family
$\mathcal{B}(M)\subseteq\binom{[n]}{k}$ of bases satisfying
the basis exchange axiom: for $B,C\in\mathcal{B}(M)$ and
$i\in B\setminus C$, there is $j\in C\setminus B$ such that
$(B\setminus\{i\})\cup\{j\}\in\mathcal{B}(M)$.
An element is a \emph{loop} if it belongs to no basis,
and a \emph{coloop} if it belongs to every basis.

Let $\grass{k}{n}$ be the Grassmannian of $k$-planes in $\CC^n$.
A full-rank matrix $A=[v_1\ \cdots\ v_n]$ representing
$V\in\grass{k}{n}$ defines the matroid
\[
\mathcal{B}_V=
\left\{I\in\binom{[n]}{k}:\det(A_I)\neq 0\right\},
\]
where $A_I$ consists of the columns indexed by $I$
in increasing order.
The \emph{totally nonnegative Grassmannian}
$\grass{k}{n}_{\geq 0}$ consists of real $k$-planes admitting
a representing matrix with all maximal minors nonnegative.
A \emph{positroid} is a matroid represented by such a plane
\cite[Section~3]{Postnikov2006}.

An \emph{affine permutation} of period $n$ is a bijection
$f:\ZZ\to\ZZ$ satisfying $f(i+n)=f(i)+n$.
It is determined by its window $[f(1),\dots,f(n)]$, which are pairwise distinct modulo $n$ and hence $\frac{1}{n}\sum_{i=1}^n(f(i)-i)\in\ZZ$. Define their \emph{average displacement} and \emph{affine length} by
\[
\operatorname{av}(f)=\frac{1}{n}\sum_{i=1}^n(f(i)-i),
\qquad
\ell(f)=
\#\left\{(i,j)\in[n]\times\ZZ:
i<j,\ f(i)>f(j)\right\}.
\]
Equivalently, $\ell(f)$ counts the inversion pairs $\{(i,j):i<j,f(i)>f(j)\}$ modulo $(i,j)\sim(i+n,j+n)$. An alternative formula for the affine length is
\begin{equation}\label{eq:bg-affine-window-length}
\ell(f)
=\sum_{1\leq i<j\leq n}
\left|\left\lfloor\frac{f(j)-f(i)}{n}\right\rfloor\right|,
\end{equation}
see \cite[Proposition 3.1]{MR1392503}. Write $\widetilde{S}_n^d$ for the affine permutations
of average displacement $d$.
Thus $\widetilde{S}_n^0$ is the affine symmetric group,
and $f\mapsto f+d$ is a length-preserving bijection
$\widetilde{S}_n^0\to\widetilde{S}_n^d$.
The \emph{bounded affine permutations} of type $(k,n)$ form
\[
\bound{k}{n}=
\left\{f\in\widetilde{S}_n^k:
i\leq f(i)\leq i+n\text{ for all }i\in\ZZ\right\}.
\]
They are in bijection with rank-$k$ positroids on $[n]$
\cite[Theorem~3.1]{KnutsonLamSpeyer2013}.

Following \cite[Section~1.2]{KnutsonLamSpeyer2013},
extend the columns of $A$ above periodically by $v_{i+n}=v_i$ and set
\[
f_V(i)=\min\left\{j\geq i:
v_i\in\operatorname{span}(v_{i+1},\dots,v_j)\right\},
\]
with the empty span for $j=i$ equal to $\{0\}$.
This is independent of $A$ and lies in $\bound{k}{n}$.
For $f\in\bound{k}{n}$, define the \emph{open positroid stratum}
and its closure, the \emph{positroid variety}, by
\[
\Pi_f^\circ=\{V\in\grass{k}{n}:f_V=f\},
\qquad
\Pi_f=\overline{\Pi_f^\circ},
\]
where the closure is Zariski closure.

To identify this with the cyclic rank definition in
\cite[Corollary~3.12 and Sections~5.1--5.2]
{KnutsonLamSpeyer2013}, put the cyclic interval rank
$r_{ij}(V)=\dim\operatorname{span}(v_i,\dots,v_j)$ for $i\leq j$.
Then
\[
\begin{aligned}
r_{ij}(V)&=\#\{a\in\ZZ:i\leq a\leq j,\ f_V(a)>j\},\\
f_V(i)&=\min\{j\geq i:r_{ij}(V)=r_{i+1,j}(V)\},
\end{aligned}
\]
with $r_{i+1,i}(V)=0$.
The first identity follows by adding columns from right to left and counting the dimension variation.
Hence fixing $f_V$ is equivalent to fixing all cyclic interval ranks.

These strata are nonempty, smooth, irreducible, and locally closed,
and form a stratification
\[
\grass{k}{n}=\bigsqcup_{f\in\bound{k}{n}}\Pi_f^\circ,
\qquad
\codim{\grass{k}{n}}{\Pi_f}=\ell(f)
\]
by \cite[Theorems~5.9--5.10]{KnutsonLamSpeyer2013}.
The matroid $\mathcal{B}_V$ of $V$ need not be constant on $\Pi_f^\circ$, but its sets of loops and coloops are respecively constant.

Finally, the definition of $f_V$ gives
\[
i\text{ is a loop of }\mathcal{B}_V\iff f_V(i)=i,
\qquad
i\text{ is a coloop of }\mathcal{B}_V\iff f_V(i)=i+n.
\]
Indeed, these conditions mean that $v_i=0$ and that deleting
column $i$ lowers the rank, respectively.
Thus the positroids with neither loops nor coloops correspond to
\[
\bound{k}{n}^{\circ}=
\left\{f\in\bound{k}{n}:
i<f(i)<i+n\text{ for all }i\in\ZZ\right\}.
\]
Since nonzero column rescaling preserves all cyclic interval ranks,
each $\Pi_f^\circ$ is invariant under the coordinatewise action
of $(\CC^\times)^n$.

\subsection{Chow groups and Chow rings}\label{subsec:chow}

For a complex algebraic variety $X$, let $CH_i(X)$
denote the group of $i$-dimensional algebraic cycles modulo rational
equivalence, with integer coefficients. Let $CH_i^\QQ(X)=CH_i(X)\otimes_\ZZ\QQ$.
If $X$ is smooth of pure dimension $d$, set
$CH^r(X)=CH_{d-r}(X)$ and $CH^r_\QQ(X)=CH_{d-r}^\QQ(X)$;
the intersection product makes $CH^*(X)$
a graded ring.

For a closed inclusion $i:Z\hookrightarrow X$ with open complement
$j:U\hookrightarrow X$, the \emph{localization sequence}
\[
CH_r(Z)\xrightarrow{i_*}
CH_r(X)\xrightarrow{j^*}
CH_r(U)\longrightarrow0
\]
is exact \cite[Proposition~1.8]{MR1644323}.
For smooth $X$, restriction is therefore a surjective ring
homomorphism whose kernel consists of classes represented by
cycles supported on $Z$.

The \emph{exterior product}
\[
\bigoplus_{a+b=r}
CH_a(X)\otimes_{\mathbb Z}CH_b(Y)
\longrightarrow CH_r(X\times Y)
\]
is induced by $[V]\times[W]=[V\times W]$.
For proper morphisms $f$ and $g$, it satisfies
\[
(f\times g)_*(\alpha\times\beta)=f_*(\alpha)\times g_*(\beta)
\]
\cite[Proposition~1.10]{MR1644323}.

A \emph{cellular decomposition} is a finite filtration by closed
subvarieties whose successive differences are disjoint unions
of affine spaces.
If $Y$ admits such a decomposition, the exterior product is
surjective for every $X$ \cite[Example~1.10.2]{MR1644323}.
In particular, this applies to $Y=\PP^m$.
For smooth cellular $X$, the cycle class map gives
\[
CH^r(X)\cong H^{2r}(X;\mathbb Z),
\qquad H^{2r+1}(X;\mathbb Z)=0
\]
\cite[Example~19.1.11(b)]{MR1644323}.

All the results hold for rational coefficients by applying the functor $\otimes_\ZZ\QQ$.

\subsection{Technical identities and notations}\label{subsec:polynomial-identity}
\begin{definition}\label{def:res-var-set}
    For each non-emtpy subset $S\subseteq[n]$, let $\QQ[\xxx_S]$ denote the polynomial ring over $\QQ$ in the variable set $\xxx_S\coloneqq\{x_i\,:\,i\in S\}$. We similarly define the restricted version of everything in Definition~\ref{def:R_{n,a,b}}:
    \begin{itemize}
        \item $I_{S,a,b}\coloneqq\langle x_i^a,e_b(S\setminus\{i\})\,:\,i\in S\rangle\subseteq\QQ[\xxx_S]$,
        \item $I_{S,\infty,b}\coloneqq\langle e_b(S\setminus\{i\})\,:\,i\in S\rangle\subseteq\QQ[\xxx_S]$,
        \item $R_{S,a,b}\coloneqq\QQ[\xxx_S]/I_{S,a,b}$, and
        \item $R_{S,\infty,b}\coloneqq\QQ[\xxx_S]/I_{S,\infty,b}$.
    \end{itemize}
\end{definition}

We will repeatedly apply the following technical facts to our subsequent algebraic computation.

\begin{lemma}\label{lem:higher-e}
    Given integers $n>b>0,b^\prime\ge b$, for all $i\in[n]$ we have
    \[e_{b^\prime}([n]\setminus\{i\}),e_{b^\prime}([n]),x_i\cdot e_{b^\prime-1}([n]\setminus\{i\})\in I_{n,\infty,b}.\]
\end{lemma}
\begin{proof}
    We prove it by induction on $b^\prime$. For $b^\prime=b$, by definition we have $e_{b}([n]\setminus\{i\})\in I_{n,\infty,b}$ for all $i\in[n]$, so
    \[e_b([n])=\frac{\sum_{i=1}^n e_b([n]\setminus\{i\})}{n-b}\in I_{n,\infty,b}.\]
    Furthermore,
    \[x_i\cdot e_{b-1}([n]\setminus\{i\})=e_b([n])-e_b([n]\setminus\{i\})\in I_{n,\infty,b}.\]
    Suppose that Lemma~\ref{lem:higher-e} holds for $b^\prime=k$. We want to show it for $b^\prime=k+1$. First we have $e_{k}([n]\setminus\{i\})\in I_{n,\infty,b}$ according to the induction assumption, so $x_i\cdot e_k([n]\setminus\{i\})\in I_{n,\infty,b}$. If $k\ge n$, then $e_{k+1}([n]\setminus\{i\})=e_{k+1}([n])=0\in I_{n,\infty,b}$. If $k<n$, then $e_{k+1}([n])=\frac{\sum_{i=1}^n x_i\cdot e_{k}([n]\setminus\{i\})}{k+1}\in I_{n,\infty,b}$. To sum up, we always have $e_{k+1}([n])\in I_{n,\infty,b}$. Consequently, $e_{k+1}([n]\setminus\{i\})=e_{k+1}([n])-x_i\cdot e_k([n]\setminus\{i\})\in I_{n,\infty,b}$. Therefore, Lemma~\ref{lem:higher-e} holds for $b^\prime=k+1$, finishing the induction.
\end{proof}

\begin{lemma}\label{lem:elementary}
    For any integer $c\ge 0$ and subset $S\subseteq[n]$ with $i_0\in S$, we have
    \begin{align}\label{eq:elementary}e_c(S\setminus\{i_0\}) = \sum_{k=c+1}^{\lvert S \rvert} \frac{(-1)^{k-c-1}}{x_{i_0}^{k-c}}e_{k}(S)\end{align}
    with the convention that the right hand side equals $0$ if $c\ge\lvert S\rvert$.
\end{lemma}
\begin{proof}
    We prove Equation~\eqref{eq:elementary} by reversed induction on $c$. For $c\ge\lvert S\rvert$, both sides of Equation~\eqref{eq:elementary} vanish. Suppose that \eqref{eq:elementary} holds for $c=d$ where $0<d\le\lvert S\rvert$. For $c=d-1$, we have
    \begin{align*}
        e_{d-1}(S\setminus\{i_0\})=\frac{e_d(S)-e_d(S\setminus\{i_0\})}{x_{i_0}}=\frac{e_d(S)-\sum_{k=d+1}^{\lvert S \rvert} \frac{(-1)^{k-d-1}}{x_{i_0}^{k-d}}e_{k}(S)}{x_{i_0}}=\sum_{k=d}^{\lvert S \rvert} \frac{(-1)^{k-d}}{x_{i_0}^{k-d+1}}e_{k}(S)
    \end{align*}
    where the second equality arises from our induction assumption.
\end{proof}
\begin{corollary}\label{cor:elementary-prod-vanish}
    Given integers $c,d\ge 0$ and subset $S\subseteq[n]$ such that $c+d>\lvert S\rvert$, for any distinct integers $i_1,\dots,i_\beta\in[n]\setminus S$ we have
    \begin{align}\label{eq:elementary-prod-vanish}e_c(S)\cdot \prod_{k=1}^{\beta} x_{i_k}^d \in \langle e_{r}(S\sqcup\{i_1,\dots,i_\beta\})\,:\, r\ge c+\beta\rangle\subseteq\QQ[\xxx_n]\end{align}
    with the convention that $\prod_{k=1}^\beta x_{i_k}^d=1$ and $\{i_1,\dots,i_\beta\}=\varnothing$ if $\beta=0$.
\end{corollary}
\begin{proof}
    We prove \eqref{eq:elementary-prod-vanish} by induction on $\beta$. The case $\beta=0$ is clear. Suppose that \eqref{eq:elementary-prod-vanish} holds for $\beta=t_0$. Then, for $\beta=t_0+1$ we have
    \begin{align*}
        e_c(S)\cdot \prod_{k=1}^{t_0+1} x_{i_k}^d = \bigg(e_c(S)\cdot \prod_{k=1}^{t_0} x_{i_k}^d\bigg)\cdot x_{i_{t_0+1}}^d\in \langle e_{r}(S\sqcup\{i_1,\dots,i_{t_0}\})\,:\,r\ge c+t_0\rangle\cdot x_{i_{t_0+1}}^d.
    \end{align*}
    In addition, Lemma~\ref{lem:elementary} yields
    \[e_{r}(S\sqcup\{i_1,\dots,i_{t_0}\}) = \sum_{k=r+1}^{\lvert S \rvert+t_0+1} \frac{(-1)^{k-r-1}}{x_{i_{t_0+1}}^{k-r}}e_{k}(S\sqcup\{i_1,\dots,i_{t_0+1}\}).\]
    It follows that
    \begin{align*}
        e_c(S)\cdot \prod_{k=1}^{t_0+1} x_{i_k}^d &\in\Bigg\langle x_{i_{t_0+1}}^d\cdot\sum_{k=r+1}^{\lvert S \rvert+t_0+1} \frac{(-1)^{k-r-1}}{x_{i_{t_0+1}}^{k-r}}e_{k}(S\sqcup\{i_1,\dots,i_{t_0+1}\})\,:\,r\ge c+t_0\Bigg\rangle \\
        &=\Bigg\langle\sum_{k=r+1}^{\lvert S \rvert+t_0+1} (-1)^{k-r-1}\cdot x_{i_{t_0+1}}^{r+d-k} \cdot e_{k}(S\sqcup\{i_1,\dots,i_{t_0+1}\})\,:\,r\ge c+t_0\Bigg\rangle
    \end{align*}
    where every exponent of $x_{i_{t_0+1}}$ is $r+d-k\ge (c+t_0)+d-(\lvert S\rvert +t_0+1) =c+d-\lvert S\rvert-1\ge 0$. As a result, we have
    \[e_c(S)\cdot \prod_{k=1}^{t_0+1} x_{i_k}^d \in \langle e_k(S\sqcup\{i_1,\dots,i_{t_0+1}\})\,:\,k\ge c+t_0+1\rangle,\]
    so \eqref{eq:elementary-prod-vanish} holds for $\beta=t_0+1$. By induction, the proof is complete.
\end{proof}

The following result in Thibon~\cite[Page 186]{MR1166042} is standard, which gives a formula to the boundary case $\F_{n,\infty,n}=\grFrob(\QQ[\xxx_n];q)$.
\begin{proposition}\label{prop:Thibon-poly}
We have
    \[\grFrob(\QQ[\xxx_n];q)=h_n\left[\frac{\xxx}{1-q}\right].\]
\end{proposition}

\section{A recursive formula for $\F_{n,a,b}$ with $a<\infty$}\label{sec:F_{n,a,b}-recurs}

We will show a recursive formula of $\F_{n,a,b}$ by analyzing the graded $\symm_n$-module structure of $I_{n,a-1,b}/I_{n,a,b}$. We begin with constructing a filtration of $I_{n,a-1,b}/I_{n,a,b}$.

\begin{definition}\label{def:fitration-n,a,b}
    For positive integers $n,a,d$, let $J_{n,a,d}\subseteq\mathbb{Q}[\xxx_n]$ be the ideal given by
    \[J_{n,a,d}\coloneqq\Bigg\langle\prod_{i\in T} x_{i}^{a-1}\,:\,T\in\binom{[n]}{d}\Bigg\rangle.\]
    Let $K_{n,a,b,d}\coloneqq I_{n,a,b}+J_{n,a,d}$.
\end{definition}

\begin{lemma}\label{lem:filtration-n,a,b}
    For positive integers $n,a,b$ with $a+b>n$ and $b<n$, we have $J_{n,a,b}\subseteq I_{n,a,b}$, and hence the filtration
    \begin{align}\label{eq:filtration-n,a,b}
        I_{n,a,b}=K_{n,a,b,b}\subseteq K_{n,a,b,b-1}\subseteq\dots\subseteq K_{n,a,b,1}=I_{n,a-1,b}.
    \end{align}
\end{lemma}
\begin{proof}
    To show $J_{n,a,b}\subseteq I_{n,a,b}$, it suffices to show that for any distinct integers $i_1,\dots,i_b\in[n]$ we have $\prod_{k=1}^b x_{i_k}^{a-1}\in I_{n,a,b}$. Since $b<n$, we fix $i_0\in[n]\setminus\{i_1,\dots,i_b\}$. Take $S=[n]\setminus\{i_0,\dots,i_b\}$ and $c=0,d=a-1$ in Corollary~\ref{cor:elementary-prod-vanish} (note that $c+d=0+a-1=a-1>(n-b)-1=\lvert S\rvert$, so Corollary~\ref{cor:elementary-prod-vanish} works in this case). Then Corollary~\ref{cor:elementary-prod-vanish} implies $\prod_{k=1}^b x_{i_k}^{a-1}\in \langle e_r([n]\setminus\{i_0\})\,:\,r\ge b\rangle\subseteq I_{n,\infty,b}\subseteq I_{n,a,b}$ where the first containment arises from Lemma~\ref{lem:higher-e}, so we have $J_{n,a,b}\subseteq I_{n,a,b}$. It follows that $K_{n,a,b,b}=I_{n,a,b}+J_{n,a,b}=I_{n,a,b}$ and hence the filtration \eqref{eq:filtration-n,a,b} holds.
\end{proof}

Now we describe the relationship between each successive quotient $K_{n,a,b,m}/K_{n,a,b,m+1}$ of the filtration \eqref{eq:filtration-n,a,b} and $R_{n-m,a-1,b-m}$.

\begin{lemma}\label{lem:n,a,b-isom}
    Given positive integers $n,a,b$ with $a+b>n$ and $b<n$, we have the graded $\symm_n$-module isomorphism
    \[(K_{n,a,b,m}/K_{n,a,b,m+1})(m(a-1))\cong \ind_{\symm_{n-m}\times\symm_{m}}^{\symm_n} (R_{n-m,a-1,b-m}\otimes \mathbf{1}_{\symm_m})\]
    for all $1\le m\le b-1$,
    where $M(r)$ denotes the grading shift of the graded module $M$ given by $M(r)_d\coloneqq M_{r+d}$.
\end{lemma}
\begin{proof}
    Note that
    \[\ind_{\symm_{n-m}\times\symm_{m}}^{\symm_n} (R_{n-m,a-1,b-m}\otimes \mathbf{1}_{\symm_m}) \cong \bigoplus_{S\in\binom{[n]}{n-m}} R_{S,a-1,b-m}\]
    as a graded $\symm_n$-module. It suffices to show that the graded $\symm_n$-module homomorphism
    \begin{align*}\Phi_{n,a,b,m}\,:\,\bigoplus_{S\in\binom{[n]}{n-m}} R_{S,a-1,b-m} &\overset{\cong}{\longrightarrow} (K_{n,a,b,m}/K_{n,a,b,m+1})(m(a-1)) \\
    (\overline{f_S})_{S\in\binom{[n]}{n-m}} &\mapsto \overline{\sum_{S\in\binom{[n]}{n-m}} f_S\cdot \prod_{i\in[n]\setminus S} x_i^{a-1}} \end{align*}
    is an isomorphism, where $\bar{f}$ denotes the coset of $f$ in the corresponding quotient.

    \textbf{The well-definedness of $\Phi_{n,a,b,m}$}: For any $S\in\binom{[n]}{n-m}$, we show that $I_{S,a-1,b-m}\cdot \prod_{j\in[n]\setminus S} x_j^{a-1} \subseteq K_{n,a,b,m+1}$. Since $I_{S,a-1,b-m}=\langle x_i^{a-1},e_{b-m}(S\setminus\{i\})\,:\,i\in S\rangle$ and the clear fact that $x_i^{a-1}\cdot \prod_{j\in[n]\setminus S} x_j^{a-1}\in J_{n,a,m+1}\subseteq K_{n,a,b,m+1}$, it remains to show that
    \[e_{b-m}(S\setminus\{i\})\cdot\prod_{j\in[n]\setminus S} x_j^{a-1}\in K_{n,a,b,m+1}.\]
    In fact, since $(b-m)+(a-1)=a+b-m-1>n-m-1=\lvert S\setminus\{i\}\rvert$, Corollary~\ref{cor:elementary-prod-vanish} yields
    \[e_{b-m}(S\setminus\{i\})\cdot\prod_{j\in[n]\setminus S} x_j^{a-1}\in \langle e_r([n]\setminus\{i\})\,:\,r\ge b\rangle,\]
    which, together with Lemma~\ref{lem:higher-e}, implies
    \[e_{b-m}(S\setminus\{i\})\cdot\prod_{j\in[n]\setminus S} x_j^{a-1} \in I_{n,\infty,b}\subseteq I_{n,a,b}\subseteq K_{n,a,b,m+1}.\]
    Therefore, $\Phi_{n,a,b,m}$ is well defined.

    \textbf{The surjectivity of $\Phi_{n,a,b,m}$}: By construction, the $\QQ$-vector space $K_{n,a,b,m}/K_{n,a,b,m+1}$ is spanned by the cosets of the form $\overline{g\cdot\prod_{i \in T} x_i^{a-1}}$ where $T\in\binom{[n]}{m}$ and $g$ is a monomial. If there exists $i_0\in T$ such that $x_{i_0}\mid g$, then $g\cdot\prod_{i \in T} x_i^{a-1}$ is divisible by $x_{i_0}^a\in I_{n,a,b}$ and thus $\overline{g\cdot\prod_{i \in T} x_i^{a-1}}=0$. Otherwise, we have $g\in\QQ[\xxx_{[n]\setminus T}]$, revealing that $\overline{g\cdot\prod_{i \in T} x_i^{a-1}}=\Phi_{n,a,b,m}(\bar{g})$ where $\bar{g}\in R_{[n]\setminus T,a-1,b-m}$. Consequently, the map $\Phi_{n,a,b,m}$ is surjective.

    \textbf{The injectivity of $\Phi_{n,a,b,m}$}: Given $(\overline{f_S})_{S\in\binom{[n]}{n-m}}\in\ker(\Phi_{n,a,b,m})$, it suffices to show that for all $S\in\binom{[n]}{n-m}$ we have $\overline{f_S}=0$. Since $\Phi_{n,a,b,m}$ is a graded map, we may assume that all the $f_S$ are homogeneous of the same degree $D$. Since $x_i^{a-1}\in I_{S,a-1,b-m}$ for every $i\in S$, we may assume that every exponent of each variable in $f_S$ is at most $a-2$.

    \noindent\textbf{Claim}: $$\sum_{S\in\binom{[n]}{n-m}} f_S\cdot\prod_{i\in[n]\setminus S}x_i^{a-1}\in I_{n,\infty,b}+J_{n,a,m+1}$$
    In fact, $\Phi_{n,a,b,m}\bigg((\overline{f_S})_{S    \in\binom{[n]}{n-m}}\bigg)=0$ implies that
    \[\sum_{S\in\binom{[n]}{n-m}} f_S\cdot\prod_{i\in[n]\setminus S}x_i^{a-1}\in K_{n,a,b,m+1}=I_{n,a,b}+J_{n,a,m+1},\]
    so we write
    \begin{align}\label{eq:n,a,b-ideal-expan}\sum_{S\in\binom{[n]}{n-m}} f_S\cdot\prod_{i\in[n]\setminus S}x_i^{a-1} = \sum_{i=1}^n g_i\cdot x_i^a + \sum_{i=1}^n r_i\cdot e_b([n]\setminus\{i\}) + \sum_{T\in\binom{[n]}{m+1}} \gamma_T\cdot\prod_{i\in T}x_i^{a-1} \end{align}
    where $g_i,r_i,\gamma_T\in\QQ[\xxx_n]$ are all homogeneous and $\deg(g_i)=D+m(a-1)-a,\deg(r_i)=D+m(a-1)-b,\deg(\gamma_T)=D+m(a-1)-(m+1)(a-1)=D-a+1$. As any term divisible by some $x_i^a$ can be absorbed into the first summation $\sum_{i=1}^n g_i\cdot x_i^a$, we may assume that:
    \begin{itemize}
        \item every exponent of each variable in both $r_i$ and $\gamma_T$ is at most $a-1$, and
        \item $\gamma_T\in\QQ[\xxx_{[n]\setminus T}]$.
    \end{itemize}
    Consider the direct sum decomposition
    \[\QQ[\xxx_n]=V_a\oplus I_a\]
    where $V_a\coloneqq\spa_\QQ\{\prod_{i=1}^nx_i^{d_i}\,:\,0\le d_i\le a-1\}$ and $I_a\coloneqq\langle x_i^a\,:\, i\in[n]\rangle$, which yields a linear projection $\pi_a\,:\,\QQ[\xxx_n]\twoheadrightarrow V_a$. For any monomial $\prod_{i=1}^n x_i^{d_i}\in V_a$ (i.e., $d_i\le a-1$) and $i_0\in[n]$, we write $A=\{i\in[n]\setminus\{i_0\}\,:\,d_i=a-1\}$. Then we have
    \begin{align*}&\pi_a\Bigg(\bigg(\prod_{i=1}^n x_i^{d_i} \bigg) \cdot e_b([n]\setminus\{i_0\})\Bigg)=\bigg(\prod_{i=1}^n x_i^{d_i} \bigg)\cdot e_b([n]\setminus(\{i_0\}\cup A)) \\
    &= \bigg( \prod_{i\in[n]\setminus A}x_{i}^{d_i} \bigg)\cdot\bigg( \prod_{i\in A}x_i^{a-1} \bigg)\cdot e_b([n]\setminus(\{i_0\}\cup A)) \\
    &\in \bigg( \prod_{i\in[n]\setminus A}x_{i}^{d_i} \bigg)\cdot \langle e_r([n]\setminus\{i_0\}) \,:\, r\ge b+\lvert A \rvert \rangle\subseteq \bigg( \prod_{i\in[n]\setminus A}x_{i}^{d_i} \bigg)\cdot I_{n,\infty,b}\subseteq I_{n,\infty,b}
    \end{align*}
    where the belonging ``$\in$'' arises from Corollary~\ref{cor:elementary-prod-vanish} and the first containment ``$\subseteq$'' arises from Lemma~\ref{lem:higher-e}. In addition, note that every monomial appearing in $r_i$ belongs to $V_a$. Therefore, $\pi_a(r_i\cdot e_b([n]\setminus\{i\}))\in I_{n,\infty,b}$, so applying $\pi_a$ to both sides of Equation~\eqref{eq:n,a,b-ideal-expan} results in
    \begin{align*}
        &\sum_{S\in\binom{[n]}{n-m}} f_S\cdot\prod_{i\in[n]\setminus S}x_i^{a-1} = \pi_a\Bigg( \sum_{S\in\binom{[n]}{n-m}} f_S\cdot\prod_{i\in[n]\setminus S}x_i^{a-1} \Bigg) \\ &= \sum_{i=1}^n \pi_a(g_i\cdot x_i^a) + \sum_{i=1}^n \pi_a(r_i\cdot e_b([n]\setminus\{i\})) + \sum_{T\in\binom{[n]}{m+1}} \pi_a\bigg(\gamma_T\cdot\prod_{i\in T}x_i^{a-1}\bigg) \\
        &= \sum_{i=1}^n \pi_a(r_i\cdot e_b([n]\setminus\{i\})) + \sum_{T\in\binom{[n]}{m+1}} \gamma_T\cdot\prod_{i\in T}x_i^{a-1} \in I_{n,\infty,b} + J_{n,a,m+1},
    \end{align*}
    concluding the proof of the claim above.

    For any $S_0\in\binom{[n]}{n-m}$, let $\mathrm{ev}_{S_0}\,:\,\QQ[\xxx_n] \rightarrow \QQ[\xxx_{S_0}]$ be the evaluation map given by \[x_i\mapsto\begin{cases}
        1, & \text{if $i\notin S_0$} \\
        x_i, & \text{otherwise.}
    \end{cases}\]
    Note that \[e_b([n]\setminus\{i\})=\sum_{k=0}^{m} e_{k}([n]\setminus (\{i\}\cup S_0)) \cdot e_{b-k}(S_0\setminus\{i\})\]
    because $\lvert [n]\setminus (\{i\}\cup S_0) \rvert \le n-(n-m) = m$. It follows that
    \begin{align}\label{eq:evaluate-e}
        \mathrm{ev}_{S_0}(e_b([n]\setminus\{i\}))\in\langle e_r(S_0\setminus\{i\})\,:\, r\ge b-m \rangle\subseteq I_{S_0,\infty,b-m} \subseteq I_{S_0,a-1,b-m}
    \end{align}
    where the first containment arises from Lemma~\ref{lem:higher-e}. Furthermore, since $\lvert [n]\setminus S_0 \rvert =m<m+1$, we have
    \begin{align}\label{eq:evaluate-x}
        \mathrm{ev}_{S_0}(J_{n,a,m+1})\subseteq \langle x_i^{a-1}\,:\,i\in S_0 \rangle \subseteq I_{S_0,a-1,b-m}.
    \end{align}
    Combining the claim above and the facts \eqref{eq:evaluate-e} and \eqref{eq:evaluate-x}, we obtain
    \begin{align}\label{eq:ev-all}
        \mathrm{ev}_{S_0} \Bigg( \sum_{S\in\binom{[n]}{n-m}} f_S\cdot\prod_{i\in[n]\setminus S}x_i^{a-1} \Bigg) \in I_{S_0,a-1,b-m}.
    \end{align}
    
    Moreover, since every exponent of each variable in $f_S$ is at most $a-2$, the degree of each monomial appearing in $\mathrm{ev}_{S_0}\Bigg(f_S\cdot\prod_{i\in[n]\setminus S} x_i^{a-1}\Bigg)$ is strictly greater than $D$ if $S\neq S_0$.
    Therefore, the degree-$D$ homogeneous component of
    \[\mathrm{ev}_{S_0} \Bigg( \sum_{S\in\binom{[n]}{n-m}} f_S\cdot\prod_{i\in[n]\setminus S}x_i^{a-1} \Bigg) \]
    is $\mathrm{ev}_{S_0}(f_{S_0}\cdot\prod_{i\in[n]\setminus S_0}x_i^{a-1}) = f_{S_0}$. Consequently, the fact that $I_{S_0,a-1,b-m}$ is a homogeneous ideal and the fact \eqref{eq:ev-all} implies that the degree-$D$ homogeneous component of
    \[\mathrm{ev}_{S_0} \Bigg( \sum_{S\in\binom{[n]}{n-m}} f_S\cdot\prod_{i\in[n]\setminus S}x_i^{a-1} \Bigg) \]
    belongs to $I_{S_0,a-1,b-m}$, i.e., $f_{S_0}\in I_{S_0,a-1,b-m}$. The arbitrariness of $S_0\in\binom{[n]}{n-m}$ implies that $(\overline{f_S})_{S\in\binom{[n]}{n-m}} = 0$. In conclusion, we have shown that $\ker(\Phi_{n,a,b,m})=0$ and hence $\Phi_{n,a,b,m}$ is injective.
\end{proof}
 
We eventually obtain a recursive formula of $\F_{n,a,b}$:

\begin{proposition}\label{prop:recursive-n,a,b}
    For integers $n,a,b>0$ with $a+b>n$ and $b<n$, we have
    \[\F_{n,a,b} = \sum_{m=0}^{b-1} q^{m(a-1)} \cdot h_m \cdot \F_{n-m,a-1,b-m}. \]
\end{proposition}
\begin{proof}
    Note that
    \begin{align*}
        &\F_{n,a,b} - \F_{n,a-1,b} = \grFrob(R_{n,a,b};q) - \grFrob(R_{n,a-1,b};q) = \grFrob(I_{n,a-1,b}/I_{n,a,b};q) \\
        &= \sum_{m=1}^{b-1} \grFrob(K_{n,a,b,m}/K_{n,a,b,m+1};q)
    \end{align*}
    by Lemma~\ref{lem:filtration-n,a,b}. Furthermore, Lemma~\ref{lem:n,a,b-isom} implies that
    \[\grFrob(K_{n,a,b,m}/K_{n,a,b,m+1};q)= q^{m(a-1)} \cdot h_m \cdot \grFrob(R_{n-m,a-1,b-m};q) = q^{m(a-1)} \cdot h_m \cdot \F_{n-m,a-1,b-m}\]
    for $1\le m\le b-1$.
    Consequently, we deduce that
    \[\F_{n,a,b}-\F_{n,a-1,b} = \sum_{m=1}^{b-1} q^{m(a-1)} \cdot h_m \cdot \F_{n-m,a-1,b-m}\]
    and thus
    \[\F_{n,a,b} = \sum_{m=0}^{b-1} q^{m(a-1)} \cdot h_m \cdot \F_{n-m,a-1,b-m}.\qedhere\]
\end{proof}

\section{A recursive formula for $\F_{n,\infty,b}$}\label{sec:F_{n,infty,b}-recurs}

Our strategy is similar to Section~\ref{sec:F_{n,a,b}-recurs}. We begin with constructing a filtration of $I_{n,\infty,b-1}/I_{n,\infty,b}$ (note that Lemma~\ref{lem:higher-e} reveals that $I_{n,\infty,b}\subseteq I_{n,\infty,b-1}$).

\begin{definition}\label{def:n,infty,b-filtration}
    For positive integers $n,b,d$ such that $1<b<n$, let $L_{n,b,d}\subseteq\QQ[\xxx_n]$ be the ideal given by
    \[L_{n,b,d}\coloneqq \Bigg\langle \prod_{i\in T} e_{b-1}([n]\setminus\{i\})\,:\,T\in\binom{[n]}{d}\Bigg\rangle\]
    and define $N_{n,b,d}=I_{n,\infty,b} + L_{n,b,d}$.
\end{definition}

\begin{lemma}\label{lem:n,infty,b-filtration}
    For integers $1<b<n$, we have $L_{n,b,n-b+2}\subseteq I_{n,\infty,b}$ and hence the filtration
    \[I_{n,\infty,b}=N_{n,b,n-b+2}\subseteq N_{n,b,n-b+1}\subseteq\dots\subseteq N_{n,b,1} = I_{n,\infty,b-1}.\]
\end{lemma}
\begin{proof}
    Lemma~\ref{lem:higher-e} implies that \begin{align}\label{eq:special-vanish}x_i\cdot e_{b-1}([n]\setminus\{i\})\equiv 0 \mod{I_{n,\infty,b}}.\end{align} This fact can help us remove variables as follows: for any $T\in\binom{[n]}{n-b+2}$ and a fixed $i_0\in T$ we have
    \begin{align*}
        \prod_{i\in T}e_{b-1}([n]\setminus\{i\}) &= e_{b-1}([n]\setminus\{i_0\}) \cdot \prod_{i\in T\setminus\{i_0\}} e_{b-1}([n]\setminus\{i\}) \\
        &\equiv e_{b-1}([n]\setminus(\{i_0\}\cup T\setminus\{i_0\})) \cdot \prod_{i\in T\setminus\{i_0\}} e_{b-1}([n]\setminus\{i\}) \mod{I_{n,\infty,b}} \\
        &= e_{b-1}([n]\setminus T) \cdot \prod_{i\in T\setminus\{i_0\}} e_{b-1}([n]\setminus\{i\}) = 0
    \end{align*}
    where the congruence ``$\equiv$'' arises from applying \eqref{eq:special-vanish} for all $i\in T\setminus\{i_0\}$, and the last equality is because the fact that $\lvert [n]\setminus T \rvert = n-(n-b+2)=b-2<b-1$ means $e_{b-1}([n]\setminus T)=0$. Consequently, we have that $L_{n,b,n-b+2}\subseteq I_{n,\infty,b}$ and hence $N_{n,b,n-b+2}=I_{n,\infty,b}+L_{n,b,n-b+2}=I_{n,\infty,b}$. Moreover, Lemma~\ref{lem:higher-e} yields that $e_{b}([n]\setminus\{i\})\in I_{n,\infty,b-1}$ and thus $I_{n,\infty,b}\subseteq I_{n,\infty,b-1}=L_{n,b,1}$, which implies that $N_{n,b,1}=I_{n,\infty,b}+L_{n,b,1}=L_{n,b,1}=I_{n,\infty,b-1}$. Therefore, we have the filtration
    \[I_{n,\infty,b}=N_{n,b,n-b+2}\subseteq N_{n,b,n-b+1}\subseteq\dots\subseteq N_{n,b,1} = I_{n,\infty,b-1}.\qedhere\]
\end{proof}

Now we describe the relationship between each successive quotient $N_{n,b,m}/N_{n,b,m+1}$ of the filtration in Lemma~\ref{lem:n,infty,b-filtration} and $R_{n-m,\infty,b-1}$. We will eventually see that the following surjection is actually an isomorphism when proving Proposition~\ref{prop:recursive-n,infty,b}.

\begin{lemma}\label{lem:n,infty,b-surj}
    Given integers $1<b<n$, for any $1\le m\le n-b+1$ we have an graded $\symm_n$-module surjection
    \[\ind_{\symm_{n-m}\times\symm_m}^{\symm_n}( R_{n-m,\infty,b-1}\otimes\mathbf{1}_{\symm_m})\twoheadrightarrow (N_{n,b,m}/N_{n,b,m+1})(m(b-1))\]
    where $M(r)$ denotes the grading shift of the graded module $M$ given by $M(r)_d\coloneqq M_{r+d}$.
\end{lemma}
\begin{proof}
    Note that
    \[\ind_{\symm_{n-m}\times\symm_m}^{\symm_n}( R_{n-m,\infty,b-1}\otimes\mathbf{1}_{\symm_m})\cong\bigoplus_{S\in\binom{[n]}{n-m}} R_{S,\infty,b-1}\]
    as a graded $\symm_n$-module.
    Consider the graded $\symm_n$-module homomorphism
    \begin{align*}
        \Psi_{n,b,m}\,:\, \bigoplus_{S\in\binom{[n]}{n-m}} R_{S,\infty,b-1} & \longrightarrow (N_{n,b,m}/N_{n,b,m+1})(m(b-1)) \\
        (\overline{f_S})_{S\in\binom{[n]}{n-m}} & \mapsto \overline{\sum_{S\in\binom{[n]}{n-m}} f_S \cdot \prod_{i\in[n]\setminus S} e_{b-1}([n]\setminus\{i\})}.
    \end{align*}
    It suffices to show that $\Psi_{n,b,m}$ is well defined and surjective.

    \textbf{The well-definedness of $\Psi_{n,b,m}$}: For any generator $e_{b-1}(S\setminus\{i_0\})$ (where $i_0\in S$ and $\lvert S\rvert =n-m$) of $I_{S,\infty,b-1}$, we have
    \begin{align*}
        &e_{b-1}(S\setminus\{i_0\}) \cdot \prod_{i\in[n]\setminus S} e_{b-1}([n]\setminus\{i\}) \\ &\equiv e_{b-1}((S\setminus\{i_0\}) \cup ([n]\setminus S)) \cdot \prod_{i\in[n]\setminus S} e_{b-1}([n]\setminus\{i\}) \mod{I_{n,\infty,b}} \\
        &= e_{b-1}([n]\setminus\{i_0\}) \cdot \prod_{i\in[n]\setminus S} e_{b-1}([n]\setminus\{i\}) \in L_{n,b,m+1}
    \end{align*}
    where the congruence ``$\equiv$'' arises from applying \eqref{eq:special-vanish} for all $i\in[n]\setminus S$. Therefore, for any $S\in\binom{[n]}{n-m}$ we have
    \[I_{S,\infty,b-1} \cdot \prod_{i\in[n]\setminus S} e_{b-1}([n]\setminus\{i\}) \subseteq I_{n,\infty,b}+ L_{n,b,m+1}=N_{n,b,m+1},\]
    so $\Psi_{n,b,m}$ is well defined.

    \textbf{The surjectivity of $\Psi_{n,b,m}$}: By construction, $N_{n,b,m}/N_{n,b,m+1}$ is spanned by all the cosets of the form $\overline{f\cdot \prod_{i\in T} e_{b-1}([n]\setminus\{i\})}$ as a $\QQ$-vector space, where $f\in\QQ[\xxx_n]$ and $T\in\binom{[n]}{m}$. Let $\mathrm{ev}^T\,:\,\QQ[\xxx_n]\rightarrow\QQ[\xxx_{[n]\setminus T}]$ be the evaluation map given by
    \[x_i\mapsto\begin{cases}
        0, & \text{if $i\in T$} \\
        x_i, & \text{otherwise}.
    \end{cases}\]
    Applying \eqref{eq:special-vanish} for all $i\in T$, we obtain
    \begin{align*}
        f\cdot \prod_{i\in T} e_{b-1}([n]\setminus\{i\}) \equiv \mathrm{ev}^T(f) \cdot \prod_{i\in T} e_{b-1}([n]\setminus\{i\}) \mod{I_{n,\infty,b}}.
    \end{align*}
    It follows that
    \begin{align*}
        \overline{f\cdot \prod_{i\in T} e_{b-1}([n]\setminus\{i\})} = \overline{\mathrm{ev}^T(f) \cdot \prod_{i\in T} e_{b-1}([n]\setminus\{i\})} = \Psi_{n,b,m}(\overline{\mathrm{ev}^T(f)})
    \end{align*}
    where $\overline{\mathrm{ev}^T(f)}\in R_{[n]\setminus T,\infty,b-1}$. Consequently, $\Psi_{n,b,m}$ is surjective.
\end{proof}

\begin{corollary}\label{cor:n,infty,b-ineq}
    For integers $1<b<n$, we have
    \[\F_{n,\infty,b}\le \sum_{m=0}^{n-b+1} q^{m(b-1)} \cdot h_m \cdot \F_{n-m,\infty,b-1} .\]
\end{corollary}
\begin{proof}
    Note that
    \begin{align*}
        &\F_{n,\infty,b} - \F_{n,\infty,b-1} = \grFrob(R_{n,\infty,b};q)-\grFrob(R_{n,\infty,b-1};q) \\
        &= \grFrob(I_{n,\infty,b-1}/I_{n,\infty,b};q) = \sum_{m=1}^{n-b+1} \grFrob(N_{n,b,m}/N_{n,b,m+1};q) \\
        &\le \sum_{m=1}^{n-b+1} q^{m(b-1)} \cdot h_m \cdot \grFrob(R_{n-m,\infty,b-1};q) = \sum_{m=1}^{n-b+1} q^{m(b-1)} \cdot h_m \cdot \F_{n-m,\infty,b-1}
    \end{align*}
    where the second-to-last equality arises from Lemma~\ref{lem:n,infty,b-filtration} and the inequality arises from Lemma~\ref{lem:n,infty,b-surj}. Therefore, we have
    \[\F_{n,\infty,b}\le \sum_{m=0}^{n-b+1} q^{m(b-1)} \cdot h_m \cdot \F_{n-m,\infty,b-1} . \qedhere\]
\end{proof}

Our goal is to force the inequality in Corollary~\ref{cor:n,infty,b-ineq} to be an equality. Our strategy is the dimension-counting for the boundary case $b=n-1$. 

\begin{lemma}\label{lem:boundary-case-dim}
    For any integer $n>1$, we have
    \[\Hilb(R_{n,\infty,n-1};q) = \sum_{k=0}^{n-2} \binom{n}{k}\cdot \bigg(\frac{q}{1-q}\bigg)^k.\]
\end{lemma}
\begin{proof}
    Since $I_{n,\infty,n-1}=\langle \prod_{j\in[n]\setminus\{i\}} x_j \,:\, i\in[n] \rangle$ is generated by monomials, the Gröbner basis of $I_{n,\infty,n-1}$ with respect to any monomial order just consists of all these monomials. Therefore, the standard monomial basis $\mathcal{B}$ of $R_{n,\infty,n-1}$ consists of all the monomials not divisible by $\prod_{j\in[n]\setminus\{i\}} x_j$ for any $i\in[n]$. Equivalently, $\mathcal{B}$ consists of all the monomials where at most $n-2$ distinct variables appear. For any $S\subseteq[n]$, let $\mathcal{B}_S\coloneqq\{\mmm\in\QQ[\xxx_S]\,:\, \text{$\mmm$ is a monomial divisible by $x_i$ for all $i\in S$}\}$. We have the generating function identity
    \[\sum_{\mmm\in\mathcal{B}_S}q^{\deg(\mmm)} = \prod_{i\in S}\bigg(\sum_{d=1}^\infty q^{d}\bigg) = \bigg( \frac{q}{1-q} \bigg)^{\lvert S \rvert}.\]
    It follows that
    \begin{align*}
        \Hilb(R_{n,\infty,n-1};q) = \sum_{\mmm\in\mathcal{B}} q^{\deg(\mmm)} =\sum_{\substack{S\subseteq[n] \\ \lvert S \rvert \le n-2}} \sum_{\mmm\in\mathcal{B}_S} q^{\deg(\mmm)} =\sum_{k=0}^{n-2} \binom{n}{k} \cdot \bigg( \frac{q}{1-q} \bigg)^k,
    \end{align*}
    concluding our proof.
\end{proof}

\begin{proposition}\label{prop:recursive-n,infty,b}
    The inequality in Corollary~\ref{cor:n,infty,b-ineq} is actually an equality, i.e., for $1<b<n$ we have
    \[\F_{n,\infty,b} = \sum_{m=0}^{n-b+1} q^{m(b-1)} \cdot h_m \cdot \F_{n-m,\infty,b-1} .\]
\end{proposition}
\begin{proof}
    Let $H_{n,\infty,b}\coloneqq\Hilb(R_{n,\infty,b};q)$. Corollary~\ref{cor:n,infty,b-ineq} gives
    \begin{align}\label{ineq:hilb}
        H_{n,\infty,b}\le \sum_{m=0}^{n-b+1} q^{m(b-1)} \cdot \binom{n}{m} \cdot H_{n-m,\infty,b-1}
    \end{align}
    for all $1<b<n$.
    It suffices to show that \eqref{ineq:hilb} is actually an equality. Our strategy is to iteratively apply \eqref{ineq:hilb} so that we can obtain an upper bound of $H_{n,\infty,n-1}$ using the base cases $H_{\beta,\infty,\beta}$ and $H_{\gamma,\infty,1}$ with $\gamma>1,\beta>0$, and then show that this upper bound equals the expression in Lemma~\ref{lem:boundary-case-dim}, which would force \eqref{ineq:hilb} to be an equality because every possible $(n,b)$ with $1<b<n$ occurs in this iterating process.

    We start with computing base cases. Given $n>1$, Lemma~\ref{lem:higher-e} implies that for all $i\in[n]$ we have $x_i\in I_{n,\infty,1}$. It follows that $R_{n,\infty,1}\cong\mathbb{Q}$ as a graded $\symm_n$-module, and thus for all $n>1$ we have
    \begin{align}\label{eq:boundary1}
        H_{n,\infty,1} = \Hilb(\QQ;q)=1.
    \end{align}
    For any $n\ge 1$, we have $I_{n,\infty,n}=0$ and hence $R_{n,\infty,n}=\QQ[\xxx_n]$, meaning that
    \begin{align}\label{eq:boundary2}
        H_{n,\infty,n}=\Hilb(\QQ[\xxx_n];q)=\Hilb(\QQ[x]^{\otimes n};q) = \Hilb(\QQ[x];q)^n=\bigg(\frac{1}{1-q}\bigg)^n.
    \end{align}
    \textbf{Note that Equation~\eqref{eq:boundary1} is valid only for $n>1$. For $H_{1,\infty,1}$, we should apply Equation~\eqref{eq:boundary2} rather than \eqref{eq:boundary1}.}

    Think of each $H_{n,\infty,b}$ as the point $(n,b)$ where $1\le b\le n$. As we iterate \eqref{ineq:hilb}, the value of $n+b$ is strictly decreasing, implying that the iterating process must eventually stop in finitely many steps. Since we can apply \eqref{ineq:hilb} whenever $1<b<n$, this iterating process stops exactly when either $(n,b)=(\beta,\beta)$ for some $\beta>0$ or $(n,b)=(\gamma,1)$ for some $\gamma>1$ (i.e., the two base cases mentioned in Equations~\eqref{eq:boundary1} and \eqref{eq:boundary2}). We deal with these two kinds of iterating paths as follows:
    \begin{itemize}
        \item If an iterating path starts from $(n,n-1)$ ($n>2$) and ends at $(\beta,\beta)$ ($\beta>0$), then this iterating path can be written by
        \[(n,n-1),(n-m_{n-2},n-2),\dots,(n-\sum_{i=j}^{n-2}m_i,j),\dots,(n-\sum_{i=\beta}^{n-2}m_i,\beta)=(\beta,\beta)\]
        where for all $\beta<j\le n-2$ we have $n-\sum_{i=j}^{n-2}m_i>j$ (otherwise, \eqref{ineq:hilb} is not valid for $n^\prime=n-\sum_{i=j}^{n-2}m_i,b^\prime=j$, which is a contradiction).
        \item If an iterating path starts from $(n,n-1)$ ($n>2$) and ends at $(\gamma,1)$ ($\gamma>1$), then this iterating path can be written by
        \[(n,n-1),(n-m_{n-2},n-2),\dots,(n-\sum_{i=j}^{n-2}m_i,j),\dots,(n-\sum_{i=1}^{n-2}m_i,1)=(\gamma,1)\]
        where for all $1<j\le n-2$ we have $n-\sum_{i=j}^{n-2}m_i>j$ (otherwise, \eqref{ineq:hilb} is not valid for $n^\prime=n-\sum_{i=j}^{n-2}m_i,b^\prime=j$, which is a contradiction). Note that for $j=1$, this inequality $n-\sum_{i=j}^{n-2}m_i>j$ automatically holds because $\gamma>1$.
    \end{itemize}
    The two possible types of iteration paths are illustrated in Figure~\ref{fig:hilb-iteration-paths}.
    
\begin{figure}[htbp]
\centering

\begin{minipage}{0.48\textwidth}
\centering
\begin{tikzpicture}[
    x=0.62cm, y=0.62cm,
    traj/.style={
        draw=orange!85!black,
        line width=1.5pt,
        -{Stealth[length=2mm]}
    },
    waypoint/.style={
        circle, draw=orange!85!black,
        fill=white, line width=1pt,
        inner sep=1.7pt
    },
    endpoint/.style={
        circle, fill=red!85!black,
        inner sep=2.7pt
    }
]
    \fill[cyan!8]
        (1,1) -- (9.5,1) -- (9.5,9.5) -- cycle;

    \foreach \k in {2,...,9} {
        \draw[gray!22, thin]
            (\k,1) -- (\k,\k);
        \draw[gray!22, thin]
            (\k,\k) -- (9.5,\k);
    }

    \draw[-{Stealth}, gray!80]
        (0,0) -- (10.4,0) node[right] {$n'$};
    \draw[-{Stealth}, gray!80]
        (0,0) -- (0,10.2) node[above] {$b'$};

    \draw[blue!75!black, very thick]
        (1,1) -- (9.5,9.5)
        node[pos=0.72, above left] {$n'=b'$};

    \draw[violet!80!black, very thick]
        (1,1) -- (9.5,1)
        node[pos=0.80, below] {$b'=1$};

    \draw[traj] (9,8) -- (9,7);
    \draw[traj] (9,7) -- (8,6);
    \draw[traj] (8,6) -- (7,5);
    \draw[traj] (7,5) -- (4,4);

    \foreach \x/\y in {9/8,9/7,8/6,7/5} {
        \node[waypoint] at (\x,\y) {};
    }

    \node[endpoint] at (4,4) {};

    \node[anchor=west, orange!85!black]
        at (9.12,8) {$(n,n-1)$};

    \node[anchor=south east, red!85!black]
        at (3.85,4.08) {$(\beta,\beta)$};

\end{tikzpicture}

\smallskip
\textbf{(a)} Ending on $n'=b'$.
\end{minipage}
\hfill
\begin{minipage}{0.48\textwidth}
\centering
\begin{tikzpicture}[
    x=0.62cm, y=0.62cm,
    traj/.style={
        draw=orange!85!black,
        line width=1.5pt,
        -{Stealth[length=2mm]}
    },
    waypoint/.style={
        circle, draw=orange!85!black,
        fill=white, line width=1pt,
        inner sep=1.7pt
    },
    endpoint/.style={
        circle, fill=red!85!black,
        inner sep=2.7pt
    }
]
    \fill[cyan!8]
        (1,1) -- (9.5,1) -- (9.5,9.5) -- cycle;

    \foreach \k in {2,...,9} {
        \draw[gray!22, thin]
            (\k,1) -- (\k,\k);
        \draw[gray!22, thin]
            (\k,\k) -- (9.5,\k);
    }

    \draw[-{Stealth}, gray!80]
        (0,0) -- (10.4,0) node[right] {$n'$};
    \draw[-{Stealth}, gray!80]
        (0,0) -- (0,10.2) node[above] {$b'$};

    \draw[blue!75!black, very thick]
        (1,1) -- (9.5,9.5)
        node[pos=0.72, above left] {$n'=b'$};

    \draw[violet!80!black, very thick]
        (1,1) -- (9.5,1)
        node[pos=0.80, below] {$b'=1$};

    \draw[traj] (9,8) -- (9,7);
    \draw[traj] (9,7) -- (8,6);
    \draw[traj] (8,6) -- (7,5);
    \draw[traj] (7,5) -- (7,4);
    \draw[traj] (7,4) -- (6,3);
    \draw[traj] (6,3) -- (5,2);
    \draw[traj] (5,2) -- (4,1);

    \foreach \x/\y in {
        9/8,9/7,8/6,7/5,7/4,6/3,5/2
    } {
        \node[waypoint] at (\x,\y) {};
    }

    \node[endpoint] at (4,1) {};

    \node[anchor=west, orange!85!black]
        at (9.12,8) {$(n,n-1)$};

    \node[anchor=south east, red!85!black]
        at (3.85,1.18) {$(\gamma,1)$};

\end{tikzpicture}

\smallskip
\textbf{(b)} Ending on $b'=1$.
\end{minipage}

\caption{The two types of iteration paths arising from
\eqref{ineq:hilb}. Each arrow represents an application
of \eqref{ineq:hilb}. The iteration terminates upon reaching
either the diagonal boundary $n'=b'$ or the horizontal
boundary $b'=1$.}
\label{fig:hilb-iteration-paths}
\end{figure}
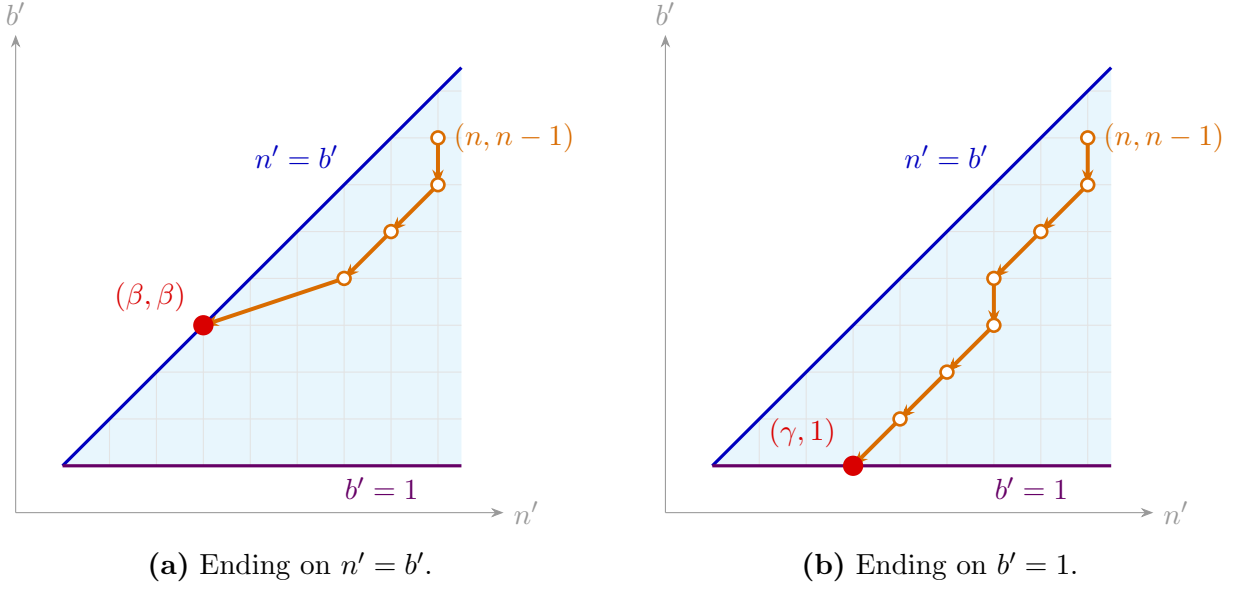

    Therefore, summing over these two kinds of iterating paths, for $n>2$ we have
    \begin{align}\label{ineq:dyck-path-iterate}
        H_{n,\infty,n-1} \le& \sum_{\beta=1}^{n-2} \sum_{\substack{m_{\beta},\dots,m_{n-2}\ge 0 \\ \sum_{i=\beta}^{n-2} m_i = n-\beta \\ \text{ $\sum_{i=j}^{n-2} m_i < n-j$ for all $\beta<j\le n-2$ } }} \Bigg(\prod_{i=\beta}^{n-2} \Big(q^{im_i} \cdot \binom{n-\sum_{j=i+1}^{n-2}m_j}{m_i} \Big) \Bigg) \cdot H_{\beta,\infty,\beta}  \\
        \nonumber&+ \sum_{\gamma=2}^n \sum_{\substack{m_1,\dots,m_{n-2}\ge 0 \\ \sum_{i=1}^{n-2} m_i = n-\gamma \\ \text{$\sum_{i=j}^{n-2} m_i < n-j$ for all $1\le j\le n-2$} }} \Bigg(\prod_{i=1}^{n-2} \Big( q^{im_i} \cdot \binom{n-\sum_{j=i+1}^{n-2}m_j}{m_i} \Big) \Bigg) \cdot H_{\gamma,\infty,1} \\
        \nonumber = & \sum_{\beta=1}^{n-2} \sum_{\substack{m_{\beta},\dots,m_{n-2}\ge 0 \\ \sum_{i=\beta}^{n-2} m_i = n-\beta \\ \text{ $\sum_{i=j}^{n-2} m_i < n-j$ for all $\beta<j\le n-2$ } }} \Bigg(\prod_{i=\beta}^{n-2} \Big(q^{im_i} \cdot \binom{n-\sum_{j=i+1}^{n-2}m_j}{m_i} \Big) \Bigg) \cdot \Big(\frac{1}{1-q}\Big)^\beta \\
        \nonumber&+ \sum_{\gamma=2}^n \sum_{\substack{m_1,\dots,m_{n-2}\ge 0 \\ \sum_{i=1}^{n-2} m_i = n-\gamma \\ \text{$\sum_{i=j}^{n-2} m_i < n-j$ for all $1\le j\le n-2$} }} \Bigg(\prod_{i=1}^{n-2} \Big( q^{im_i} \cdot \binom{n-\sum_{j=i+1}^{n-2}m_j}{m_i} \Big) \Bigg) \cdot 1 \\
        \nonumber = & \sum_{\beta=1}^{n-2} \sum_{\substack{m_{\beta},\dots,m_{n-2}\ge 0 \\ \sum_{i=\beta}^{n-2} m_i = n-\beta \\ \text{ $\sum_{i=j}^{n-2} m_i < n-j$ for all $\beta<j\le n-2$ } }}\!\!\!\!\!\!\!\!\!\!\!\!\!\!\!\!\!\!\!\!\!\! \Bigg( q^{\sum_{i=\beta}^{n-2}im_i} \cdot \binom{n}{m_\beta,\dots,m_{n-2}, n - \sum_{i=\beta}^{n-2} m_i } \Bigg) \cdot \Big(\frac{1}{1-q}\Big)^\beta \\
        \nonumber&+ \sum_{\gamma=2}^n \sum_{\substack{m_1,\dots,m_{n-2}\ge 0 \\ \sum_{i=1}^{n-2} m_i = n-\gamma \\ \text{$\sum_{i=j}^{n-2} m_i < n-j$ for all $1\le j\le n-2$} }} \Bigg(q^{\sum_{i=1}^{n-2}im_i} \cdot \binom{n}{m_1,\dots,m_{n-2}, n - \sum_{i=1}^{n-2} m_i } \Bigg) \\
        \nonumber = & \sum_{\beta=1}^{n-2} \sum_{\substack{m_{\beta},\dots,m_{n-2}\ge 0 \\ \sum_{i=\beta}^{n-2} m_i = n-\beta \\ \text{ $\sum_{i=j}^{n-2} m_i < n-j$ for all $\beta<j\le n-2$ } }} \Bigg( q^{\sum_{i=\beta}^{n-2}im_i} \cdot \binom{n}{m_\beta,\dots,m_{n-2}, \beta } \Bigg) \cdot \Big(\frac{1}{1-q}\Big)^\beta \\
        \nonumber&+ \sum_{\gamma=2}^n \sum_{\substack{m_1,\dots,m_{n-2}\ge 0 \\ \sum_{i=1}^{n-2} m_i = n-\gamma \\ \text{$\sum_{i=j}^{n-2} m_i < n-j$ for all $1\le j\le n-2$} }} \Bigg(q^{\sum_{i=1}^{n-2}im_i} \cdot \binom{n}{m_1,\dots,m_{n-2}, \gamma } \Bigg)
    \end{align}
    where the first identity follows from Equations~\eqref{eq:boundary1} and \eqref{eq:boundary2}.

    We deal with the two outer sums on the last two rows of \eqref{ineq:dyck-path-iterate} via the two claims~\eqref{eq:path1} and \eqref{eq:path2} below, respectively.
    
    \noindent\textbf{Claim:}
    \begin{align}
        \label{eq:path1}
        &\sum_{\beta=1}^{n-2} \sum_{\substack{m_{\beta},\dots,m_{n-2}\ge 0 \\ \sum_{i=\beta}^{n-2} m_i = n-\beta \\ \text{ $\sum_{i=j}^{n-2} m_i < n-j$ for all $\beta<j\le n-2$ } }} \Bigg( q^{\sum_{i=\beta}^{n-2}im_i} \cdot \binom{n}{m_\beta,\dots,m_{n-2}, \beta } \Bigg) \cdot \Big(\frac{1}{1-q}\Big)^\beta \\
        \nonumber&= \sum_{\substack{c_1,c_2,\dots\ge 0 \\ c_0\ge 2 \\ \sum_{i=0}^\infty c_i=n \\ \text{$\sum_{i=0}^j c_i \le j+1 $ for some $1\le j\le n-2$} }} \binom{n}{c_0,c_1,\dots} \cdot q^{\sum_{i=0}^\infty ic_i}.
    \end{align}
    In fact, since
    \[\bigg(\frac{1}{1-q}\bigg)^\beta = (1+q+q^2+\dots)^\beta = \sum_{\substack{d_0,d_1,\dots\ge0 \\ \sum_{i=0}^\infty d_i=\beta}} \binom{\beta}{d_0,d_1,\dots} \cdot q^{\sum_{i=0}^\infty id_i},\]
    the left-hand-side of Equation~\eqref{eq:path1} equals
    \begin{align}\label{eq:LHS-hilb}
        &\sum_{\beta=1}^{n-2} \sum_{\substack{m_{\beta},\dots,m_{n-2}\ge 0 \\ \sum_{i=\beta}^{n-2} m_i = n-\beta \\ \text{ $\sum_{i=j}^{n-2} m_i < n-j$ for all $\beta<j\le n-2$ } }} \!\!\!\!\!\! q^{\sum_{i=\beta}^{n-2}im_i} \cdot \binom{n}{m_\beta,\dots,m_{n-2}, \beta }  \cdot \sum_{\substack{d_0,d_1,\dots\ge0 \\ \sum_{i=0}^\infty d_i=\beta}} \binom{\beta}{d_0,d_1,\dots} \cdot q^{\sum_{i=0}^\infty id_i} \\
        \nonumber&= \sum_{\beta=1}^{n-2} \sum_{\substack{m_{\beta},\dots,m_{n-2}\ge 0 \\ d_0,d_1,\dots\ge 0 \\ \sum_{i=0}^\infty d_i =\beta \\ \sum_{i=\beta}^{n-2} m_i = n-\beta \\ \text{ $\sum_{i=j}^{n-2} m_i < n-j$ for all $\beta<j\le n-2$ } }} \!\!\!\!\!\! q^{\sum_{i=\beta}^{n-2}im_i} \cdot \binom{n}{m_\beta,\dots,m_{n-2}, \beta }  \cdot \binom{\beta}{d_0,d_1,\dots} \cdot q^{\sum_{i=0}^\infty id_i} \\
        \nonumber&= \sum_{\beta=1}^{n-2} \sum_{\substack{m_{\beta},\dots,m_{n-2}\ge 0 \\ d_0,d_1,\dots\ge 0 \\ \sum_{i=0}^\infty d_i =\beta \\ \sum_{i=\beta}^{n-2} m_i = n-\beta \\ \text{ $\sum_{i=j}^{n-2} m_i < n-j$ for all $\beta<j\le n-2$ } }} \binom{n}{m_\beta,\dots,m_{n-2},d_0,d_1,\dots} \cdot q^{\sum_{i=\beta}^{n-2} im_i + \sum_{i=0}^\infty id_i} \\
        \nonumber&= \sum_{\beta=1}^{n-2} \sum_{\substack{m_{\beta},\dots,m_{n-2}\ge 0 \\ d_0,d_1,\dots\ge 0 \\ \sum_{i=0}^\infty d_i =\beta \\ \sum_{i=\beta}^{n-2} m_i = n-\beta \\ \text{ $\sum_{i=j}^{n-2} m_i < n-j$ for all $\beta<j\le n-2$ } }} \binom{n}{m_\beta,\dots,m_{n-2},d_0,d_1,\dots} \cdot q^{\sum_{i=\beta}^{n-2} (i-\beta)m_i + \sum_{i=0}^\infty (i+n-\beta)d_i} \\
        \nonumber&= \sum_{\beta=1}^{n-2} \sum_{\substack{c_0,\dots,c_{n-2-\beta}\ge 0 \\ c_{n-\beta},c_{n-\beta+1},\dots\ge 0 \\ \sum_{i=n-\beta}^\infty c_i =\beta \\ \sum_{i=0}^{n-2-\beta} c_i = n-\beta \\ \text{ $\sum_{i=j}^{n-2-\beta} c_i < n-j-\beta$ for all $0<j\le n-2-\beta$ } }} \!\!\!\!\!\!\!\!\!\!\!\!\!\!\!\!\!\!\binom{n}{c_0,\dots,c_{n-2-\beta},c_{n-\beta},c_{n-\beta+1},\dots} \cdot q^{\sum_{i=0}^{n-2-\beta} ic_i + \sum_{i=n-\beta}^\infty ic_i} \\
        \nonumber&= \sum_{\beta=1}^{n-2} \sum_{\substack{c_0,c_1,\cdots\ge 0 \\ c_{n-\beta-1} = 0 \\ \sum_{i=n-\beta}^\infty c_i =\beta \\ \sum_{i=0}^{n-2-\beta} c_i = n-\beta \\ \text{ $\sum_{i=j}^{n-2-\beta} c_i < n-j-\beta$ for all $0<j\le n-2-\beta$ } }} \binom{n}{c_0,c_1,\dots} \cdot q^{\sum_{i=0}^{\infty} ic_i}
    \end{align}
    where the third-to-last equality follows from the conditions $\sum_{i=0}^\infty d_i =\beta$ and $\sum_{i=\beta}^{n-2}m_i = n-\beta$ under the summation (which imply that $\sum_{i=\beta}^{n-2} im_i + \sum_{i=0}^\infty id_i = \sum_{i=\beta}^{n-2}(i-\beta)m_i + \sum_{i=0}^\infty (i+n-\beta)d_i + \beta \sum_{i=\beta}^{n-2}m_i + (\beta-n)\sum_{i=0}^\infty d_i = \sum_{i=\beta}^{n-2}(i-\beta)m_i + \sum_{i=0}^\infty (i+n-\beta)d_i + \beta (n-\beta) + (\beta-n)\beta = \sum_{i=\beta}^{n-2}(i-\beta)m_i + \sum_{i=0}^\infty (i+n-\beta)d_i$), the second-to-last equality is just the changes of index $m_i\mapsto c_{i-\beta}$ ($i=\beta,\dots,n-2$) and $d_i\mapsto c_{i+n-\beta}$ ($i=0,1,\dots$), and the last equality follows from inserting a new index $c_{n-\beta-1}=0$.

    For $1\le \beta\le n-2$, let $A_\beta$ be the set of integer sequences $(c_0,c_1,\dots)$ occurring in the inner summation on the last row of Equation~\eqref{eq:LHS-hilb}, and define $B_\beta$ by
    \[B_\beta\coloneqq\left\{(c_0,c_1,\dots)\in(\ZZ_{\ge 0})^\infty\,:\,\begin{aligned}
        &c_0\ge 2,\quad
        \sum_{i=0}^\infty c_i =n, \\
        &\min\{j\in[n-2]\,:\,\sum_{i=0}^j c_i \le j+1\} = n-1-\beta
    \end{aligned}\right\}.\]
    We show that $A_\beta=B_\beta$ as follows.

    Proof of $A_\beta\subseteq B_\beta$: For any $(c_0,c_1,\dots)\in A_\beta$, if $\beta=n-2$ then we have $c_0 =\sum_{i=0}^{n-2-\beta} c_i = n-\beta =2$; if $\beta<n-2$ then we have $c_0 = \sum_{i=0}^{n-2-\beta}c_i - \sum_{i=1}^{n-2-\beta}c_i = n-\beta - \sum_{i=1}^{n-2-\beta}c_i > n-\beta-(n-1-\beta) = 1$ and hence $c_0\ge 2$. Furthermore, $\sum_{i=0}^\infty c_i = \sum_{i=0}^{n-2-\beta}c_i + c_{n-1-\beta} + \sum_{i=n-\beta}^\infty c_i = (n-\beta) + 0 + \beta = n$. Finally, for any $1\le j<n-1-\beta$ we have $\sum_{i=0}^j c_i = \sum_{i=0}^{n-2-\beta} c_i - \sum_{i=j+1}^{n-2-\beta}c_i = (n-\beta)-\sum_{i=j+1}^{n-2-\beta}c_i > (n-\beta) - (n-(j+1)-\beta) = j+1$, but for $j=n-1-\beta$ we have $\sum_{i=0}^j c_i = \sum_{i=0}^{n-2-\beta}c_i + c_{n-1-\beta} = n-\beta + 0 = j + 1$. Therefore, $(c_0,c_1,\dots)\in B_\beta$. The arbitrariness of $(c_0,c_1,\dots)$ yields $A_\beta\subseteq B_\beta$.

    Proof of $B_\beta\subseteq A_\beta$: For any $(c_0,c_1,\dots)\in B_\beta$, the identity $\min\{j\in[n-2]\,:\,\sum_{i=0}^j c_i \le j+1\} = n-1-\beta$ implies that $\sum_{i=0}^{n-2-\beta}c_i> n-1-\beta$ and $\sum_{i=0}^{n-1-\beta}c_i \le n-\beta$, which result in the inequality $c_{n-1-\beta}<1$ and thus $c_{n-1-\beta}=0$. It follows that $n-1-\beta < \sum_{i=0}^{n-2-\beta}c_i \le n-\beta $, so $\sum_{i=0}^{n-2-\beta}c_i =  n-\beta$. Then the identity $\sum_{i=0}^\infty c_i = n$ means $\sum_{i=n-\beta}^\infty c_i = \sum_{i=0}^\infty c_i - \sum_{i=0}^{n-2-\beta}c_i - c_{n-1-\beta} = n - (n-\beta) - 0 = \beta$. Moreover, for all $0<j\le n-2-\beta$ we have $\sum_{i=j}^{n-2-\beta}c_i = \sum_{i=0}^{n-2-\beta}c_i - \sum_{i=0}^{j-1}c_i = (n-\beta) - \sum_{i=0}^{j-1}c_i < (n-\beta)-j = n-j-\beta$. Consequently, we have $(c_0,c_1,\dots)\in A_\beta$. The arbitrariness of $(c_0,c_1,\dots)$ means $B_\beta\subseteq A_\beta$.

    Now we have show that $A_\beta=B_\beta$, so we apply this fact to the last row of Equation~\eqref{eq:LHS-hilb}, obtaining that the left hand side of Equation~\eqref{eq:path1} equals
    \begin{align*}
        \sum_{\beta=1}^{n-2} \sum_{(c_0,c_1,\dots)\in A_\beta} \binom{n}{c_0,c_1,\dots} \cdot q^{\sum_{i=0}^\infty ic_i} = \sum_{\beta=1}^{n-2} \sum_{(c_0,c_1,\dots)\in B_\beta} \binom{n}{c_0,c_1,\dots} \cdot q^{\sum_{i=0}^\infty ic_i}
    \end{align*}
    which is equal to the right-hand-side of Equation~\eqref{eq:path1}. Therefore, Equation~\eqref{eq:path1} in the claim above is true.

    \noindent\textbf{Claim:}
    \begin{align}\label{eq:path2}
        &\sum_{\gamma=2}^n \sum_{\substack{m_1,\dots,m_{n-2}\ge 0 \\ \sum_{i=1}^{n-2} m_i = n-\gamma \\ \text{$\sum_{i=j}^{n-2} m_i < n-j$ for all $1\le j\le n-2$} }} \Bigg(q^{\sum_{i=1}^{n-2}im_i} \cdot \binom{n}{m_1,\dots,m_{n-2}, \gamma } \Bigg) \\
        \nonumber&= \sum_{\substack{c_1,c_2,\dots\ge 0 \\ c_0\ge 2 \\ \sum_{i=0}^\infty c_i=n \\ \text{$\sum_{i=0}^j c_i > j+1 $ for all $1\le j\le n-2$} }} \binom{n}{c_0,c_1,\dots} \cdot q^{\sum_{i=0}^\infty ic_i}.
    \end{align}
    In fact, for the left-hand-side of Equation~\eqref{eq:path2} we can replace $m_i$ with $c_i$ ($i=1,\dots,n-2$) and let $c_0=\gamma$. Then Equation~\eqref{eq:path2} is clear.

    Substituting Equations~\eqref{eq:path1} and \eqref{eq:path2} into the last two rows of Equation~\eqref{ineq:dyck-path-iterate} respectively, we obtain
    \begin{align*}
        H_{n,\infty,n-1} \le& \sum_{\substack{c_1,c_2,\dots\ge 0 \\ c_0\ge 2 \\ \sum_{i=0}^\infty c_i=n \\ \text{$\sum_{i=0}^j c_i \le j+1 $ for some $1\le j\le n-2$} }} \binom{n}{c_0,c_1,\dots} \cdot q^{\sum_{i=0}^\infty ic_i} \\
        &+ \sum_{\substack{c_1,c_2,\dots\ge 0 \\ c_0\ge 2 \\ \sum_{i=0}^\infty c_i=n \\ \text{$\sum_{i=0}^j c_i > j+1 $ for all $1\le j\le n-2$} }} \binom{n}{c_0,c_1,\dots} \cdot q^{\sum_{i=0}^\infty ic_i} \\
        =& \sum_{\substack{c_1,c_2,\dots\ge 0 \\ c_0\ge 2 \\ \sum_{i=0}^\infty c_i=n}} \binom{n}{c_0,c_1,\dots} \cdot q^{\sum_{i=0}^\infty ic_i} = \sum_{k=0}^{n-2} \sum_{\substack{c_1,c_2,\dots\ge 0 \\ \sum_{i=1}^\infty c_i = k \\ c_0 = n - k}} \binom{n}{c_0,c_1,\dots} \cdot q^{\sum_{i=0}^\infty ic_i} \\
        =& \sum_{k=0}^{n-2} \sum_{\substack{c_1,c_2,\dots\ge 0 \\ \sum_{i=1}^\infty c_i = k}} \binom{n}{n-k,c_1,c_2,\dots} \cdot q^{\sum_{i=1}^\infty ic_i} = \sum_{k=0}^{n-2} \sum_{\substack{c_1,c_2,\dots\ge 0 \\ \sum_{i=1}^\infty c_i = k}} \binom{n}{k} \cdot \binom{k}{c_1,c_2,\dots} \cdot q^{\sum_{i=1}^\infty ic_i} \\
        =& \sum_{k=0}^{n-2}\binom{n}{k} \cdot \sum_{\substack{c_1,c_2,\dots\ge 0 \\ \sum_{i=1}^\infty c_i = k}} \binom{k}{c_1,c_2,\dots} \cdot q^{\sum_{i=1}^\infty ic_i} = \sum_{k=0}^{n-2}\binom{n}{k} \cdot (q+q^2+q^3+\dots)^k \\
        =& \sum_{k=0}^{n-2} \binom{n}{k} \cdot \bigg(\frac{q}{1-q} \bigg)^k.
    \end{align*}
    Therefore, Lemma~\ref{lem:boundary-case-dim} implies that this inequality is actually an equality, forcing all the inequalities occurring in the iterating process before \eqref{ineq:dyck-path-iterate} to be equalities. Consequently, for all $1<b<n$ the inequality~\eqref{ineq:hilb} is actually an equality, implying that the inequality in Corollary~\ref{cor:n,infty,b-ineq} is actually an equality and finishing our proof.
\end{proof}

\section{Closed formulae for $\F_{n,a,b}$}\label{sec:F_{n,a,b}-closed}

In this section, we give a closed formula for the Frobenius character $\F_{n,a,b}$ with $a+b\ge n$. We first show that the two recursive formulas given by Propositions~\ref{prop:recursive-n,a,b} and \ref{prop:recursive-n,infty,b} and stabilization limits $\F_{n,\infty,1},\F_{n,\infty,n}$ uniquely determine $\F_{n,a,b}$, then construct an explicit quasisymmetric (actually symmetric) function $G_{n,a,b}$ satisfying the same recursive formulas and stabilization limits, and hence conclude that $\F_{n,a,b}=G_{n,a,b}$. The uniqueness is easy (see Lemma~\ref{lem:unique-recursive}), while analyzing $G_{n,a,b}$ is involved with lots of combinatorial details that we suggest to skip on a first read.

\begin{definition}\label{def:stable-limit}
    For a family of quasisymmetric functions $QF_{n,a,b}\in \qsym_n \otimes_\QQ \QQ[[q]]$, we write
    \[\lim_{a\rightarrow\infty} QF_{n,a,b} = QF_{n,\infty,b}\]
    if there exists a quasisymmetric function $QF_{n,\infty,b}\in \qsym_n \otimes_\QQ \QQ[[q]]$ such that for each $d\ge 0$ there exists $N=N(n,b,d)\ge 0$ such that for all $a>N$ we have $\langle q^d \rangle QF_{n,a,b}= \langle q^d \rangle QF_{n,\infty,b}$. Here $\langle q^d \rangle QF\in\qsym_n$ denotes the coefficient of $q^d$ in $QF\in\qsym_n\otimes_\QQ\QQ[[q]]$.
\end{definition}

\begin{lemma}\label{lem:unique-recursive}
    If a family of quasisymmetric functions $QF_{n,a,b}\in \qsym_n \otimes_\QQ \QQ[[q]]$ ($n,a,b>0,a+b\ge n,b\le n$) satisfies that:
    \begin{itemize}
    \item For $0<b \le n$, we have that
    \[\lim_{a\rightarrow\infty} QF_{n,a,b} = QF_{n,\infty,b}\]
    exists,
        \item For any $n,a,b>0$ such that $a+b>n,b<n$ (which implies $a-1>0$), we have
        \[QF_{n,a,b} = \sum_{m=0}^{b-1} q^{m(a-1)} \cdot h_m \cdot QF_{n-m,a-1,b-m},\]
        \item For any $1<b<n$, we have
        \[QF_{n,\infty,b} = \sum_{m=0}^{n-b+1} q^{m(b-1)} \cdot h_m \cdot QF_{n-m,\infty,b-1},\]
        \item For any $n>0$, we have
        \[QF_{n,\infty,n} = h_n\left[\frac{\xxx}{1-q}\right],\] and
        \item For any $n>1$, we have
        \[QF_{n,\infty,1}=h_n,\]
    \end{itemize}
    then for all $n,a,b>0$ with $a+b\ge n,b<n$ we have $QF_{n,a,b} = \F_{n,a,b}$.
\end{lemma}
\begin{proof}
    Note that for all $n,a,b>0$ with $a+b\ge n,b<n$ we have
    \begin{align}\label{eq:inverse}
        &\sum_{m=0}^{b-1} (-1)^m \cdot q^{ma} \cdot e_m \cdot QF_{n-m,a+1,b-m} = \sum_{m=0}^{b-1} (-1)^m \cdot q^{ma} \cdot e_m \cdot \sum_{r=0}^{b-m-1} q^{ra} \cdot h_r \cdot QF_{n-m-r,a,b-m-r} \\
        &= \sum_{m=0}^{b-1} \sum_{r=0}^{b-m-1} (-1)^m \cdot q^{(m+r)a} \cdot e_m \cdot h_r \cdot QF_{n-m-r,a,b-m-r} \notag\\
        &= \sum_{\substack{m,r\ge 0 \\ m+r\le b-1}} (-1)^m \cdot q^{(m+r)a} \cdot e_m \cdot h_r \cdot QF_{n-m-r,a,b-m-r} \notag\\
        &= \sum_{k=0}^{b-1} \sum_{\substack{m,r\ge 0 \\ m+r=k}} (-1)^m \cdot q^{(m+r)a} \cdot e_m \cdot h_r \cdot QF_{n-m-r,a,b-m-r} \notag\\
        &= \sum_{k=0}^{b-1} \sum_{\substack{m,r\ge 0 \\ m+r=k}} (-1)^m \cdot q^{ka} \cdot e_m \cdot h_r \cdot QF_{n-k,a,b-k} = \sum_{k=0}^{b-1} \Bigg(\sum_{\substack{m,r\ge 0 \\ m+r=k}} (-1)^m \cdot  e_m \cdot h_r\Bigg) \cdot q^{ka} \cdot QF_{n-k,a,b-k} \notag\\
        &= \sum_{k=0}^{b-1} \delta_{k,0} \cdot q^{ka} \cdot QF_{n-k,a,b-k} = QF_{n,a,b}.\notag
    \end{align}
     For each $n,a,b>0$ with $a+b\ge n$ and $b<n$, this formula computes $QF_{n,a,b}$ using $QF_{n',a',b'}$ with $0<n'\le n,0<b'\le b,a'\ge a+1$ such that $b'<n'$. Therefore, for $d\ge 0$ taking $
M(n,b,d)=
\max_{\substack{0<n'\le n\\0<b'\le b,\ b'<n'\\0\le d'\le d}}
N(n',b',d').
$ with $N(n',b',d')$ defined in Definition~\ref{def:stable-limit} and iterating this formula~\eqref{eq:inverse} for $M(n,b,d)$ times give $\langle q^d \rangle QF_{n,a,b}$ using $QF_{n',\infty,b'}$ with $0<n'\le n,0<b'\le b$ such that $b'<n'$. As a result, the family of quasisymmetric functions $\{QF_{n,a,b}\}_{n,a,b>0,a+b\ge n,b<n}$ is uniquely determined by the family of quasisymmetric functions $\{QF_{n,\infty,b}\}_{0<b< n}$. In addition, the recursive formula for $QF_{n,\infty,b}$ and the boundary conditions for $QF_{n,\infty,b}$ with $b=1$ or $b=n$ above also uniquely determine $QF_{n,\infty,b}$ with $0<b<n$. Consequently, all the properties above together uniquely determine $QF_{n,a,b}$ ($n,a,b>0,a+b\ge n,b<n$). Since the same properties hold for $\F_{n,a,b}$ as well (see Propositions~\ref{prop:Thibon-poly}, \ref{prop:recursive-n,a,b} and \ref{prop:recursive-n,infty,b}), we deduce that $\F_{n,a,b} = QF_{n,a,b}$ for all $n,a,b>0$ with $a+b\ge n$ and $b<n$.
\end{proof}

We will see that $\F_{n,a,b} = G_{n,a,b}$ where $G_{n,a,b}$ is defined below. (See Theorem~\ref{thm:closed-formula-n,a,b}.)
\begin{definition}\label{def:G}
    For all $n,a,b>0$ such that $b<n$ and $a+b\ge n$, define
    \[G_{n,a,b}\coloneqq \sum_{r=0}^{b-1}h_r\Bigg[\Bigg(\sum_{k=n-b}^{a-1}q^k\Bigg)\cdot \xxx\Bigg] \cdot \sum_{\substack{f\in\widetilde{S}_{n-r}^{0} \\ \text{$i-n+b<f(i)<i+b-r$ for all $i$}}} q^{\ell(f)} \cdot F_{\DES{f\mid_{[n-r]}},n-r}.\]
    In particular, we use the conventions that:
    \begin{itemize}
        \item for $0<b\le n$ we let $G_{n,0,b} \coloneqq 0$, and
        \item for $n,a>0$ we let $G_{n,a,n} \coloneqq h_n[(\sum_{k=0}^{a-1} q^k) \cdot \xxx]$.
    \end{itemize}
\end{definition}

\begin{lemma}\label{lem:trivial}
    For $n>1,a>0$ with $a+1\ge n$, we have
    \[G_{n,a,1} = h_n.\]
\end{lemma}
\begin{proof}
    By Definition~\ref{def:G}, we have
    \begin{align*}
    &G_{n,a,1} = \sum_{\substack{f\in\widetilde{S}_n^0 \\ \text{$i-n+1 < f(i) < i+1$ for all $i$}}} q^{\ell(f)} \cdot F_{\DES{f\mid_{[n]}},n} \\
    &= \sum_{\substack{f\in\widetilde{S}_n^0 \\ \text{$i-n+1 < f(i) \le i$ for all $i$}}} q^{\ell(f)} \cdot F_{\DES{f\mid_{[n]}},n} = q^0 \cdot F_{\varnothing,n} = h_n
    \end{align*}
    where the second-to-last equality follows from the fact that $f(i) = i$ ($i\in\ZZ$) is the unique affine permutation satisfying all the conditions under the summation notation. In fact, the condition $f(i)\le i$ implies that $\frac{1}{n}\cdot\sum_{i=1}^n (f(i) - i) \le 0$ where the equality holds if and only if $f(i) = i$ for all $i\in[n]$ (which means that $f(i)=i$ for all $i\in\ZZ$ by periodicity).
\end{proof}

We show that $G_{n,a,b}$ satisfies the same recursive formula as $\F_{n,a,b}$.
\begin{lemma}\label{lem:not-dual-recurrsive}
    For $n,a,b>0$ with $b<n$ and $a+b>n$, we have
    \[G_{n,a,b} = \sum_{m=0}^{b-1} q^{m(a-1)} \cdot h_m \cdot G_{n-m,a-1,b-m}.\]
\end{lemma}
\begin{proof}
    We begin with computing the right-hand-side:
    \begin{align*}
        &\sum_{m=0}^{b-1} q^{m(a-1)} \cdot h_m \cdot G_{n-m,a-1,b-m} = \\
        &\sum_{m=0}^{b-1} q^{m(a-1)} \cdot h_m \cdot \sum_{r=0}^{b-m-1} h_r \Bigg[\Bigg(\sum_{k=n-b}^{a-2} q^k \Bigg) \cdot \xxx\Bigg] \cdot\!\!\!\!\!\!\!\!\!\!\!\!\!\!\!\!\!\!\!\! \sum_{\substack{f\in\widetilde{S}_{n-m-r}^0 \\ \text{$i-n+b < f(i) < i+b-m-r$ for all $i$}}}\!\!\!\!\!\!\!\!\!\!\!\! q^{\ell(f)} \cdot F_{\DES{f\mid_{[n-m-r]}},n-m-r}\\
        & = \sum_{\substack{m,r\ge 0 \\ m+r\le b-1}} q^{m(a-1)} \cdot h_m \cdot h_r\Bigg[\Bigg(\sum_{k=n-b}^{a-2} q^k \Bigg) \cdot \xxx\Bigg]\cdot\!\!\!\!\!\!\!\!\!\!\!\!\!\!\!\!\!\!\!\! \sum_{\substack{f\in\widetilde{S}_{n-m-r}^0 \\ \text{$i-n+b < f(i) < i+b-m-r$ for all $i$}}}\!\!\!\!\!\!\!\!\!\!\!\! q^{\ell(f)} \cdot F_{\DES{f\mid_{[n-m-r]}},n-m-r} \\
        & = \sum_{\substack{m,r\ge 0 \\ m+r\le b-1}} h_m[q^{a-1}\cdot \xxx] \cdot h_r\Bigg[\Bigg(\sum_{k=n-b}^{a-2} q^k \Bigg) \cdot \xxx\Bigg]\cdot\!\!\!\!\!\!\!\!\!\!\!\!\!\!\!\!\!\!\!\! \sum_{\substack{f\in\widetilde{S}_{n-m-r}^0 \\ \text{$i-n+b < f(i) < i+b-m-r$ for all $i$}}}\!\!\!\!\!\!\!\!\!\!\!\! q^{\ell(f)} \cdot F_{\DES{f\mid_{[n-m-r]}},n-m-r} \\
        & = \sum_{d=0}^{b-1} \sum_{\substack{m,r\ge 0 \\ m+r=d}} h_m[q^{a-1}\cdot \xxx] \cdot h_r\Bigg[\Bigg(\sum_{k=n-b}^{a-2} q^k \Bigg) \cdot \xxx\Bigg]\cdot\!\!\!\!\!\!\!\!\!\!\!\!\!\!\!\!\!\!\!\! \!\!\sum_{\substack{f\in\widetilde{S}_{n-m-r}^0 \\ \text{$i-n+b < f(i) < i+b-m-r$ for all $i$}}}\!\!\!\!\!\!\!\!\!\!\!\! q^{\ell(f)} \cdot F_{\DES{f\mid_{[n-m-r]}},n-m-r} \\
        & = \sum_{d=0}^{b-1} \sum_{\substack{m,r\ge 0 \\ m+r=d}} h_m[q^{a-1}\cdot \xxx] \cdot h_r\Bigg[\Bigg(\sum_{k=n-b}^{a-2} q^k \Bigg) \cdot \xxx\Bigg]\cdot\!\!\!\!\!\!\!\!\!\!\!\!\!\!\!\!\!\!\!\! \!\!\sum_{\substack{f\in\widetilde{S}_{n-d}^0 \\ \text{$i-n+b < f(i) < i+b-d$ for all $i$}}}\!\!\!\!\!\!\!\!\!\!\!\! q^{\ell(f)} \cdot F_{\DES{f\mid_{[n-d]}},n-d} \\
        & = \sum_{d=0}^{b-1} h_d\Bigg[\Bigg(\sum_{k=n-b}^{a-1} q^k \Bigg) \cdot \xxx\Bigg]\cdot\!\!\!\!\!\!\!\!\!\!\!\!\!\!\!\!\!\!\!\! \!\!\sum_{\substack{f\in\widetilde{S}_{n-d}^0 \\ \text{$i-n+b < f(i) < i+b-d$ for all $i$}}}\!\!\!\!\!\!\!\!\!\!\!\! q^{\ell(f)} \cdot F_{\DES{f\mid_{[n-d]}},n-d} = G_{n,a,b},
    \end{align*}
    which concludes our proof.
\end{proof}

We state some technical definitions necessary for the proof of the following results.
\begin{definition}\label{def:fix-std}
    For any injection $w\,:\,[n]\rightarrow \ZZ$ (we also regard it as a word in the alphabet $\ZZ$ of length $n$ without repeated letters), define
    \begin{align*}
        \FIX{w} &\coloneqq \{i\in[n] \,:\, w(i) = i\}, \\
        \fix{w} &\coloneqq \lvert \FIX{w} \rvert.
    \end{align*}
    Let $\std{w}\in\symm_n$ be the unique permutation such that for all $1\le i<j\le n$ we have that $\std{w}(i)<\std{w}(j)$ if and only if $w(i)<w(j)$.
\end{definition}

\begin{lemma}\label{lem:affine-fixed-point-insertion}
    Let $n,K>0$, let $\varepsilon\in\{-1,1\}$, and let
    $\xi\in\{0,1\}$.  Suppose that
    \[
        K\le n+\xi-1.
    \]
    Then
    \begin{align}
        &\sum_{r=0}^{n-K+\xi-1}q^{rK}h_r
        \sum_{\substack{
            u\in\widetilde S_{n-r}^{\varepsilon K}\\
            \text{$0<\varepsilon(u(i)-i)<n-r+\xi$ for all $i$}
        }}
        q^{\ell(u)}
        F_{\DES{u\mid_{[n-r]}},n-r}
        \notag\\
        &\qquad\qquad=
        \sum_{\substack{
            w\in\widetilde S_n^{\varepsilon K}\\
            \text{$0\le\varepsilon(w(i)-i)<n+\xi$ for all $i$}
        }}
        q^{\ell(w)}
        F_{\DES{w\mid_{[n]}},n}.
        \label{eq:affine-fixed-point-insertion}
    \end{align}
\end{lemma}

\begin{proof}
    We first stratify the summation range on right-hand side of \eqref{eq:affine-fixed-point-insertion} according to $\fix{w\mid_{[n]}}$, and we will see that this statistic corresponds to the index $r$ on the left-hand side. With this goal, we fix an integer $r\ge0$, put
    \[
        N\coloneqq n-r,
    \] and construct a bijection
    \begin{align}
        &\binom{[n]}r
        \times
        \left\{
            u\in\widetilde S_N^{\varepsilon K}:
            \text{$0<\varepsilon(u(i)-i)<N+\xi$ for all $i$}
        \right\}
        \notag\\[-1mm]
        &\qquad\longleftrightarrow
        \left\{
        \begin{array}{c}
            w\in\widetilde S_n^{\varepsilon K}:
            \text{$0\le\varepsilon(w(i)-i)<n+\xi$ for all $i$},\\
            \fix{w\mid_{[n]}}=r
        \end{array}
        \right\}.
        \label{bij:oriented-fixed-point-insertion}
    \end{align}
    
    For $P\in\binom{[n]}r$, let
    \[
        P^{\mathrm{per}}\coloneqq P+n\ZZ.
    \]
    Write
    \[
        [n]\setminus P
        =
        \{\gamma_1<\gamma_2<\cdots<\gamma_N\},
    \]
    and define the strictly increasing bijection
    \[
        \rho_P:
        \ZZ\setminus P^{\mathrm{per}}
        \longrightarrow
        \ZZ
    \]
    by
    \[
        \rho_P(\gamma_j+\nu n)
        \coloneqq
        j+\nu N
        \qquad
        (j\in[N],\ \nu\in\ZZ).
    \]
    In particular,
    \begin{equation}\label{eq:oriented-compression-periodicity}
        \rho_P(i+n)=\rho_P(i)+N
        \qquad
        (i\notin P^{\mathrm{per}}).
    \end{equation}

    Let
    \[
        u\in\widetilde S_N^{\varepsilon K}
    \]
    satisfy
    \[
        0<\varepsilon(u(i)-i)<N+\xi
        \qquad(i\in\ZZ).
    \]
    Define its insertion along $P$ by
    \begin{equation}\label{eq:oriented-insertion-map}
        \iota_P(u)(i)
        \coloneqq
        \begin{cases}
            i,
            &i\in P^{\mathrm{per}},\\[1mm]
            \rho_P^{-1}\bigl(u(\rho_P(i))\bigr),
            &i\notin P^{\mathrm{per}}.
        \end{cases}
    \end{equation}

    Equation~\eqref{eq:oriented-compression-periodicity} and the
    periodicity of $u$ imply
    \[
        \iota_P(u)(i+n)=\iota_P(u)(i)+n.
    \]
    Indeed, if $i\notin P^{\mathrm{per}}$, then
    \begin{align*}
        \iota_P(u)(i+n)
        &=
        \rho_P^{-1}\bigl(u(\rho_P(i+n))\bigr)\\
        &=
        \rho_P^{-1}\bigl(u(\rho_P(i)+N)\bigr)\\
        &=
        \rho_P^{-1}\bigl(u(\rho_P(i))+N\bigr)\\
        &=
        \rho_P^{-1}\bigl(u(\rho_P(i))\bigr)+n\\
        &=
        \iota_P(u)(i)+n.
    \end{align*}
    The equality is immediate when $i\in P^{\mathrm{per}}$.

    The map $\iota_P(u)$ is the identity on $P^{\mathrm{per}}$.
    On $\ZZ\setminus P^{\mathrm{per}}$, it is conjugate to the
    bijection $u$ via $\rho_P$.  Hence $\iota_P(u)$ is a bijection
    of $\ZZ$.

    We next check the bounds.  Let $i\notin P^{\mathrm{per}}$ and put
    $j=\rho_P(i)$.

    If $\varepsilon=1$, then
    \[
        j<u(j)<j+N+\xi.
    \]
    When $\xi=0$, monotonicity and \eqref{eq:oriented-compression-periodicity} give
    \[
        i<\iota_P(u)(i)<i+n.
    \]
    When $\xi=1$, integrality gives $u(j)\le j+N$, and hence (monotonicity and \eqref{eq:oriented-compression-periodicity} yield)
    \[
        i<\iota_P(u)(i)\le i+n<i+n+1.
    \]
    Thus in both cases
    \[
        0<\iota_P(u)(i)-i<n+\xi.
    \]

    If $\varepsilon=-1$, then
    \[
        j-N-\xi<u(j)<j.
    \]
    When $\xi=0$, monotonicity and \eqref{eq:oriented-compression-periodicity} give
    \[
        i-n<\iota_P(u)(i)<i.
    \]
    When $\xi=1$, integrality gives $u(j)\ge j-N$, and hence (monotonicity and \eqref{eq:oriented-compression-periodicity} yield)
    \[
        i-n-1<i-n\le\iota_P(u)(i)<i.
    \]
    Thus in both cases
    \[
        0<i-\iota_P(u)(i)<n+\xi.
    \]

    Combining the two possibilities for $\varepsilon$, we obtain
    \[
        0<
        \varepsilon\bigl(\iota_P(u)(i)-i\bigr)
        <n+\xi
        \qquad(i\notin P^{\mathrm{per}}).
    \]
    Together with the first line of
    \eqref{eq:oriented-insertion-map}, this gives
    \begin{equation}\label{eq:oriented-inserted-bounds}
        0\le
        \varepsilon\bigl(\iota_P(u)(i)-i\bigr)
        <n+\xi
        \qquad(i\in\ZZ),
    \end{equation}
    and also
    \begin{equation}\label{eq:oriented-inserted-fixed-set}
        \FIX{\iota_P(u)\mid_{[n]}}=P.
    \end{equation}

    We now verify the average displacement.  Since $u$ is an affine
    permutation of period $N$, there exist
    $\theta\in\symm_N$ and integers
    $\lambda_1,\ldots,\lambda_N$ such that
    \[
        u(j)=\theta(j)+\lambda_jN
        \qquad(j\in[N]).
    \]
    Because $u\in\widetilde S_N^{\varepsilon K}$ and
    $\sum_{j=1}^N\theta(j)=\sum_{j=1}^Nj$, we have
    \begin{equation}\label{eq:oriented-type-lambda}
        \varepsilon K
        =\operatorname{av}(u)=
        \frac1N\sum_{j=1}^N\bigl(u(j)-j\bigr)
        =
        \sum_{j=1}^N\lambda_j.
    \end{equation}
    By the definition of $\rho_P$,
    \[
        \iota_P(u)(\gamma_j)
        =
        \gamma_{\theta(j)}+\lambda_jn.
    \]
    The positions in $P$ contribute zero to the total displacement $\sum_{i=1}^n
        \bigl(\iota_P(u)(i)-i\bigr)$.
    Therefore
    \begin{align*}
        \operatorname{av}(\iota_P(u))=\frac1n\sum_{i=1}^n
        \bigl(\iota_P(u)(i)-i\bigr)
        =
        \frac1n
        \sum_{j=1}^N
        \bigl(
            \gamma_{\theta(j)}-\gamma_j+\lambda_jn
        \bigr)=
        \sum_{j=1}^N\lambda_j=
        \varepsilon K,
    \end{align*}
    where the last equality follows from
    \eqref{eq:oriented-type-lambda}.  Thus
    \[
        \iota_P(u)\in\widetilde S_n^{\varepsilon K}.
    \]

    Conversely, let
    \[
        w\in\widetilde S_n^{\varepsilon K}
    \]
    satisfy
    \[
        0\le\varepsilon(w(i)-i)<n+\xi
        \qquad(i\in\ZZ)
    \]
    and suppose that
    \[
        \FIX{w\mid_{[n]}}=P.
    \]
    Since $w$ fixes $P^{\mathrm{per}}$ pointwise and is bijective,
    it maps $\ZZ\setminus P^{\mathrm{per}}$ bijectively onto itself.
    Hence
    \[
        u
        \coloneqq
        \rho_P\circ
        w\bigm|_{\ZZ\setminus P^{\mathrm{per}}}
        \circ\rho_P^{-1}
    \]
    is well defined and bijective.

    If $j\in\ZZ$ and $i=\rho_P^{-1}(j)$, then
    \[
        \rho_P^{-1}(j+N)=i+n.
    \]
    Using the periodicity of $w$ and
    \eqref{eq:oriented-compression-periodicity}, we obtain
    \begin{align*}
        u(j+N)
        =
        \rho_P\bigl(w(i+n)\bigr)=
        \rho_P\bigl(w(i)+n\bigr)=
        \rho_P\bigl(w(i)\bigr)+N=
        u(j)+N.
    \end{align*}
    Therefore $u$ is an affine permutation of period $N$.

    Since $i\notin P^{\mathrm{per}}$, we have $w(i)\ne i$, and hence
    \[
        0<\varepsilon(w(i)-i)<n+\xi.
    \]
    Applying $\rho_P$ and using the same endpoint analysis as above,
    now in reverse, gives
    \[
        0<\varepsilon(u(j)-j)<N+\xi.
    \]
    The average-displacement calculation above, read in reverse, gives
    \[
        \operatorname{av}(u)=\frac1N\sum_{j=1}^N\bigl(u(j)-j\bigr)
        =
        \frac1n\sum_{i=1}^n\bigl(w(i)-i\bigr)
        =
        \varepsilon K.
    \]
    Thus $u\in\widetilde S_N^{\varepsilon K}$, and insertion $(P,u)\mapsto \iota_P(u)$ and
    deletion $w\mapsto \left(\FIX{w\mid_{[n]}},\rho_P\circ
        w\bigm|_{\ZZ\setminus P^{\mathrm{per}}}
        \circ\rho_P^{-1}\right)$ are mutually inverse.  We have proved the bijection \eqref{bij:oriented-fixed-point-insertion}.

    We next determine the possible range of $r$ such that both sides of this bijection are nonempty.  For every $i\in[N]$,
    the integer
    \[
        \varepsilon(u(i)-i)
    \]
    lies between $1$ and $N+\xi-1$.  Hence
    \[
        K
        =
        \frac1N
        \sum_{i=1}^N
        \varepsilon(u(i)-i)
        \le N+\xi-1.
    \]
    Since $N=n-r$, this is equivalent to
    \begin{equation*}
        r\le n-K+\xi-1.
    \end{equation*}
    Consequently, the
    strata indexed by
    \begin{align}
    \label{eq:oriented-fixed-point-range}
        \fix{w\mid_{[n]}}=r,\quad0\le r\le n-K+\xi-1
    \end{align}
    exhaust the right-hand side of
    \eqref{eq:affine-fixed-point-insertion}.

    We now compare affine lengths of $u$ and $\iota_P(u)$.  For an affine permutation $w$
    occurring in the right-hand side of
    \eqref{eq:affine-fixed-point-insertion}, define
    \[
        \operatorname{cr}^{\varepsilon}_p(w)
        \coloneqq
        \begin{cases}
            \left|
                \{j\in\ZZ:j\le p<w(j)\}
            \right|,
            &\varepsilon=1,\\[2mm]
            \left|
                \{j\in\ZZ:w(j)\le p<j\}
            \right|,
            &\varepsilon=-1,
        \end{cases}
        \qquad(p\in\ZZ).
    \]
    These cardinalities are finite because
    \[
        0\le\varepsilon(w(j)-j)<n+\xi.
    \]

    We claim that $\operatorname{cr}^{\varepsilon}_p(w)$ is
    independent of $p$.  When $\varepsilon=1$, moving the cut from
    $p-1$ to $p$ adds the interval beginning at $p$ precisely when
    $w(p)>p$.  In that case, the unique preimage of $p$ is strictly
    smaller than $p$, so exactly one interval ending at $p$ is
    removed.  If $w(p)=p$, no interval is added or removed.

    When $\varepsilon=-1$, moving the cut from $p-1$ to $p$ removes
    the interval ending at $p$ precisely when $w(p)<p$.  In that
    case, the unique preimage of $p$ is strictly larger than $p$, so
    exactly one interval beginning at $p$ is added.  If $w(p)=p$,
    again no interval is added or removed.  This proves the claim.

    Double-counting, modulo $(j,p)\sim(j+n,p+n)$, the
    ordered integer pairs $(j,p)$ such that $\begin{cases}
        j\le p < w(j), &\text{if $\varepsilon=1$} \\
        w(j) \le p < j, &\text{if $\varepsilon=-1$}
    \end{cases}$ gives
    \[
        \sum_{p=1}^n\operatorname{cr}^{\varepsilon}_p(w)
        =
        \sum_{j=1}^n\varepsilon(w(j)-j)
        =
        nK,
    \]
    where counting the representatives $(j,p)$ with $p\in[n]$ gives the left-hand summation, and counting the representatives $(j,p)$ with $j\in[n]$ gives the right-hand summation.
    Therefore
    \begin{equation}\label{eq:oriented-crossing-number}
        \operatorname{cr}^{\varepsilon}_p(w)=K
        \qquad(p\in\ZZ).
    \end{equation}

    Recall that for an affine permutation $v\in\widetilde{S}_n$, its affine length $\ell(v)$ counts the number of the equivalence classes of its inversions $(i,j)$ such that $i<j$ and $v(i)>v(j)$ modulo $(i,j)\sim(i+n,j+n)$. Since $\rho_P$ is order preserving and satisfies
    \eqref{eq:oriented-compression-periodicity}, it induces a
    bijection between the affine inversion classes of $u$ and those
    affine inversion classes of $\iota_P(u)$ whose two indices lie
    outside $P^{\mathrm{per}}$.

    Suppose first that $\varepsilon=1$.  An inversion involving an
    inserted fixed point cannot have that fixed point as its first
    entry.  For each $p\in P$, the inversion classes involving
    $p+n\ZZ$ therefore have unique representatives
    \[
        (j,p),
        \qquad
        j<p<\iota_P(u)(j),
    \]
    and their number is
    $\operatorname{cr}^{1}_p(\iota_P(u))=K$.

    Suppose instead that $\varepsilon=-1$.  An inversion involving an
    inserted fixed point cannot have that fixed point as its second
    entry.  For each $p\in P$, the inversion classes involving
    $p+n\ZZ$ therefore have unique representatives
    \[
        (p,j),
        \qquad
        \iota_P(u)(j)<p<j,
    \]
    and their number is
    $\operatorname{cr}^{-1}_p(\iota_P(u))=K$.

    There are no inversions between two inserted fixed points.
    Since $|P|=r$, in both cases we obtain
    \begin{equation}\label{eq:oriented-insertion-length}
        \ell(\iota_P(u))=\ell(u)+rK.
    \end{equation}

    It remains to compare descent sets of $u$ and $\iota_P(u)$.  Put
    \[
        \pi\coloneqq\std{u\mid_{[N]}}\in\symm_N,
        \qquad
        \eta_r\coloneqq(1,2,\ldots,r),
    \]
    where $\eta_0$ is the empty word.

    If $\varepsilon=1$, define
    \[
        \pi^{+r}
        \coloneqq
        \bigl(r+\pi(1),\ldots,r+\pi(N)\bigr),
    \]
    and let $\omega_P^{1}$ be the unique shuffle of
    $\eta_r$ and $\pi^{+r}$ in which the letters of $\eta_r$ occupy
    precisely the positions in $P$.

    If $\varepsilon=-1$, define
    \[
        \widehat{\eta}_r
        \coloneqq
        (N+1,N+2,\ldots,N+r),
    \]
    and let $\omega_P^{-1}$ be the unique shuffle of
    $\pi$ and $\widehat{\eta}_r$ in which the letters of
    $\widehat{\eta}_r$ occupy precisely the positions in $P$.

    We claim that
    \begin{equation}\label{eq:oriented-descent-comparison}
        \DES{\iota_P(u)\mid_{[n]}}
        =
        \DES{\omega_P^{\varepsilon}}.
    \end{equation}
    Fix $i\in[n-1]$.

    If neither $i$ nor $i+1$ lies in $P$, then
    \[
        \rho_P(i+1)=\rho_P(i)+1,
    \]
    and the order-preserving property of $\rho_P^{-1}$ shows that
    both words have the same comparison at $i$.

    If both $i$ and $i+1$ lie in $P$, then
    \[
        \iota_P(u)(i)=i<i+1=\iota_P(u)(i+1),
    \]
    and the corresponding letters of the inserted increasing word
    are also increasing.  Thus neither word has a descent at $i$.

    Suppose that $i\notin P$ and $i+1\in P$.  When
    $\varepsilon=1$, we have $\iota_P(u)(i)>i$.  Moreover,
    $\iota_P(u)(i)\notin P^{\mathrm{per}}$, so it cannot equal
    $i+1$.  Hence
    \[
        \iota_P(u)(i)>i+1=\iota_P(u)(i+1).
    \]
    Both words therefore have a descent at $i$, since in
    $\omega_P^1$ every nonfixed letter is larger than every inserted
    letter.  When $\varepsilon=-1$, we have
    \[
        \iota_P(u)(i)<i<i+1=\iota_P(u)(i+1),
    \]
    and neither word has a descent at $i$, since in
    $\omega_P^{-1}$ every nonfixed letter is smaller than every
    inserted letter.

    Finally, suppose that $i\in P$ and $i+1\notin P$.  When
    $\varepsilon=1$, we have
    \[
        \iota_P(u)(i)=i<i+1<\iota_P(u)(i+1),
    \]
    and neither word has a descent at $i$.  When
    $\varepsilon=-1$, we have $\iota_P(u)(i+1)<i+1$.
    Furthermore, $\iota_P(u)(i+1)\notin P^{\mathrm{per}}$, so it
    cannot equal $i$.  Hence
    \[
        \iota_P(u)(i)=i>\iota_P(u)(i+1).
    \]
    Both words have a descent at $i$, because every inserted letter
    in $\omega_P^{-1}$ is larger than every nonfixed letter.

    These four cases prove
    \eqref{eq:oriented-descent-comparison}.  Applying the shuffle
    product rule for fundamental quasisymmetric functions gives
    \begin{align}
        \sum_{P\in\binom{[n]}r}
        F_{\DES{\iota_P(u)\mid_{[n]}},n}
        &=
        h_rF_{\DES{u\mid_{[N]}},N}.
        \label{eq:oriented-insertion-shuffle}
    \end{align}

    Combining the bijection
    \eqref{bij:oriented-fixed-point-insertion}, the length formula
    \eqref{eq:oriented-insertion-length}, and the shuffle identity
    \eqref{eq:oriented-insertion-shuffle}, we obtain
    \begin{align}
        &\sum_{\substack{
            w\in\widetilde S_n^{\varepsilon K}\\
            \text{$0\le\varepsilon(w(i)-i)<n+\xi$ for all $i$}\\
            \fix{w\mid_{[n]}}=r
        }}
        q^{\ell(w)}F_{\DES{w\mid_{[n]}},n}
        \notag\\
        &\qquad=
        q^{rK}h_r
        \sum_{\substack{
            u\in\widetilde S_{n-r}^{\varepsilon K}\\
            \text{$0<\varepsilon(u(i)-i)<n-r+\xi$ for all $i$}
        }}
        q^{\ell(u)}
        F_{\DES{u\mid_{[n-r]}},n-r}.
        \label{eq:oriented-fixed-point-stratum}
    \end{align}
    Summing \eqref{eq:oriented-fixed-point-stratum} over
    $0\le r\le n-K+\xi-1$ and using
    \eqref{eq:oriented-fixed-point-range} proves
    \eqref{eq:affine-fixed-point-insertion}.
\end{proof}

\begin{lemma}\label{lem:a+b-n=1-dual}
    For all $1<b<n$, we have
    \begin{align*}
        &\sum_{r=0}^{n-b}q^{r(b-1)}h_r
        \sum_{\substack{
            f\in\widetilde S_{n-r}^0\\
            \text{$i-n+b+r-1<f(i)<i+b-1$ for all $i$}
        }}
        q^{\ell(f)}
        F_{\DES{f\mid_{[n-r]}},n-r}\\
        &\qquad=
        \sum_{\substack{
            f\in\widetilde S_n^0\\
            \text{$i-n+b\le f(i)<i+b$ for all $i$}
        }}
        q^{\ell(f)}
        F_{\DES{f\mid_{[n]}},n}.
    \end{align*}
\end{lemma}

\begin{proof}
    Put
    \[
        K\coloneqq b-1.
    \]
    For every positive integer $N$, vertical translation by $-K$,
    \[
        f(i)\longmapsto f(i)-K,
    \]
    is a bijection from $\widetilde S_N^0$ to
    $\widetilde S_N^{-K}$.  It preserves affine length and descent set.

    For $0\le r\le n-b$, put $N=n-r$.  The bounds in the inner sum
    transform as
    \begin{align*}
        i-n+b+r-1<f(i)<i+b-1
        &\iff
        i-N<f(i)-K<i\\
        &\iff
        0<-\bigl((f(i)-K)-i\bigr)<N.
    \end{align*}
    On the right-hand side, integrality gives
    \begin{align*}
        i-n+b\le f(i)<i+b
        &\iff
        i-n+1\le f(i)-K\le i\\
        &\iff
        i-n<f(i)-K\le i\\
        &\iff
        0\le-\bigl((f(i)-K)-i\bigr)<n.
    \end{align*}
    Hence the asserted identity is equivalent to
    \begin{align*}
        &\sum_{r=0}^{n-b}q^{rK}h_r
        \sum_{\substack{
            u\in\widetilde S_{n-r}^{-K}\\
            \text{$0<-(u(i)-i)<n-r$ for all $i$}
        }}
        q^{\ell(u)}F_{\DES{u\mid_{[n-r]}},n-r}\\
        &\qquad=
        \sum_{\substack{
            w\in\widetilde S_n^{-K}\\
            \text{$0\le-(w(i)-i)<n$ for all $i$}
        }}
        q^{\ell(w)}F_{\DES{w\mid_{[n]}},n}.
    \end{align*}
    Since
    \[
        n-K-1=n-b,
    \]
    this is Lemma~\ref{lem:affine-fixed-point-insertion} with
    \[
        \varepsilon=-1,
        \qquad
        \xi=0.
    \]
\end{proof}

\begin{lemma}\label{lem:a+b-n=1-not-dual}
    For all $1<b<n$, we have
    \begin{align*}&\sum_{r=0}^{b-1}q^{r(n-b)} \cdot h_r \cdot \sum_{\substack{f\in\widetilde{S}_{n-r}^0 \\ \text{$i-n+b< f(i)<i+b-r$ for all $i$}}} q^{\ell(f)} \cdot F_{\DES{f\mid_{[n-r]}},n-r} \\
    &= \sum_{\substack{f\in\widetilde{S}_{n}^0 \\ \text{$i-n+b\le f(i)<i+b$ for all $i$}}} q^{\ell(f)} \cdot F_{\DES{f\mid_{[n]}},n} .\end{align*}
\end{lemma}
\begin{proof}
    Put
    \[
        K\coloneqq n-b.
    \]
    For every positive integer $N$, vertical translation by $K$,
    \[
        f(i)\longmapsto f(i)+K,
    \]
    is a bijection from $\widetilde S_N^0$ to
    $\widetilde S_N^K$.  It preserves affine length and descent set.

    For $0\le r\le b-1$, the bounds in the inner sum transform as
    \[
        i-n+b<f(i)<i+b-r
        \iff
        i<f(i)+K<i+n-r.
    \]
    The bounds on the right-hand side transform as
    \[
        i-n+b\le f(i)<i+b
        \iff
        i\le f(i)+K<i+n.
    \]
    Therefore the asserted identity is equivalent to
    \begin{align*}
        &\sum_{r=0}^{b-1}q^{rK}h_r
        \sum_{\substack{
            u\in\widetilde S_{n-r}^{K}\\
            \text{$0<u(i)-i<n-r$ for all $i$}
        }}
        q^{\ell(u)}F_{\DES{u\mid_{[n-r]}},n-r}\\
        &\qquad=
        \sum_{\substack{
            w\in\widetilde S_n^{K}\\
            \text{$0\le w(i)-i<n$ for all $i$}
        }}
        q^{\ell(w)}F_{\DES{w\mid_{[n]}},n}.
    \end{align*}
    Since
    \[
        n-K-1=b-1,
    \]
    this is precisely Lemma~\ref{lem:affine-fixed-point-insertion}
    with
    \[
        \varepsilon=1,
        \qquad
        \xi=0.
    \]
\end{proof}

In order to analyze the stabilization limit $G_{n,\infty,b}$, we need the following technical result.
\begin{lemma}\label{lem:dual-recursive}
    For $n,a,b>0$ with $1<b<n$ and $a+b>n$, we have
    \begin{align}\label{eq:dual-recursive}G_{n,a,b} = \sum_{m=0}^{n-b+1} q^{m(b-1)} \cdot h_m \cdot G_{n-m,a-m,b-1}.\end{align}
\end{lemma}
\begin{proof}
    We prove Equation~\eqref{eq:dual-recursive} by induction on $a+b-n$. For the base case $a+b-n=1$, from Lemma~\ref{lem:a+b-n=1-not-dual} we deduce that
    \begin{align}\label{eq:base-case-not-dual}
        &G_{n,a,b} = G_{n,n-b+1,b} = \\
        \nonumber& \sum_{r=0}^{b-1}h_r[q^{n-b} \cdot \xxx] \cdot \sum_{\substack{f\in\widetilde{S}_{n-r}^0 \\ \text{$i-n+b<f(i)<i+b-r$ for all $i$}}} q^{\ell(f)} \cdot F_{\DES{f\mid_{[n-r]}},n-r} = \\
        \nonumber& \sum_{r=0}^{b-1} q^{r(n-b)} \cdot h_r \cdot \sum_{\substack{f\in\widetilde{S}_{n-r}^0 \\ \text{$i-n+b<f(i)<i+b-r$ for all $i$}}} q^{\ell(f)} \cdot F_{\DES{f\mid_{[n-r]}},n-r} = \\
        \nonumber& \sum_{\substack{f\in\widetilde{S}_{n}^0 \\ \text{$i-n+b\le f(i)<i+b$ for all $i$}}} q^{\ell(f)} \cdot F_{\DES{f\mid_{[n]}},n} .
    \end{align}
    Moreover, Lemma~\ref{lem:a+b-n=1-dual} implies that
    \begin{align}\label{eq:base-case-dual}
        & \sum_{m=0}^{n-b+1} q^{m(b-1)} \cdot h_m \cdot G_{n-m,a-m,b-1} = \sum_{m=0}^{n-b+1} q^{m(b-1)} \cdot h_m \cdot G_{n-m,n-b-m+1,b-1} = \\
        \nonumber &\sum_{m=0}^{n-b} q^{m(b-1)} \cdot h_m \cdot G_{n-m,n-b-m+1,b-1} = \sum_{\substack{f\in\widetilde{S}_{n}^0 \\ \text{$i-n+b\le f(i)<i+b$ for all $i$}}} q^{\ell(f)} \cdot F_{\DES{f\mid_{[n]}},n} .
    \end{align}
    As Equations~\eqref{eq:base-case-not-dual} and \eqref{eq:base-case-dual} share the same right-hand-side, the left-hand-sides of both must be equal to each other as well. It follows that Equation~\eqref{eq:dual-recursive} holds when $a+b-n=1$.

    Suppose, for the sake of induction, that Equation~\eqref{eq:dual-recursive} holds when $a+b-n=c$ ($c\ge 1$). Then we deal with the case where $a+b-n=c+1$. First applying Lemma~\ref{lem:not-dual-recurrsive} and then applying the induction assumption, we obtain
    \begin{align}\label{eq:undual-dual}
        &G_{n,a,b} = \sum_{m=0}^{b-1} q^{m(a-1)} \cdot h_m \cdot G_{n-m,a-1,b-m} \qquad\text{(Lemma~\ref{lem:not-dual-recurrsive})}\\
        \nonumber &= \sum_{m=0}^{b-2} q^{m(a-1)} \cdot h_m \cdot G_{n-m,a-1,b-m} + q^{(b-1)(a-1)} \cdot h_{b-1} \cdot G_{n-b+1, a-1, 1} \\
        \nonumber &= \sum_{m=0}^{b-2} q^{m(a-1)} \cdot h_m \cdot G_{n-m,a-1,b-m} + q^{(b-1)(a-1)} \cdot h_{b-1} \cdot h_{n-b+1} \qquad\text{(Lemma~\ref{lem:trivial})}\\
        \nonumber &= \sum_{m=0}^{b-2} q^{m(a-1)} \cdot h_m \cdot \sum_{r=0}^{n-b+1} q^{r(b-m-1)} \cdot h_r \cdot G_{n-m-r,a-1-r,b-m-1} + q^{(b-1)(a-1)} \cdot h_{b-1} \cdot h_{n-b+1}\\ \nonumber &\text{(The equality above follows from induction hypotheses)} \\
        \nonumber &= \sum_{m=0}^{b-2}\sum_{r=0}^{n-b+1} q^{m(a-1)+r(b-m-1)} \cdot h_m \cdot h_r \cdot G_{n-m-r,a-1-r,b-m-1} + q^{(b-1)(a-1)} \cdot h_{b-1} \cdot h_{n-b+1} \\
        \nonumber &= \sum_{m=0}^{b-2}\sum_{r=0}^{n-b+1} q^{m(a-1)+r(b-1)-rm} \cdot h_m \cdot h_r \cdot G_{n-m-r,a-1-r,b-m-1} + q^{(b-1)(a-1)} \cdot h_{b-1} \cdot h_{n-b+1} \\
        \nonumber &= \sum_{m=0}^{b-2}\sum_{r=0}^{n-b} q^{m(a-1)+r(b-1)-rm} \cdot h_m \cdot h_r \cdot G_{n-m-r,a-1-r,b-m-1} \\ \nonumber&+ \sum_{m=0}^{b-2} q^{m(a-1)+(n-b+1)(b-1)-(n-b+1)m} \cdot h_m \cdot h_{n-b+1} \cdot G_{n-m-(n-b+1),a-1-(n-b+1),b-m-1} \\ \nonumber&+ q^{(b-1)(a-1)} \cdot h_{b-1} \cdot h_{n-b+1} \\
        \nonumber&= \sum_{m=0}^{b-2}\sum_{r=0}^{n-b} q^{m(a-1)+r(b-1)-rm} \cdot h_m \cdot h_r \cdot G_{n-m-r,a-1-r,b-m-1} \\ \nonumber&+ \sum_{m=0}^{b-2} q^{m(a+b-n-2)+(n-b+1)(b-1)} \cdot h_m \cdot h_{n-b+1} \cdot G_{b-m-1,a+b-n-2,b-m-1} \\ \nonumber&+ q^{(b-1)(a-1)} \cdot h_{b-1} \cdot h_{n-b+1} \\
        \nonumber&= \sum_{m=0}^{b-2}\sum_{r=0}^{n-b} q^{m(a-1)+r(b-1)-rm} \cdot h_m \cdot h_r \cdot G_{n-m-r,a-1-r,b-m-1} \\ \nonumber&+ \sum_{m=0}^{b-2} q^{m(c-1)+(n-b+1)(b-1)} \cdot h_m \cdot h_{n-b+1} \cdot G_{b-m-1,c-1,b-m-1} \\ \nonumber&+ q^{(b-1)(a-1)} \cdot h_{b-1} \cdot h_{n-b+1} \\
        \nonumber&= \sum_{m=0}^{b-2}\sum_{r=0}^{n-b} q^{m(a-1)+r(b-1)-rm} \cdot h_m \cdot h_r \cdot G_{n-m-r,a-1-r,b-m-1} \\ \nonumber&+ \sum_{m=0}^{b-2} q^{m(c-1)+(n-b+1)(b-1)} \cdot h_m \cdot h_{n-b+1} \cdot h_{b-m-1}\Bigg[\Bigg(\sum_{k=0}^{c-2}q^k\Bigg) \cdot \xxx\Bigg] + q^{(b-1)(a-1)} \cdot h_{b-1} \cdot h_{n-b+1} \\
        \nonumber&= \sum_{m=0}^{b-2}\sum_{r=0}^{n-b} q^{m(a-1)+r(b-1)-rm} \cdot h_m \cdot h_r \cdot G_{n-m-r,a-1-r,b-m-1} \\ \nonumber&+ q^{(n-b+1)(b-1)} \cdot h_{n-b+1} \cdot \sum_{m=0}^{b-1} q^{m(c-1)} \cdot h_m \cdot h_{b-m-1}\Bigg[\Bigg(\sum_{k=0}^{c-2}q^k\Bigg) \cdot \xxx\Bigg] \\
        \nonumber&= \sum_{m=0}^{b-2}\sum_{r=0}^{n-b} q^{m(a-1)+r(b-1)-rm} \cdot h_m \cdot h_r \cdot G_{n-m-r,a-1-r,b-m-1} \\ \nonumber&+ q^{(n-b+1)(b-1)} \cdot h_{n-b+1} \cdot \sum_{m=0}^{b-1} h_m[q^{c-1} \cdot \xxx] \cdot h_{b-m-1}\Bigg[\Bigg(\sum_{k=0}^{c-2}q^k\Bigg) \cdot \xxx\Bigg] \\
        \nonumber&= \sum_{m=0}^{b-2}\sum_{r=0}^{n-b} q^{m(a-1)+r(b-1)-rm} \cdot h_m \cdot h_r \cdot G_{n-m-r,a-1-r,b-m-1} \\ \nonumber&+ q^{(n-b+1)(b-1)} \cdot h_{n-b+1} \cdot h_{b-1}\Bigg[\Bigg(\sum_{k=0}^{c-1}q^k\Bigg) \cdot \xxx\Bigg]
    \end{align}
    where the first identity follows from Lemma~\ref{lem:not-dual-recurrsive}, the third identity follows from Lemma~\ref{lem:trivial}, and the fourth identity arises from the induction assumption.
    By Lemma~\ref{lem:not-dual-recurrsive}, we additionally have
    \begin{align}\label{eq:dual-undual}
        &\sum_{m=0}^{n-b+1} q^{m(b-1)} \cdot h_m \cdot G_{n-m,a-m,b-1} \\
        \nonumber &= \sum_{m=0}^{n-b} q^{m(b-1)} \cdot h_m \cdot G_{n-m,a-m,b-1} + q^{(n-b+1)(b-1)} \cdot h_{n-b+1} \cdot G_{b-1,a-n+b-1,b-1} \\
        \nonumber&= \sum_{m=0}^{n-b} q^{m(b-1)} \cdot h_m \cdot \sum_{r=0}^{b-2} q^{r(a-m-1)} \cdot h_r \cdot G_{n-m-r,a-m-1,b-1-r} + q^{(n-b+1)(b-1)} \cdot h_{n-b+1} \cdot G_{b-1,c,b-1} \\
        \nonumber&= \sum_{m=0}^{n-b} \sum_{r=0}^{b-2} q^{m(b-1)+r(a-m-1)} \cdot h_m \cdot h_r \cdot G_{n-m-r,a-m-1,b-1-r} + q^{(n-b+1)(b-1)} \cdot h_{n-b+1} \cdot h_{b-1}\Bigg[\Bigg(\sum_{k=0}^{c-1} q^k \Bigg) \cdot \xxx \Bigg] \\
        \nonumber&= \sum_{r=0}^{n-b} \sum_{m=0}^{b-2} q^{r(b-1)+m(a-r-1)} \cdot h_r \cdot h_m \cdot G_{n-r-m,a-r-1,b-1-m} + q^{(n-b+1)(b-1)} \cdot h_{n-b+1} \cdot h_{b-1}\Bigg[\Bigg(\sum_{k=0}^{c-1} q^k \Bigg) \cdot \xxx \Bigg] \\
        \nonumber&= \sum_{m=0}^{b-2} \sum_{r=0}^{n-b} q^{r(b-1)+m(a-1)-rm} \cdot h_r \cdot h_m \cdot G_{n-r-m,a-r-1,b-1-m} + q^{(n-b+1)(b-1)} \cdot h_{n-b+1} \cdot h_{b-1}\Bigg[\Bigg(\sum_{k=0}^{c-1} q^k \Bigg) \cdot \xxx \Bigg].
    \end{align}
    Since Equations~\eqref{eq:undual-dual} and \eqref{eq:dual-undual} have the same right-hand-side, their left-hand-sides must be equal to each other as well, implying that Equation~\eqref{eq:dual-recursive} still holds when $a+b-n=c+1$, which finishes the induction.
\end{proof}

We now prove a closed formula for $\F_{n,a,b}$.

\begin{theorem}\label{thm:closed-formula-n,a,b}
    For all $n,a,b>0$ such that $b<n$ and $a+b\ge n$, we have
    \[\F_{n,a,b} = \sum_{r=0}^{b-1}h_r\Bigg[\Bigg(\sum_{k=n-b}^{a-1}q^k\Bigg)\cdot \xxx\Bigg] \cdot \sum_{\substack{f\in\widetilde{S}_{n-r}^0 \\ \text{$i-n+b<f(i)<i+b-r$ for all $i$}}} q^{\ell(f)} \cdot F_{\DES{f\mid_{[n-r]}},n-r}.\]
\end{theorem}
\begin{proof}
    In order to apply Lemma~\ref{lem:unique-recursive} to $G_{n,a,b}$, we need to verify all the conditions in Lemma~\ref{lem:unique-recursive}.
    By Definition~\ref{def:G}, we have that $\lim_{a\rightarrow\infty}G_{n,a,b} = G_{n,\infty,b}$ exists (in the sense of Definition~\ref{def:stable-limit}). According to Lemma~\ref{lem:dual-recursive}, we let $a\rightarrow\infty$ (in the sense of Definition~\ref{def:stable-limit}) and thus obtain
    \[G_{n,\infty,b} = \sum_{m=0}^{n-b+1} q^{m(b-1)} \cdot h_m \cdot G_{n-m,\infty,b-1}.\]
    Additionally, Lemma~\ref{lem:trivial} implies $G_{n,\infty,1} = h_n$ for $n>1$, and the conventions Definition~\ref{def:G} imply $G_{n,\infty,n} = h_n\big[\frac{\xxx}{1-q}\big]$ for $n>0$. All these results and Lemma~\ref{lem:not-dual-recurrsive} mean that all the conditions in Lemma~\ref{lem:unique-recursive} holds for $G_{n,a,b}$, so Lemma~\ref{lem:unique-recursive} yields $\F_{n,a,b} = G_{n,a,b}$ for all $n,a,b>0$ such that $b<n$ and $a+b\ge n$, completing our proof.
\end{proof}
\begin{remark}\label{rmk:closed-formula-n,a,b}
    Theorem~\ref{thm:closed-formula-n,a,b} implies that the its right-hand-side, a quasisymmetric function, is actually a symmetric function, which is not combinatorially obvious.
\end{remark}

In particular, for $0\le a+b-n \le 2$ we have a concise fundamental quasisymmetric function expansion formula for $\F_{n,a,b}$.
\begin{corollary}[{Theorem~\ref{thm:main:a+b-n=0,1,2}}]\label{cor:a+b-n=0,1,2}
    Let $n,a,b>0$ satisfy $0\le a+b-n\le 2$ and $b<n$. Then
    \[\F_{n,a,b} = \sum_{\substack{f\in\widetilde{S}_n^0 \\ \text{$i-a<f(i)<i+b$ for all $i$}}} q^{\ell(f)} \cdot F_{\DES{f\mid_{[n]}},n}.\]
\end{corollary}

\begin{proof}
    By Theorem~\ref{thm:closed-formula-n,a,b}, it is enough to prove
    the asserted formula with $\F_{n,a,b}$ replaced by $G_{n,a,b}$.

    We first treat the case $b=1$.  Since $b<n$, we have $n>1$, and
    Lemma~\ref{lem:trivial} gives
    \[
        G_{n,a,1}=h_n.
    \]
    If $f\in\widetilde S_n^0$ satisfies
    \[
        i-a<f(i)<i+1
        \qquad(i\in\ZZ),
    \]
    then integrality gives $f(i)\le i$ for every $i$.  Since
    \[
        \frac1n\sum_{i=1}^n\bigl(f(i)-i\bigr)=\operatorname{av}(f)=0,
    \]
    every displacement $f(i)-i$ must be zero.  Hence $f$ is the identity
    affine permutation.  Conversely, the identity satisfies the
    displayed bounds because $a>0$.  Therefore the sum in the
    statement is
    \[
        F_{\varnothing,n}=h_n=G_{n,a,1}.
    \]
    We may thus assume that $1<b<n$.

    Suppose first that $a+b-n=0$.  Then $a=n-b$, so the alphabet
    \[
        \sum_{k=n-b}^{a-1}q^k
    \]
    in Definition~\ref{def:G} is empty.  Since $h_0[0]=1$ and
    $h_r[0]=0$ for $r>0$, only the term $r=0$ remains.  Thus
    \[
        G_{n,a,b}
        =
        \sum_{\substack{
            f\in\widetilde S_n^0\\
            \text{$i-a<f(i)<i+b$ for all $i$}
        }}
        q^{\ell(f)}F_{\DES{f\mid_{[n]}},n}.
    \]

    Suppose next that $a+b-n=1$.  Since $a=n-b+1$,
    Lemma~\ref{lem:a+b-n=1-not-dual} gives
    \[
        G_{n,a,b}
        =
        \sum_{\substack{
            f\in\widetilde S_n^0\\
            \text{$i-n+b\le f(i)<i+b$ for all $i$}
        }}
        q^{\ell(f)}F_{\DES{f\mid_{[n]}},n}.
    \]
    Because $f(i)$ is integral,
    \[
        i-n+b\le f(i)
        \iff
        i-n+b-1<f(i)
        \iff
        i-a<f(i).
    \]
    This proves the assertion in this case.

    It remains to consider $a+b-n=2$.  Put
    \[
        \kappa\coloneqq n-b,
        \qquad
        K\coloneqq\kappa+1=a-1.
    \]
    Lemma~\ref{lem:not-dual-recurrsive} gives
    \begin{equation}\label{eq:cor-small-excess-recursion-short}
        G_{n,a,b}
        =
        \sum_{r=0}^{b-1}
        q^{rK}h_rG_{n-r,a-1,b-r}.
    \end{equation}
    For every $0\le r\le b-1$,
    \[
        (a-1)+(b-r)-(n-r)=1.
    \]
    If $b-r>1$, the case $a+b-n=1$ proved above applies to
    $G_{n-r,a-1,b-r}$.  If $b-r=1$, the case $b=1$ treated at the
    beginning of the proof applies instead.  Consequently,
    \begin{equation}\label{eq:cor-small-excess-inner-short}
        G_{n-r,a-1,b-r}
        =
        \sum_{\substack{
            v\in\widetilde S_{n-r}^0\\
            \text{$i-K<v(i)<i+b-r$ for all $i$}
        }}
        q^{\ell(v)}
        F_{\DES{v\mid_{[n-r]}},n-r}
    \end{equation}
    for every $0\le r\le b-1$.

    Substitute \eqref{eq:cor-small-excess-inner-short} into
    \eqref{eq:cor-small-excess-recursion-short}, and vertically
    translate each $v$ by $K$:
    \[
        u(i)\coloneqq v(i)+K.
    \]
    This translation preserves affine length and descent set and maps
    $\widetilde S_{n-r}^0$ bijectively to
    $\widetilde S_{n-r}^K$.  Since $K+b=n+1$, the bounds transform as
    \begin{align*}
        i-K<v(i)<i+b-r
        &\iff
        i<u(i)<i+n-r+1.
    \end{align*}
    Therefore
    \begin{align}
        G_{n,a,b}
        ={}&
        \sum_{r=0}^{b-1}q^{rK}h_r
        \sum_{\substack{
            u\in\widetilde S_{n-r}^K\\
            \text{$0<u(i)-i<n-r+1$ for all $i$}
        }}
        q^{\ell(u)}
        F_{\DES{u\mid_{[n-r]}},n-r}.
        \label{eq:cor-small-excess-translated-short}
    \end{align}
    Since
    \[
        n-K=b-1,
    \]
    Lemma~\ref{lem:affine-fixed-point-insertion}, with
    \[
        \varepsilon=1,
        \qquad
        \xi=1,
    \]
    transforms \eqref{eq:cor-small-excess-translated-short} into
    \[
        G_{n,a,b}
        =
        \sum_{\substack{
            w\in\widetilde S_n^K\\
            \text{$0\le w(i)-i<n+1$ for all $i$}
        }}
        q^{\ell(w)}
        F_{\DES{w\mid_{[n]}},n}.
    \]

    Finally, define
    \[
        f(i)\coloneqq w(i)-K.
    \]
    This vertical translation is a bijection from the last indexing
    set $\widetilde{S}_n^K$ to $\widetilde S_n^0$ and preserves affine length and descent
    set.  Since
    \[
        a=K+1
        \qquad\text{and}\qquad
        b=n-K+1,
    \]
    integrality gives
    \begin{align*}
        0\le w(i)-i<n+1
        &\iff
        i-K\le f(i)<i+n-K+1\\
        &\iff
        i-a<f(i)<i+b.
    \end{align*}
    This proves the assertion when $a+b-n=2$ and completes the proof.
\end{proof}

Surprisingly, when $a+b-n=2$, the Hilbert series $\Hilb(R_{n,a,b};q)$ refines the number of positroid varieties $\Pi_f$ in $\grass{a-1}{n}$ according to codimension. (See Speyer~\cite{speyer2024richardsonvarietiesprojectedrichardson} for details of positroid varieties.)
\begin{corollary}[{Theorem~\ref{thm:main:positroid-counting}}]\label{cor:positroid}
    For $1<k\le n$, we have
    \[\Hilb(R_{n,k+1,n-k+1};q) = \sum_{f\in\bound{k}{n}} q^{\ell(f)} = \sum_{f\in\bound{k}{n}} q^{\codim{\grass{k}{n}}{\Pi_f}}.\]
\end{corollary}
\begin{proof}
    The case $a+b-n=2$ of Corollary~\ref{cor:a+b-n=0,1,2} implies
    \begin{align*}
        \grFrob(R_{n,k+1,n-k+1};q) &= \sum_{\substack{f\in\widetilde{S}_n^0 \\ \text{$i-(k+1)<f(i)<i+(n-k+1)$ for all $i$}}} q^{\ell(f)} \cdot F_{\DES{f\mid_{[n]}},n} \\
        &= \sum_{g\in\bound{k}{n}} q^{\ell(g)} \cdot F_{\DES{g\mid_{[n]}},n}
    \end{align*}
    where the second identity follows from the vertical translation $f\mapsto g$ with $g(i)=f(i)+k$ which is a bijection from $\widetilde{S}_n^0$ to $\widetilde{S}_n^k$. Therefore, we have
    \[\Hilb(R_{n,k+1,n-k+1};q) = \sum_{f\in\bound{k}{n}} q^{\ell(f)} = \sum_{f\in\bound{k}{n}} q^{\codim{\grass{k}{n}}{\Pi_f}},\]
    concluding our proof.
\end{proof}

\section{Connection with Eulerian quasisymmetric functions}\label{sec:eulerian-quasisym}

Interestingly, Corollary~\ref{cor:a+b-n=0,1,2} reveals the connection between the ungraded Frobenius images $\Frob(R_{n,a,b})$ and Eulerian quasisymmetric functions. We convert an ordinary permutation $w\in\symm_n$ to an affine permutation $\widetilde{w}\in\dschub_n^{-\operatorname{exc}(w)}$ as follows.

\begin{definition}\label{def:permutation-to-affine}
    Let $w\in\symm_n$.  Define an affine permutation
    $\widetilde{w}$ by setting
    \[
        \widetilde{w}(i)
        \coloneqq
        \begin{cases}
            w(i)-n,&w(i)>i,\\
            w(i),&w(i)\leq i,
        \end{cases}
        \qquad i\in[n],
    \]
    and then extending this definition to $\ZZ$ by
    \[
        \widetilde{w}(i+n)=\widetilde{w}(i)+n
        \qquad(i\in\ZZ).
    \]
    The residues of
    $\widetilde{w}(1),\ldots,\widetilde{w}(n)$ modulo $n$ are
    $w(1),\ldots,w(n)$, respectively.  Since these residues are
    pairwise distinct, $\widetilde{w}$ is indeed an affine
    permutation.  Furthermore,
    \begin{align*}
        \frac1n\sum_{i=1}^n
        \bigl(\widetilde{w}(i)-i\bigr)
        &=
        \frac1n\left(
            \sum_{i=1}^n\bigl(w(i)-i\bigr)
            -
            n\operatorname{exc}(w)
        \right)=-\operatorname{exc}(w).
    \end{align*}
\end{definition}

\begin{lemma}\label{lem:permutation-affine-bijection}
Let $n\geq1$ and $0\leq j\leq n$.
With $\widetilde w$ as in
Definition~\ref{def:permutation-to-affine}, the map
\[
    \beta_j(w)(i)\coloneqq\widetilde w(i)+j
    \qquad(i\in\ZZ)
\]
is a bijection
\[
    \{w\in\symm_n:\operatorname{exc}(w)=j\}
    \longrightarrow
    \left\{
        f\in\widetilde S_n^0:
        i-(n-j)<f(i)\leq i+j
        \text{ for all }i\in\ZZ
    \right\}.
\]
Moreover,
\[
    \DES{\beta_j(w)\mid_{[n]}}=\operatorname{DEX}(w),
    \qquad
    \beta_j(w)(i)=i+j\iff w(i)=i
    \quad(i\in[n]).
\]
In particular, this bijection restricts to
\[
    \{w\in\symm_n:\operatorname{exc}(w)=j,\operatorname{fix}(w)=0\}
    \longrightarrow
    \left\{
        f\in\widetilde S_n^0:
        i-(n-j)<f(i)< i+j
        \text{ for all }i\in\ZZ
    \right\}.
\]
\end{lemma}

\begin{proof}
By Definition~\ref{def:permutation-to-affine},
\[
    i-n<\widetilde w(i)\leq i,
    \qquad
    \widetilde w(i)=i\iff w(i)=i
    \quad(i\in[n]),
\]
and
\[
    \operatorname{av}(\beta_j(w))=\frac1n\sum_{i=1}^n\bigl(\beta_j(w)(i)-i\bigr)
    =j-\operatorname{exc}(w)=0.
\]
Thus $\beta_j(w)$ belongs to the stated target set.

Conversely, given $f$ in that set, let $w(i)\in[n]$ be the
representative of $f(i)-j$ modulo $n$.
These representatives form a permutation, and
\[
    i-n<f(i)-j\leq i
    \quad\Longrightarrow\quad
    f(i)-j=\widetilde w(i),
\]
because $(i-n,i]$ contains exactly one representative of each
residue class modulo $n$.
The condition $\operatorname{av}(f)=0$ then gives
\[
    0=\frac1n\sum_{i=1}^n(f(i)-i)=j+\frac1n\sum_{i=1}^n(\widetilde{w}(i)-i)
     =j-\operatorname{exc}(w),
\]
so this construction is the inverse of $\beta_j$.

Lowering each excedance entry by $n$ realizes the alphabet
order in the definition of $\operatorname{DEX}$, and translation
preserves descents. Hence
\[
    \DES{\beta_j(w)\mid_{[n]}}
    =\DES{\widetilde w\mid_{[n]}}
    =\operatorname{DEX}(w).
\]
The assertion about fixed points follows from
$\beta_j(w)(i)=i+j\iff\widetilde w(i)=i$ and periodicity.
\end{proof}

Corollary~\ref{cor:a+b-n=0,1,2} and
Lemma~\ref{lem:permutation-affine-bijection} connect the ungraded
Frobenius characteristics of $R_{n,a,b}$ with Eulerian
quasisymmetric functions.

\begin{theorem}[{Theorem~\ref{thm:main:eulerian-0}}]\label{thm:a+b-n=0:Eulerian}
Let $n$ be a positive integer. For every $0\leq k\leq n$,
\[
    \Frob(R_{n,k,n-k})=Q_{n,k,0}.
\]
\end{theorem}

\begin{proof}
For $1\leq k\leq n-1$, specialize
Corollary~\ref{cor:a+b-n=0,1,2} at $q=1$ and apply
Lemma~\ref{lem:permutation-affine-bijection} with $j=n-k$:
\begin{align*}
    \Frob(R_{n,k,n-k})
    &=\sum_{\substack{f\in\widetilde S_n^0\\
                      i-k<f(i)<i+n-k\text{ for all }i}}
        F_{\DES{f\mid_{[n]}},n}\\
    &=\sum_{\substack{w\in\symm_n\\
                      \operatorname{exc}(w)=n-k,\ \fix{w}=0}}
        F_{\operatorname{DEX}(w),n}\\
    &=Q_{n,n-k,0}=Q_{n,k,0}.
\end{align*}
The last equality is
\cite[Theorem~5.1(2)]{ShareshianWachs2010}.

For the endpoints $k=0$ or $k=n$, the generators
\[
    x_1^0=1\in I_{n,0,n},
    \qquad
    e_0([n]\setminus\{1\})=1\in I_{n,n,0}
\]
give $R_{n,0,n}=R_{n,n,0}=0$.
Also, $Q_{n,0,0}=Q_{n,n,0}=0$: the only permutation with no
excedances is the identity, and no permutation has $n$ excedances.
\end{proof}

\begin{theorem}[{Theorem~\ref{thm:main:eulerian-1}}]\label{thm:a+b-n=1:Eulerian}
Let $n$ be a positive integer. For every $0\leq k\leq n$,
\[
    \Frob(R_{n,k+1,n-k})=Q_{n,k}.
\]
\end{theorem}

\begin{proof}
For $1\leq k\leq n-1$, specialize
Corollary~\ref{cor:a+b-n=0,1,2} at $q=1$ and apply
Lemma~\ref{lem:permutation-affine-bijection} with $j=n-k-1$:
\begin{align*}
    \Frob(R_{n,k+1,n-k})
    &=\sum_{\substack{f\in\widetilde S_n^0\\
                      i-k-1<f(i)<i+n-k\text{ for all }i}}
        F_{\DES{f\mid_{[n]}},n}\\
    &=\sum_{\substack{w\in\symm_n\\
                      \operatorname{exc}(w)=n-k-1}}
        F_{\operatorname{DEX}(w),n}\\
    &=Q_{n,n-k-1}=Q_{n,k}.
\end{align*}
Here we use integrality to replace $f(i)<i+n-k$ by
$f(i)\leq i+n-k-1$; the last equality is
\cite[Theorem~5.1(3)]{ShareshianWachs2010}.

For $k=0$, the equality
$I_{n,1,n}=\langle x_1,\ldots,x_n\rangle$ gives
\[
    \Frob(R_{n,1,n})=h_n=Q_{n,0}.
\]
For $k=n$, the generator
$e_0([n]\setminus\{1\})=1$ gives
\[
    \Frob(R_{n,n+1,0})=0=Q_{n,n}.\qedhere
\]
\end{proof}

As the convolution relation \eqref{eq:tildeQ-Q} between $\widetilde{Q}_n$ and $Q_n$ aligns with the recursive formula of $\F_{n,a,b}$ given by Proposition~\ref{prop:recursive-n,a,b}, Theorem~\ref{thm:a+b-n=1:Eulerian} immediately yields the result below.

\begin{corollary}[{Theorem~\ref{thm:main:eulerian-2}}]\label{cor:Eulerian-2}
    For $1<k\le n$,
    \[\Frob(R_{n,k+1,n-k+1}) = \widetilde{Q}_{n,k}.\]
\end{corollary}
\begin{proof}
Since $k>1$, we have $n-k+1<n$.
Applying Proposition~\ref{prop:recursive-n,a,b} with
$(a,b)=(k+1,n-k+1)$ and specializing at $q=1$, we obtain
\begin{align*}
    \Frob(R_{n,k+1,n-k+1})
    &=
    \sum_{m=0}^{n-k}
        h_m\,\Frob(R_{n-m,k,n-k+1-m})\\
    &=
    \sum_{m=0}^{n-k}h_m Q_{n-m,k-1}\\
    &=
    \sum_{r=k}^{n}h_{n-r}Q_{r,k-1},
\end{align*}
where the second equality follows from
Theorem~\ref{thm:a+b-n=1:Eulerian}.

On the other hand, extracting the coefficient of $t^k$ in
\eqref{eq:tildeQ-Q} gives
\[
    \widetilde Q_{n,k}
    =
    \sum_{r=1}^{n}h_{n-r}Q_{r,k-1}
    =
    \sum_{r=k}^{n}h_{n-r}Q_{r,k-1},
\]
since a permutation in $\symm_r$ has at most $r-1$ excedances,
and hence $Q_{r,k-1}=0$ for $r<k$.
Comparing the two expressions proves the assertion.
\end{proof}

Forgetting the $\symm_n$-module structure, $R_{n,k+1,n-k}$ exhibits connections to Eulerian numbers.

\begin{corollary}\label{cor:eulerian-number}
    For integers $n>0$ and $0\le k\le n$, we have
    \[\dim_{\QQ} R_{n,k+1,n-k} = A_{n,k}.\]
\end{corollary}
\begin{proof}
    This result immediately follows from Theorem~\ref{thm:a+b-n=1:Eulerian}.
\end{proof}

Moreover, the array consisting of $\dim_\QQ R_{n,a,b}$ satisfies a beautiful generating function.

\begin{corollary}\label{cor:generating-function}
    Let $d\ge 0$ be a nonnegative integer. Then
    \begin{align}\label{eq:generating-funciton}
    \sum_{n\ge 0} \sum_{k=d}^{n+d} (\dim_\QQ R_{n,k,n+d-k}) \cdot \frac{t^n}{n!} \cdot q^k = \frac{1-q}{\exp(tq) - q \cdot \exp(t)} \cdot (q\cdot\exp(t))^d  
    \end{align}
    with the convention that $\dim_\QQ R_{0,m,0} = 1$ for $m\ge 0$.
\end{corollary}
\begin{proof}
    We prove \eqref{eq:generating-funciton} by induction on $d$. When $d=0$, Theorem~\ref{thm:a+b-n=0:Eulerian} yields that $\dim_\QQ R_{n,k,n-k}$ is the number of dearrangements in $\symm_n$ with exactly $k$ exceedances. Then \cite[Proposition 5]{MR993741} gives \eqref{eq:generating-funciton} when $d=0$.

    Suppose, for the sake of induction, that \eqref{eq:generating-funciton} holds when $d=d_0$ ($d_0\ge 0$). We want to show that \eqref{eq:generating-funciton} holds when $d=d_0+1$. First, we claim that for all $n\ge 0$ and $d_0+1\le k \le n+d_0+1$ we have
    \begin{align}
        \label{eq:dim-recursive}
        \dim_\QQ R_{n,k,n+d_0+1-k}=\sum_{m=0}^{n+d_0+1-k} \binom{n}{m} \cdot \dim_\QQ R_{n-m,k-1,n+d_0+1-k-m}.
    \end{align}
    
    Indeed, if $n>0$ and $d_0+1< k\le n+d_0$ then the recursive formula in Proposition~\ref{prop:recursive-n,a,b} implies that
    \begin{align*}
        \dim_\QQ R_{n,k,n+d_0+1-k} &= \sum_{m=0}^{n+d_0-k} \binom{n}{m} \cdot \dim_\QQ R_{n-m,k-1,n+d_0+1-k-m} \\
        \nonumber&=\sum_{m=0}^{n+d_0+1-k} \binom{n}{m} \cdot \dim_\QQ R_{n-m,k-1,n+d_0+1-k-m}.
    \end{align*}
    If $n>0,k=d_0+1$, then \eqref{eq:dim-recursive} reduces to
    \[\dim_\QQ R_{n,d_0+1,n} = \sum_{m=0}^n \binom{n}{m} \cdot \dim_\QQ R_{n-m,d_0,n-m}\]
    which is true because: The vector space isomorphism $R_{n,a,n}=\QQ[x_1,\dots,x_n]/\langle x_i^a:i\in[n]\rangle\cong(\QQ[x]/\langle x^a \rangle)^{\otimes n}$ implies that $\dim_\QQ R_{n,a,n} = a^n$ for $n,a\ge 0$ with the convention that $0^0=1$, so we have
    \begin{align*}
        \dim_\QQ R_{n,d_0+1,n} = (d_0+1)^n = \sum_{m=0}^n \binom{n}{m} \cdot d_0^{n-m} = \sum_{m=0}^n \binom{n}{m} \cdot \dim_\QQ R_{n-m,d_0,n-m}.
    \end{align*}
    If $n>0,k=n+d_0+1$, then \eqref{eq:dim-recursive} reduces to
    \[\dim_\QQ R_{n,n+d_0+1,0} = \binom{n}{0} \cdot \dim_\QQ R_{n,n+d_0,0}\]
    which is true because both sides are $0$. If $n=0$, then $d_0+1\le k \le n+d_0+1$ means $k=d_0+1$ and hence \eqref{eq:dim-recursive} reduces to
    \[\dim_\QQ R_{0,d_0+1,0} = \binom{0}{0} \cdot \dim_\QQ R_{0,d_0,0}\]
    which is true because both sides equal $1$ by convention. Therefore, \eqref{eq:dim-recursive} holds for all $n\ge 0$ and $d_0+1 \le k \le n+d_0+1$.

    We now begin to show \eqref{eq:generating-funciton} for $d=d_0+1$. Applying \eqref{eq:dim-recursive} and the induction assumption, we obtain
    \begin{align*}
        &\sum_{n\ge 0} \sum_{k=d_0+1}^{n+d_0+1} (\dim_\QQ R_{n,k,n+d_0+1-k}) \cdot \frac{t^n}{n!} \cdot q^k \\ =& \sum_{n\ge 0} \sum_{k=d_0+1}^{n+d_0+1} \Bigg(\sum_{m=0}^{n+d_0+1-k} \binom{n}{m} \cdot \dim_\QQ R_{n-m,k-1,n+d_0+1-k-m}\Bigg) \cdot \frac{t^n}{n!} \cdot q^k \\
        =& \sum_{n\ge 0} \sum_{k=d_0+1}^{n+d_0+1} \Bigg(\sum_{m=0}^{n+d_0+1-k} \frac{t^m}{m!}\cdot \frac{t^{n-m}}{(n-m)!} \cdot \dim_\QQ R_{n-m,k-1,n+d_0+1-k-m}\Bigg) \cdot q^k \\
        =& \sum_{n\ge 0} \sum_{m=0}^n \sum_{k=d_0+1}^{n+d_0+1-m} \frac{t^m}{m!}\cdot \frac{t^{n-m}}{(n-m)!} \cdot \dim_\QQ R_{n-m,k-1,n+d_0+1-k-m} \cdot q^k \\
        =& \sum_{m\ge 0} \sum_{n\ge m} \sum_{k=d_0+1}^{n+d_0+1-m} \frac{t^m}{m!}\cdot \frac{t^{n-m}}{(n-m)!} \cdot \dim_\QQ R_{n-m,k-1,n+d_0+1-k-m} \cdot q^k \\
        =& \sum_{m\ge 0}\frac{t^m}{m!}\cdot  \sum_{n\ge m} \sum_{k=d_0+1}^{n+d_0+1-m} \frac{t^{n-m}}{(n-m)!} \cdot \dim_\QQ R_{n-m,k-1,n+d_0+1-k-m} \cdot q^k \\
        =& \sum_{m\ge 0}\frac{t^m}{m!}\cdot  \sum_{n\ge 0} \sum_{k=d_0+1}^{n+d_0+1} \frac{t^{n}}{n!} \cdot \dim_\QQ R_{n,k-1,n+d_0+1-k} \cdot q^k \\
        =& \sum_{m\ge 0}\frac{t^m}{m!}\cdot  \sum_{n\ge 0} \sum_{k=d_0}^{n+d_0} \frac{t^{n}}{n!} \cdot \dim_\QQ R_{n,k,n+d_0-k} \cdot q^{k+1}
        = \sum_{m\ge 0}\frac{t^m}{m!}\cdot q \cdot  \sum_{n\ge 0} \sum_{k=d_0}^{n+d_0} \frac{t^{n}}{n!} \cdot \dim_\QQ R_{n,k,n+d_0-k} \cdot q^{k} \\ =& \sum_{m\ge 0}\frac{t^m}{m!}\cdot q \cdot \frac{1-q}{\exp(tq) - q \cdot \exp(t)} \cdot (q\cdot\exp(t))^{d_0} = \exp(t) \cdot q \cdot \frac{1-q}{\exp(tq) - q \cdot \exp(t)} \cdot (q\cdot\exp(t))^{d_0} \\
        =& \frac{1-q}{\exp(tq) - q \cdot \exp(t)} \cdot (q\cdot\exp(t))^{d_0+1},
    \end{align*}
    concluding the proof.
\end{proof}

\section{Connection with the Delta conjecture and the Theta conjecture}\label{sec:Theta-conjecture}

For the case $a+b-n=1$, we reveal the connection between $\F_{n,a,b}$ and the Theta conjecture.

\begin{theorem}[{Theorem~\ref{thm:main:theta}}]\label{thm:theta-conjecture}
    For integers \(a,b>0\), we have
    \[
        \F_{a+b-1,a,b}
        =
        \rev_q\Big(
        \omega\left(
        \left.
        \bigl(
        \Theta_{e_{a-1}}\Theta_{e_{b-1}}\nabla e_1
        \bigr)\right|_{t=0}
        \right)\Big).
    \]
\end{theorem}

\begin{proof}
Set $n=a+b-1$ and $r=b-1$.
For $a\geq2$, Corollary~\ref{cor:a+b-n=0,1,2} gives
\[\F_{n,a,b} = \sum_{\substack{f\in\widetilde{S}_n^0 \\ \text{$i-a<f(i)<i+b$ for all $i$}}} q^{\ell(f)} \cdot F_{\DES{f\mid_{[n]}},n}.\]
This also holds for $a=1$, since $\F_{n,1,n}=h_n$, and the
summing range consists only of the identity:
its elements satisfy
\[
    f(i)\geq i,
    \qquad
    \sum_{i=1}^n(f(i)-i)=0.
\]
Consequently, Lemma~\ref{lem:permutation-affine-bijection} gives
\begin{equation}\label{eq:theta-shared-bijection-sum}
    \F_{n,a,b}
    =\sum_{\substack{w\in\symm_n\\\operatorname{exc}(w)=r}}
        q^{\ell(\beta_r(w))}F_{\operatorname{DEX}(w),n}.
\end{equation}

For an ordinary permutation $w$ with $r$ excedances, we identify it with a bicolored permutation used in \cite{liao2026bicoloredeulerianquasisymmetricfunctions} without bars, so the affine
permutation $f_w$ used in
\cite[Section~3]{liao2026bicoloredeulerianquasisymmetricfunctions}
satisfies that
\[
    f_w(i)=w(i)-2n\varepsilon_i
          =\widetilde w(i)-n\varepsilon_i\qquad(\varepsilon_i\coloneqq\mathbf{1}_{\{w(i)>i\}}).
\]
Since $\widetilde w(i)\leq0$ precisely when $\varepsilon_i=1$,
for $1\leq s<t\leq n$ we have
\[\varepsilon_s-\varepsilon_t>0\Leftrightarrow\widetilde{w}(s)\le 0\text{ and }\widetilde{w}(t)>0\Rightarrow\widetilde{w}(t)-\widetilde{w}(s)>0\]
and
\[\varepsilon_s-\varepsilon_t<0\Leftrightarrow\widetilde{w}(s)> 0\text{ and }\widetilde{w}(t)\le 0\Rightarrow\widetilde{w}(t)-\widetilde{w}(s)<0.\]
Therefore,
\begin{align*}
    &\left|\left\lfloor\frac{f_w(t)-f_w(s)}n\right\rfloor\right|
    = \left|\left\lfloor\frac{\widetilde{w}(t)-\widetilde{w}(s)}n+(\varepsilon_s-\varepsilon_t)\right\rfloor\right|=\left|\left\lfloor\frac{\widetilde{w}(t)-\widetilde{w}(s)}n\right\rfloor+(\varepsilon_s-\varepsilon_t)\right| \\=& \left|\left\lfloor\frac{\widetilde{w}(t)-\widetilde{w}(s)}n\right\rfloor\right|+\left|\varepsilon_s-\varepsilon_t\right| =
    \left|\left\lfloor
        \frac{\beta_r(w)(t)-\beta_r(w)(s)}n
    \right\rfloor\right|
    +|\varepsilon_s-\varepsilon_t|.
\end{align*}
Thus the affine-length formula
\eqref{eq:bg-affine-window-length}
gives
\[
    \ell(f_w)
    =\ell(\beta_r(w))
       +\sum_{1\leq s<t\leq n}|\varepsilon_s-\varepsilon_t|
    =\ell(\beta_r(w))+r(n-r).
\]
Consequently, if we write $\Phi^{\mathrm{unb}}_{n,k,r}(\xxx;q)$ for the symmetric function given by \cite[Definitions~1.4]
{liao2026bicoloredeulerianquasisymmetricfunctions} then \eqref{eq:theta-shared-bijection-sum} becomes
\[
    \F_{n,a,b}
    =\sum_{\substack{w\in\symm_n\\\operatorname{exc}(w)=r}}
        q^{\ell(f_w)-r(n-r)}F_{\operatorname{DEX}(w),n}
    =\Phi^{\mathrm{unb}}_{n,0,r}(\xxx;q),
\]
where the last equality follows from
\cite[Definitions~1.1 and~1.4]
{liao2026bicoloredeulerianquasisymmetricfunctions};
when the bar parameter is zero, all entries are unbarred.

Applying
\cite[Theorem~1.5]{liao2026bicoloredeulerianquasisymmetricfunctions}
with bar parameter $k=0$, excedance parameter $r$, and $u=n-k-r=a$,
we obtain
\begin{align*}
    \F_{n,a,b}
    &=q^{(a-1)(b-1)}\omega\left(
        \left.
        \bigl(\Theta_{e_r}\Theta_{e_{n-r-1}}\widetilde H_{(1)}\bigr)
        \right|_{t=0,\,q=q^{-1}}
    \right) \\ &=q^{(a-1)(b-1)}\omega\left(
        \left.
        \bigl(\Theta_{e_{b-1}}\Theta_{e_{a-1}}\widetilde H_{(1)}\bigr)
        \right|_{t=0,\,q=q^{-1}}
    \right).
\end{align*}
Finally, $\widetilde H_{(1)}=\nabla e_1$, the Theta operators
commute on positive-degree symmetric functions, and the
$q$-degree recalled at the end of Section~\ref{subsec:theta} is
$(a-1)(b-1)$. Therefore
\[
    \F_{a+b-1,a,b}
    =\rev_q\left(\omega\left(
        \left.
        \bigl(\Theta_{e_{a-1}}\Theta_{e_{b-1}}\nabla e_1\bigr)
        \right|_{t=0}
    \right)\right).\qedhere
\]
\end{proof}

Theorem~\ref{thm:theta-conjecture} and the inversion of the recursive formula in Proposition~\ref{prop:recursive-n,a,b} yield the connection between the case $a+b-n=0$ and the Delta conjecture.

\begin{corollary}[{Theorem~\ref{thm:main:e_k-perp}}]\label{cor:delta}
    For $a,b\ge 1$, we have
    \[e_a^\perp\F_{a+b,a,b} = q^{\binom{a}{2}}\rev_q(\omega((\Delta'_{e_{a-1}}e_{b})\mid_{t=0})).\]
\end{corollary}
\begin{proof}
We first invert the convolution identity in Proposition~\ref{prop:recursive-n,a,b} using
$\sum_{i=0}^{r}(-1)^i e_i h_{r-i}=\delta_{r,0}$.
Applying Proposition~\ref{prop:recursive-n,a,b} to each
$\F_{a+b-m,a+1,b-m}$, we obtain
\[
\begin{aligned}
&\sum_{m=0}^{b-1}(-1)^m q^{am}e_m
    \F_{a+b-m,a+1,b-m}\\
&\quad=
    \sum_{m=0}^{b-1}\sum_{j=0}^{b-m-1}
    (-1)^m q^{a(m+j)}e_mh_j
    \F_{a+b-m-j,a,b-m-j}\\
&\quad=
    \sum_{r=0}^{b-1}\sum_{\substack{m,j\ge 0\\m+j=r}}
    (-1)^m q^{a(m+j)}e_mh_j
    \F_{a+b-m-j,a,b-m-j}\\
&\quad=
    \sum_{r=0}^{b-1}q^{ar}
    \left(\sum_{m=0}^{r}(-1)^m e_mh_{r-m}\right)
    \F_{a+b-r,a,b-r}\\
&\quad=\F_{a+b,a,b}.
\end{aligned}
\]
Set
\[
    \mathcal U_{a,b}
    =
    \sum_{m=0}^{b-1}(-1)^m h_m\Theta_{e_a}e_{b-m}.
\]
By the identity chain above, Theorem~\ref{thm:theta-conjecture},
\eqref{eq:theta-delta}, and the coefficient-reversal formulas
at the end of Section~\ref{subsec:theta}, 
\begin{align}
    &\nonumber\F_{a+b,a,b} = \sum_{m=0}^{b-1}(-1)^m q^{am}e_m
    \F_{a+b-m,a+1,b-m} \\ = &\notag\sum_{m=0}^{b-1}(-1)^m q^{am}e_m
    \rev_q\Big(
        \omega\left(
        \left.
        \bigl(
        \Theta_{e_{a}}\Theta_{e_{b-m-1}}\nabla e_1
        \bigr)\right|_{t=0}
        \right)\Big) \qquad\text{(Theorem~\ref{thm:theta-conjecture})}\\
    =&\notag \sum_{m=0}^{b-1}(-1)^m q^{am}e_m
    q^{a(b-m-1)}\Big(
        \omega\left(
        \left.
        \bigl(
        \Theta_{e_{a}}\Theta_{e_{b-m-1}}\nabla e_1
        \bigr)\right|_{q=q^{-1},\,t=0}
        \right)\Big) \\ 
    =&\notag q^{a(b-1)}\omega\left(\sum_{m=0}^{b-1}(-1)^m h_m
        \left.
        \bigl(
        \Theta_{e_{a}}\Theta_{e_{b-m-1}}\nabla e_1
        \bigr)\right|_{q=q^{-1},\,t=0}
        \right) \\
    =&\notag q^{a(b-1)}\omega\left(\sum_{m=0}^{b-1}(-1)^m h_m
        \left.
        \bigl(
        \Theta_{e_{a}}\Delta_{e_0}^\prime e_{b-m}
        \bigr)\right|_{q=q^{-1},\,t=0}
        \right) \qquad\text{(see \eqref{eq:theta-delta})}\\
    =&\notag q^{a(b-1)}\omega\left(\sum_{m=0}^{b-1}(-1)^m h_m
        \left.
        \bigl(
        \Theta_{e_{a}} e_{b-m}
        \bigr)\right|_{q=q^{-1},\,t=0}
        \right) \\
    =& 
    q^{a(b-1)}
    \omega\left(
        \left.\mathcal U_{a,b}\right|_{q=q^{-1},\,t=0}
    \right).
    \label{eq:delta-cor-inversion}
\end{align}

We next apply $h_a^\perp$ to $\mathcal U_{a,b}$.
For $0\leq j\leq a$ and $r\geq1$,
\eqref{eq:theta-h-perp} gives
\[
    h_{a-j}^\perp\Theta_{e_a}e_r
    =
    \Theta_{e_j}\Delta_{e_{a-j}}e_r
    +
    \Theta_{e_{j+1}}\Delta_{e_{a-j-1}}e_{r-1}.
\]
Using \eqref{eq:h-perp-product} and \eqref{eq:theta-h-perp},
we obtain
\[
\begin{aligned}
h_a^\perp\mathcal U_{a,b}
&=
\sum_{m=0}^{b-1}\sum_{j=0}^{a}
(-1)^m h_{m-j}
h_{a-j}^\perp\Theta_{e_a}e_{b-m}\\
&=
\sum_{m=0}^{b-1}\sum_{j=0}^{a}
(-1)^m h_{m-j}
\Theta_{e_j}\Delta_{e_{a-j}}e_{b-m}\\
&\quad+
\sum_{m=0}^{b-1}\sum_{j=0}^{a-1}
(-1)^m h_{m-j}
\Theta_{e_{j+1}}\Delta_{e_{a-j-1}}e_{b-m-1}\\
&=
\sum_{m=0}^{b-1}\sum_{j=0}^{a}
(-1)^m h_{m-j}
\Theta_{e_j}\Delta_{e_{a-j}}e_{b-m}\\
&\quad-
\sum_{m=1}^{b}\sum_{j=1}^{a}
(-1)^m h_{m-j}
\Theta_{e_j}\Delta_{e_{a-j}}e_{b-m}.
\end{aligned}
\]
In the last equality, we replaced $(m,j)$ by $(m-1,j-1)$
in the second double sum.
Since
\[
h_{-j}=0,
\qquad
\Theta_{e_j}\Delta_{e_{a-j}}1=0
\qquad(1\leq j\leq a),
\]
we may replace the range $1\leq m\leq b$ in that sum
by $0\leq m\leq b-1$. Hence
\[
\begin{aligned}
h_a^\perp\mathcal U_{a,b}
&=
\sum_{m=0}^{b-1}\sum_{j=0}^{a}
(-1)^m h_{m-j}
\Theta_{e_j}\Delta_{e_{a-j}}e_{b-m}\\
&\quad-
\sum_{m=0}^{b-1}\sum_{j=1}^{a}
(-1)^m h_{m-j}
\Theta_{e_j}\Delta_{e_{a-j}}e_{b-m}\\
&=
\sum_{m=0}^{b-1}(-1)^m h_m
\Delta_{e_a}e_{b-m}.
\end{aligned}
\]
Consequently, Lemma~\ref{lem:delta-t-zero-convolution} yields
\[
    \left.(h_a^\perp\mathcal U_{a,b})\right|_{t=0}
    =
    q^{b-a}
    \left.(\Delta'_{e_{a-1}}e_b)\right|_{t=0}.
\]
Using $e_a^\perp\omega=\omega h_a^\perp$ in
\eqref{eq:delta-cor-inversion} and the identity above, we conclude that
\[
    e_a^\perp\F_{a+b,a,b}
    =
    q^{(a-1)b}
    \omega\left(
        \left.(\Delta'_{e_{a-1}}e_b)\right|_{q=q^{-1},\,t=0}
    \right).
\]
For $a\leq b$, the coefficient-reversal formula at the end of
Section~\ref{subsec:theta} identifies this right-hand side with
\[
    q^{\binom{a}{2}}
    \rev_q\left(
        \omega\left(
            \left.(\Delta'_{e_{a-1}}e_b)\right|_{t=0}
        \right)
    \right).
\]
For $a>b$, the Delta specialization vanishes, so this identification still holds.
\end{proof}

Corollary~\ref{cor:delta} has an algebraic version. Specifically, Corollary~\ref{cor:delta} gives $\grFrob(\alt_a(R_{a+b,a,b});q)=q^{\binom{a}{2}}\grFrob(R_{b,a};q)$ with $1\le a\le b$, where $R_{n,k}$ is the generalized coinvariant algebra \cite{HaglundRhoadesShimozono2018} reviewed in Subsection~\ref{subsec:theta}. This identity admits the following algebraic interpretation.
\begin{corollary}\label{cor:alg-of-e-perp}
    For $1\le a\le b$, multiplication by $\prod_{b+1\le i<j\le a+b}(x_i-x_j)$ gives an $\symm_b$-module isomorphism $R_{b,a}\cong\alt_a(R_{a+b,a,b})$.
\end{corollary}
\begin{remark}\label{rmk:alg-of-e-perp}
    This map is homogeneous of degree \(\binom a2\), and becomes a graded isomorphism after shifting the grading of \(R_{b,a}\) upward by \(\binom a2\).
\end{remark}
\begin{proof}
    Note that for any $S\subseteq[n]$ with $n>0$, $d\ge 0$, and $i,j\notin S$, we have
    \begin{equation}\label{eq:diff-e}
        e_d(S)\cdot (x_i-x_j) = e_{d+1}(S\cup\{i\})-e_{d+1}(S\cup\{j\}).
    \end{equation}
    We first use this observation to prove that for $S\subsetneq[n]$ with $n>0$ and $d\ge 0$ we have
    \begin{equation}\label{eq:well-def}
        e_d(S)\cdot\prod_{\substack{i<j\\i,j\in[n]\setminus S}}(x_i-x_j) \in \left\langle e_{d+n-\lvert S\rvert-1}([n]\setminus\{i\})\,:\,i\in[n] \right\rangle.
    \end{equation}
    We prove it by induction on $n-\lvert S\rvert$. For $n-\lvert S\rvert =1$, this reduces to $e_d(S)\in\left\langle e_{d}([n]\setminus\{i\})\,:\,i\in[n] \right\rangle$ which is obvious. Now assume that $n-\lvert S\rvert>1$ and that \eqref{eq:well-def} holds for all smaller values of $n-\lvert S\rvert$ for the sake of induction. Pick $i_0<j_0$ from $[n]\setminus S$ and write $S_1=S\cup\{i_0\},\,S_2=S\cup\{j_0\}$. Then
    \begin{align*}
        &e_d(S)\cdot\prod_{\substack{i<j\\i,j\in[n]\setminus S}}(x_i-x_j) =e_d(S)\cdot (x_{i_0}-x_{j_0})\cdot\prod_{\substack{\{i<j\}\subseteq[n]\setminus S\\\{i,j\}\neq\{i_0,j_0\}}}(x_i-x_j) \\
        =& (e_{d+1}(S_1)-e_{d+1}(S_2))\cdot\prod_{\substack{\{i<j\}\subseteq[n]\setminus S\\\{i,j\}\neq\{i_0,j_0\}}}(x_i-x_j) \qquad\text{(by \eqref{eq:diff-e})} \\
        =& e_{d+1}(S_1)\cdot\prod_{\substack{\{i<j\}\subseteq[n]\setminus S\\\{i,j\}\neq\{i_0,j_0\}}}(x_i-x_j) - e_{d+1}(S_2)\cdot\prod_{\substack{\{i<j\}\subseteq[n]\setminus S\\\{i,j\}\neq\{i_0,j_0\}}}(x_i-x_j) \\
        =& e_{d+1}(S_1)\cdot\left(\prod_{\substack{i<j\\i,j\in[n]\setminus S_1}}(x_i-x_j)\right) \cdot g_1 - e_{d+1}(S_2)\cdot\left(\prod_{\substack{i<j\\i,j\in[n]\setminus S_2}}(x_i-x_j)\right)\cdot g_2
    \end{align*}
    where $g_1,g_2\in\QQ[\xxx_n]$. As the induction hypotheses implies that $e_{d+1}(S_k)\cdot\left(\prod_{\substack{i<j\\i,j\in[n]\setminus S_k}}(x_i-x_j)\right)\in\left\langle e_{d+1+n-\lvert S_k\rvert-1}([n]\setminus\{i\})\,:\,i\in[n] \right\rangle=\left\langle e_{d+n-\lvert S\rvert-1}([n]\setminus\{i\})\,:\,i\in[n] \right\rangle$ ($k=1,2$), we obtain \[e_d(S)\cdot\prod_{\substack{i<j\\i,j\in[n]\setminus S}}(x_i-x_j)\in\left\langle e_{d+n-\lvert S\rvert-1}([n]\setminus\{i\})\,:\,i\in[n] \right\rangle,\]
    concluding the proof of \eqref{eq:well-def}.

    Now we use \eqref{eq:well-def} to show the well-definedness of the multiplication map $R_{b,a}\rightarrow R_{a+b,a,b}$ by $\prod_{b+1\le i<j\le a+b}(x_i-x_j)$. Let
    \[I_{b,a}=\left\langle x_1^a,\dots,x_b^a, e_{b-a+1}([b]),\dots,e_b([b])\right\rangle\]
    be the defining ideal of $R_{b,a}$. It suffices to show that \begin{align}\label{eq:real-well-def}
        I_{b,a}\cdot\prod_{b+1\le i<j\le a+b}(x_i-x_j) \subseteq I_{a+b,a,b}.
    \end{align}
    Regarding the two types of generators of $I_{b,a}$. The first type of generators $x_i^a$ are contained in $I_{a+b,a,b}$, and for the second type of generators $e_d([b])$ ($b-a+1\le d\le b$) we have
    \[e_d([b])\cdot\prod_{b+1\le i<j\le a+b}(x_i-x_j)\in\left\langle e_{d+a-1}([a+b]\setminus\{i\})\,:\,i\in[a+b] \right \rangle\subseteq I_{a+b,a,b}\]
    by \eqref{eq:well-def} and Lemma~\ref{lem:higher-e}. We have proved \eqref{eq:real-well-def} and hence the well-definedness of the multiplication map $R_{b,a}\rightarrow R_{a+b,a,b}$.

    We claim that the image of this multiplication map equals $\alt_a(R_{a+b,a,b})$. In fact, modulo the relations \(x_{b+1}^a=\cdots=x_{a+b}^a=0\), a nonzero alternant in the last \(a\) variables must use the distinct exponents \(0,\ldots,a-1\). Hence every alternating class is represented by the Vandermonde polynomial multiplied by a polynomial in the first \(b\) variables.
    
    We thus obtain an $\symm_b$-equivariant surjection
    \[R_{b,a}\twoheadrightarrow \alt_a(R_{a+b,a,b}).\]
    Furthermore, as Corollary~\ref{cor:delta} and \eqref{eq:generalized-coinvariant-delta}, after setting $q=1$, imply
    \[\Frob(R_{b,a})=\Frob(\alt_a(R_{a+b,a,b})),\]
    the surjection above must be an $\symm_b$-module isomorphism.
\end{proof}

\section{Chow ring interpretation}\label{sec:chow}
Fix $n,k>0$ be integers such that $k<n$. Let $X \coloneqq (\mathbb{P}^{k-1})^n$ be the product of complex projective spaces, and let $U_{n,k}$ be the open subscheme of $X$ given by
\[U_{n,k}\coloneqq\left\{ (\ell_1,\dots,\ell_n)\in X\,:\, \sum_{\substack{i=1\\i\neq j}}^n \ell_i =\ \CC^{k}\,\,\text{for all}\,\, j\in[n] \right\},\]
where we regard all the components $\ell_j$ as lines in a common vector space $\CC^k$ passing through the origin.

Pawlowski and Rhoades~\cite{MR4029695} defined a similar variety
\[X_{m,j}\coloneqq\{(\ell_1,\dots,\ell_m)\in (\mathbb{P}^{j-1})^m\,:\,\sum_{i=1}^m \ell_i = \CC^j \}\]
for $0<j\le m$. They computed its integral cohomology ring via cellular decomposition.
\begin{theorem}[{\cite[Theorem 5.12(3)]{MR4029695}}]\label{thm:span-line-config}
    For $0<j\le m$, we have the ring isomorphism
    \[H^*(X_{m,j}) \cong \ZZ[x_1,\dots,x_m]/\langle x_1^j,\dots,x_m^j,e_{m-j+1}([m]),e_{m-j+2}([m]),\dots,e_m([m])\rangle\]
    which maps $c_1(\mathcal{L}_i)$ to the residue class of $x_i$ for all $i\in[m]$, where $\mathcal{L}_i$ is the pullback of the tautological line bundle over the $i$-th projective space factor.
\end{theorem}
We will use Theorem~\ref{thm:span-line-config} to prove the geometric interpretation Theorem~\ref{thm:main:chow} of our $R_{n,k,n-k}$. We restate and prove it.
\begin{theorem}\label{thm:chow}
    We have the ring isomorphism
    \[CH^*_\QQ (U_{n,k}) \cong R_{n,k,n-k},\]
    which maps $c_1(\mathcal{L}_i)$ to the residue class of $x_i$ for all $i\in[n]$, where $\mathcal{L}_i$ is the pullback of the tautological line bundle over the $i$-th projective space factor.
\end{theorem}
\begin{proof}
    Let $Z_i$ be the subvariety (we may take the reduced induced scheme structure, but it does not matter because we only need the Chow group $CH^\QQ_*(Z_i)$) of $X$ given by
    \[Z_i\coloneqq\left\{(\ell_1,\dots,\ell_n)\in X\,:\,\sum_{\substack{j=1\\ j\neq i}}^n \ell_j \subsetneq \CC^k\right\}.\]
    Since $U_{n,k} = X \setminus \bigcup_{i=1}^n Z_i$, the localization sequence yields the exact sequence of abelian groups
    \begin{align}\label{exact-seq-chow}
        \bigoplus_{i=1}^n CH^\QQ_{n(k-1)-*}(Z_i) \rightarrow CH_\QQ^*(X) \rightarrow CH_\QQ^*(U_{n,k}) \rightarrow 0,
    \end{align}
    where the first map is the pushforward of the closed embeddings $Z_i\hookrightarrow X$ and the second map is the pullback of the open embedding $U_{n,k}\subset X$.

    Let $W_{m,k}\coloneqq (\PP^{k-1})^m \setminus X_{m,k}$ ($m\ge k$). Then $Z_i = W_{n-1,k}\times\PP^{k-1}$ where the $\PP^{k-1}$ is the $i$-th projective space factor of $X$ and $W_{n-1,k}$ is contained in the product of the $1,\dots,\hat{i},\dots,n$-th projective space factors of $X$. Thus Theorem~\ref{thm:span-line-config} and the localization sequence
    \[CH^\QQ_{(n-1)(k-1)-*}(W_{n-1,k}) \rightarrow CH_\QQ^*((\PP^{k-1})^{n-1}) \rightarrow CH_\QQ^*(X_{n-1,k}) \rightarrow 0\]
    imply that the image of the pushforward
    \[CH^\QQ_{(n-1)(k-1)-*}(W_{n-1,k}) \rightarrow CH_\QQ^*((\PP^{k-1})^{n-1})=\QQ[x_1,\dots,\widehat{x_i},\dots,x_n]/\langle x_1^k,\dots,\widehat{x_i^k},\dots,x_n^k\rangle\]
    is the ideal
    \[\begin{aligned}\Big\langle \overline{e_{n-k}([n]\setminus\{i\})},\overline{e_{n-k+1}([n]\setminus\{i\})},\dots,\overline{e_{n-1}([n]\setminus\{i\})}\Big\rangle\subseteq\QQ[x_1,\dots,\widehat{x_i},\dots,x_n]/\langle x_1^k,\dots,\widehat{x_i^k},\dots,x_n^k\rangle\end{aligned}.\]
    Here we use the fact that both $(\PP^{k-1})^{n-1}$ and $X_{n-1,k}$ admit cellular decompositions (see \cite{MR4029695}), so their cycle class maps identify their rational Chow rings with their even cohomology rings. It follows that the image of the pushforward
    \[CH^\QQ_{n(k-1)-*}(Z_i)\rightarrow CH_\QQ^*(X)=\QQ[\xxx_n]/\langle x_1^k,\dots,x_n^k\rangle\] is the ideal
    \begin{align}\label{image-ideal-chow}
        \Big\langle \overline{e_{n-k}([n]\setminus\{i\})},\overline{e_{n-k+1}([n]\setminus\{i\})},\dots,\overline{e_{n-1}([n]\setminus\{i\})} \Big\rangle \subseteq \QQ[\xxx_n]/\langle x_1^k,\dots,x_n^k\rangle,
    \end{align}
    where we use the fact that the exterior product
    \[\bigoplus_{r+s=d}CH^\QQ_r(W_{n-1,k})\otimes CH^\QQ_s(\PP^{k-1}) \overset{\times}{\longrightarrow} CH^\QQ_d(W_{n-1,k}\times \PP^{k-1})=CH^\QQ_d(Z_i)\]
    is surjective (because $\PP^{k-1}$ has a cellular decomposition) and the fact that for proper morphisms $f,g$ we have $(f\times g)_*(\alpha\times\beta) =(f_*\alpha) \times (g_*\beta)$.

    Combining \eqref{exact-seq-chow} and \eqref{image-ideal-chow}, we deduce that the ring $CH_\QQ^*(U_{n,k})$ is isomorphic to the quotient of $\QQ[\xxx_n]$ by the ideal
    \[\langle x_i^k:i\in[n]\rangle + \sum_{i=1}^n \langle e_{n-k}([n]\setminus\{i\}),e_{n-k+1}([n]\setminus\{i\}),\dots,e_{n-1}([n]\setminus\{i\})\rangle.\]
    It remains to show that
    \begin{align*} \langle x_i^k:i\in[n]\rangle + \sum_{i=1}^n\langle &e_{n-k}([n]\setminus\{i\}),e_{n-k+1}([n]\setminus\{i\}),\dots,e_{n-1}([n]\setminus\{i\})\rangle \\ &= \langle x_i^k,e_{n-k}([n]\setminus\{i\})\,:\,i\in[n]\rangle.\end{align*}
    Indeed, it is clear that the right-hand side is contained in the left-hand side. Also, Lemma~\ref{lem:higher-e} implies that the left-hand side is contained in the right-hand side. Consequently, this equality holds, finishing our proof.
\end{proof}

We now relate the Hilbert series of $CH^*_\QQ(U_{n,k})$
to the positroid stratification of $U_{n,k}$.
Recall that $\bound{k}{n}$ is the set of
$(k,n)$-bounded affine permutations. Let
\[
\bound{k}{n}^{\circ}\coloneqq
\left\{
f\in\bound{k}{n}:
i<f(i)<i+n\text{ for all }i\in[n]
\right\}.
\]
For $f\in\bound{k}{n}$, let $\Pi_f^\circ$
denote the corresponding open positroid stratum
of $\grass{k}{n}$, whose closure is $\Pi_f$.

For each $f\in\bound{k}{n}^{\circ}$, define
\[
U_f\coloneqq
\left\{
(\ell_1,\dots,\ell_n)\in U_{n,k}:
\operatorname{rowspan}[v_1\ \cdots\ v_n]
\in\Pi_f^\circ
\right\},
\]
where $v_i$ is any nonzero vector (written by a column vector) spanning
$\ell_i$.
This definition is independent of the choices
of $v_i$, since rescaling the columns preserves
all the rank conditions defining the positroid
strata.

\begin{proposition}\label{prop:chow-positroid-stratification}
    The space $U_{n,k}$ admits the stratification
    \[
    U_{n,k}=
    \bigsqcup_{f\in\bound{k}{n}^{\circ}}U_f,
    \]
    where every $U_f$ is nonempty and
    \[
    \codim{U_{n,k}}{U_f}
    =
    \codim{\grass{k}{n}}{\Pi_f}
    =
    \ell(f).
    \]
\end{proposition}

\begin{proof}
    A configuration $(\ell_1,\dots,\ell_n)$
    can be represented by a $k\times n$ matrix
    $A=[v_1\ \cdots\ v_n]$, where each $v_i$
    is nonzero and spans $\ell_i$.
    The defining condition of $U_{n,k}$
    is equivalent to requiring that
    deleting any column of $A$ leaves a
    matrix of rank $k$.
    Thus the column matroid of $A$ has
    neither loops nor coloops.

    For each $k$-plane $V$,
    a loop of $\mathcal{B}_V$ at $i$ corresponds to $f_V(i)=i$,
    while a coloop of $\mathcal{B}_V$ at $i$ corresponds to
    $f_V(i)=i+n$; see Subsection~\ref{subsec:positroids-affine-perm}.
    Consequently, the positroid strata
    parametrizing configurations in $U_{n,k}$
    are precisely those indexed by
    $\bound{k}{n}^{\circ}$.
    This proves the asserted stratification
    and the nonemptiness of its strata.

    It remains to compute their codimensions.
    Let $\mathcal{M}_{n,k}$ be the open
    subvariety of $\operatorname{Mat}_{k\times n}(\CC)$
    consisting of matrices whose columns are
    nonzero and whose column matroids have
    no coloops, i.e. any $n-1$ columns form a matrix of rank $k$.
    There are two natural morphisms
    \[
    \begin{aligned}
    \pi:\mathcal{M}_{n,k}&\longrightarrow U_{n,k},
    &
    [v_1\ \cdots\ v_n]
    &\longmapsto
    (\CC v_1,\dots,\CC v_n),\\
    \rho:\mathcal{M}_{n,k}&\longrightarrow
    \grass{k}{n},
    &
    A&\longmapsto\operatorname{rowspan}(A),
    \end{aligned}
    \]
    where $\pi$ is surjective and the image of $\rho$ is the union
    of the open positroid strata indexed
    by $\bound{k}{n}^{\circ}$.

    For $f\in\bound{k}{n}^{\circ}$, write
    \[
    \mathcal{M}_f
    \coloneqq\rho^{-1}(\Pi_f^\circ)
    =\pi^{-1}(U_f).
    \]
    The fibers of
    $\pi:\mathcal{M}_f\to U_f$
    are isomorphic to $(\CC^\times)^n$,
    since two matrices represent the same
    line configuration $(\ell_1,\dots,\ell_n)$ precisely when
    their corresponding columns differ
    by nonzero scalar multiples.
    On the other hand, the fibers of
    $\rho:\mathcal{M}_f\to\Pi_f^\circ$
    are isomorphic to $\gl_k(\CC)$,
    since two full-rank matrices have the
    same row span precisely when they
    differ by an invertible row operation.
    Both maps are locally trivial bundles.
    It follows that
    \[
    \dim\mathcal{M}_f
    =\dim U_f+n
    =\dim\Pi_f^\circ+k^2.
    \]
    Hence
    \[
    \dim U_f
    =\dim\Pi_f^\circ+k^2-n.
    \]
    Since
    \[
    \dim U_{n,k}=n(k-1)
    \quad\text{and}\quad
    \dim\grass{k}{n}=k(n-k),
    \]
    we obtain
    \[
    \begin{aligned}
    \codim{U_{n,k}}{U_f}
    &=n(k-1)-\dim U_f\\
    &=k(n-k)-\dim\Pi_f^\circ\\
    &=\codim{\grass{k}{n}}{\Pi_f}\\
    &=\ell(f),
    \end{aligned}
    \]
    where the last equality follows from
    Subsection~\ref{subsec:positroids-affine-perm}.
\end{proof}

The preceding proposition, together with
Theorems~\ref{thm:chow}
and~\ref{thm:main:a+b-n=0,1,2},
gives the following interpretation of
the Hilbert series of $R_{n,k,n-k}$.

\begin{corollary}[{Corollary~\ref{cor:main:chow-positroid}}]\label{cor:chow-positroid-hilbert}
    For $1\le k<n$, we have
    \[
    \begin{aligned}
    \Hilb(CH^*_\QQ(U_{n,k});q)
    =\Hilb(R_{n,k,n-k};q)=\sum_{f\in\bound{k}{n}^{\circ}}
      q^{\ell(f)}=\sum_{f\in\bound{k}{n}^{\circ}}
      q^{\codim{U_{n,k}}{U_f}}.
    \end{aligned}
    \]
    In particular, for each $d\ge 0$,
    \[
    \dim_\QQ CH^d_\QQ(U_{n,k})
    =
    \#\left\{
    f\in\bound{k}{n}^{\circ}:
    \ell(f)=d
    \right\}.
    \]
\end{corollary}
\begin{remark}\label{rmk}
    Interestingly, this means that $\dim_\QQ CH^d_\QQ(U_{n,k})$ agrees with the number of strata $U_f$ with codimension $d$, while the author cannot find a geometric proof. The simplest way--showing the linear independence of fundamental classes $\left[\overline{U_f}\right]$--fails because they are usually linearly dependent.
\end{remark}
\begin{proof}
    The first equality follows immediately
    from Theorem~\ref{thm:chow}.
    Specializing
    Theorem~\ref{thm:main:a+b-n=0,1,2}
    to $(a,b)=(k,n-k)$ gives
    \[
    \F_{n,k,n-k}
    =
    \sum_{\substack{
    f\in\widetilde{S}_n^0\\
    i-k<f(i)<i+n-k\text{ for all }i
    }}
    q^{\ell(f)}
    F_{\DES{f\mid_{[n]}},n}.
    \]
    It follows that
    \[
    \Hilb(R_{n,k,n-k};q)
    =
    \sum_{\substack{
    f\in\widetilde{S}_n^0\\
    i-k<f(i)<i+n-k\text{ for all }i
    }}
    q^{\ell(f)}.
    \]
    The translation $f\mapsto f+k$
    preserves affine length and bijectively
    maps the summation range onto
    $\bound{k}{n}^{\circ}$.
    This proves the second equality.
    The last equality follows from
    Proposition~\ref{prop:chow-positroid-stratification}.
\end{proof}

\section{Future directions}\label{sec:future}
This paper gives plenty of representation theoretical results using one family of algebras $R_{n,a,b}$ to connect three seemingly different worlds: the positroids (affine permutations), the Delta and Theta conjectures, and the Eulerian quasisymmetric functions. Interestingly, all these staff has geometric meaning: positroids arise from positive geometry, the Delta conjecture also has a geometric model \cite{MR4029695}, and Eulerian quasisymmetric functions occur in the Frobenius characteristic of the cohomology of toric varieties \cite{ShareshianWachs2010,ShareshianWachs2020}. We thus ask for geometric interpretations to our results. The following problem is an example.
\begin{problem}\label{prob:geom-proof}
    Find a geometric proof of Corollary~\ref{cor:chow-positroid-hilbert} through studying the variety $U_{n,k}$.
\end{problem}

\section{Acknowledgments}\label{sec:acknowledgments}
The author thanks Alessandro Iraci, Martina Lanini, and Brendon Rhoades for their initial observations that motivated this work, as well as for helpful discussions and suggestions during the development of the project. Special thanks are due to Brendon Rhoades for bringing the problem to the author's attention, and for many helpful discussions and suggestions throughout this work. The author also thanks Hsin-Chieh Liao for helpful discussions about Eulerian quasisymmetric functions.

\section{Declaration of generative AI and AI-assisted technologies in the manuscript preparation process}

During the preparation of this work, the author used GPT-5.5 to prove Lemmas~\ref{lem:affine-fixed-point-insertion} and \ref{lem:dual-recursive}, and GPT-6.1 sol to assist with LaTeX formatting and language polishing. The original problem concerned the rings $R_{n,k+1,n-k+1}$. To address this problem through recursion, the author independently introduced the more general family $R_{n,a,b}$ and independently formulated and proved the recursive formulas in Propositions~\ref{prop:recursive-n,a,b} and \ref{prop:recursive-n,infty,b}. The author also developed the key idea of reducing $R_{n,a,b}$ to $R_{n,\infty,b}$. These contributions form the core of this work and were made without AI assistance. The author reviewed and edited the AI-generated output as needed and takes full responsibility for the content of the published article.

\printbibliography

\end{document}